\documentclass[aop,MSNbibl,dvips]{arximspdf}
\usepackage{mathrsfs}
\usepackage[mathcal]{euscript}

\doi{10.1214/13-AOP879} 
\volume{42}
\issue{6}
\pubyear{2014}
\firstpage{2243}
\lastpage{2313}

\makeatletter
\newtheorem{thmm}{Theorem}[section]
\newtheorem{cor}[thmm]{Corollary}
\newtheorem{lem}[thmm]{Lemma}
\newtheorem{prop}[thmm]{Proposition}

\newproclaim{defn}[thmm]{Definition}
\newproclaim{aspt}[thmm]{Assumption}
\newproclaim{rem}[thmm]{Remark}
\newproclaim{aspta}{Assumption A}
\newproclaim{saspt}{Standing Assumption}
\newproclaim{example}[thmm]{Example}

\makeatother

\begin{document}
\begin{frontmatter}

\title{Conditional ergodicity in infinite dimension\thanksref{T1}}
\runtitle{Conditional ergodicity in infinite dimension}
\thankstext{T1}{Supported in part by NSF Grants DMS-10-05575 and
CAREER-DMS-1148711.}

\begin{aug}
\author[A]{\fnms{Xin Thomson} \snm{Tong}\ead[label=e2]{tong@cims.nyu.edu}}
\and
\author[B]{\fnms{Ramon} \snm{van Handel}\corref{}\ead[label=e3]{rvan@princeton.edu}}
\runauthor{X. T. Tong and R. van Handel}
\affiliation{Courant Institute of Mathematical Sciences\\ and Princeton University}
\address[A]{Department of Mathematics\\
    Courant Institute\\
    \quad of Mathematical Sciences\\
    New York, New York 10012\\
    USA\\
    \printead{e2}}
\address[B]{Sherrerd Hall Room 227\\
Princeton University\\
Sherrerd Hall\\
Princeton, New Jersey 08544\\
USA\\
\printead{e3}}
\end{aug}

\received{\smonth{8} \syear{2012}}
\revised{\smonth{6} \syear{2013}}

%
\begin{abstract}
The goal of this paper is to develop a general method to establish
conditional ergodicity of infinite-dimensional Markov chains. Given a
Markov chain in a product space, we aim to understand the ergodic
properties of its conditional distributions given one of the components.
Such questions play a fundamental role in the ergodic theory of nonlinear
filters. In the setting of Harris chains, conditional ergodicity has been
established under \mbox{general} nondegeneracy assumptions. Unfortunately,
Markov chains in infinite-dimensional state spaces are rarely amenable to
the classical theory of Harris chains due to the singularity of their
transition probabilities, while topological and functional methods that
have been developed in the ergodic theory of infinite-dimensional Markov
chains are not well suited to the investigation of conditional
distributions. We must therefore develop new measure-theoretic tools in
the ergodic theory of Markov chains that enable the investigation of
conditional ergodicity for infinite dimensional or weak-* ergodic
processes. To this end, we first develop local counterparts of zero--two
laws that arise in the theory of Harris chains. These results give rise
to ergodic theorems for Markov chains that admit asymptotic couplings or
that are locally mixing in the sense of H.~F\"ollmer, and to a
non-Markovian ergodic theorem for stationary absolutely regular sequences.
We proceed to show that local ergodicity is inherited by conditioning
on a
nondegenerate observation process. This is used to prove stability and
unique ergodicity of the nonlinear filter. Finally, we show that our
abstract results can be applied to infinite-dimensional Markov processes
that arise in several settings, including dissipative stochastic partial
differential equations, stochastic spin systems and stochastic
differential delay equations.
\end{abstract}

%
\end{frontmatter}

\tableofcontents[alignleft,level=2]

\section{Introduction}
\label{sec:intro}

The classical ergodic theory of Markov chains in general state spaces
has achieved a rather definitive form in the theory of Harris chains
\cite{Rev84,Num84,MT09}, which provides necessary and sufficient
conditions for the convergence of the transition probabilities in total
variation to an invariant measure. While this theory is formulated in
principle for any measurable state space, it is well known that its
applicability extends in practice mainly to finite-dimensional
situations. In infinite dimension, the transition probabilities from
different initial conditions tend to be mutually singular even in the
most trivial examples, so that total variation convergence is out of the
question. For this reason, many infinite-dimensional Markov processes,
including stochastic partial differential equations, interacting
particle systems and stochastic equations with memory, lie outside the
scope of the classical theory. Instead, a variety of different
approaches, including topological \cite{DZ96,HMS11,KPS10}, functional
\cite{Mar04,HS89} coupling and duality \cite{Lig05} methods, have been
employed to investigate the ergodicity of infinite-dimensional models.

The goal of this paper is to investigate questions of \emph{conditional}
ergodicity in infinite dimension. Consider a Markov chain
$(X_n,Y_n)_{n\ge0}$ taking values in a product space $E\times F$
(continuous time processes are considered analogously). The aim of
conditional ergodic theory is to understand the ergodic properties of
one component of the process $(X_n)_{n\ge0}$ under the conditional
distribution given the other component $(Y_n)_{n\ge0}$. Even when the
process $(X_n,Y_n)_{n\ge0}$ is ergodic, the inheritance of ergodicity
under conditioning is far from obvious and does not always hold. The
history of such problems dates back to an erroneous result of Kunita~\cite{Kun71}, where the inheritance of ergodicity was taken for granted
(see \cite{vH12} and the references therein). The long-standing problem
of establishing conditional ergodicity under general assumptions was
largely resolved in \cite{vH09,TvH12}, where it is shown that the
inheritance of ergodicity holds under a mild nondegeneracy assumption
when $(X_n,Y_n)_{n\ge0}$ is a Harris chain. Numerous other results in
this area, both of a qualitative and quantitative nature, are reviewed
in \cite{CR11}. All these results are, however, essentially restricted
to the setting of Harris chains, so that their applicability to
infinite-dimensional models is severely limited. In this paper, we
develop the first results of this kind that are generally applicable
beyond the Harris setting and, in particular, that allow to establish
conditional ergodicity in a wide range of infinite-dimensional models.

To give a flavor of the type of problems that our theory will address,
let us briefly describe one example that will be given in Section~\ref{sec:examples} below. Consider the velocity field $u$ of a fluid that is
modeled as a Navier--Stokes equation
\[
du = \bigl\{\nu\Delta u-(u\cdot\nabla)u-\nabla p \bigr\} \,dt + d\tilde w,\qquad
\nabla \cdot u=0 %
\]
with white in time, spatially smooth random forcing $d\tilde w$.
At regular time intervals $t_n=n\delta$, the velocity field is sampled
at the spatial locations $z_1,\ldots,z_r$ with some additive Gaussian noise
$\xi_n$, which yields the observations
\[
Y_n^i = u(t_n,z_i)+
\xi_n^i,\qquad i=1,\ldots,r. %
\]
Such models arise naturally in data assimilation problems \cite{Stu10}.
The process $(X_n,Y_n)_{n\ge0}$ with
$X_n=u(t_n,\cdot)$ is an infinite-dimensional Markov chain.
Classical ergodicity questions include the existence and uniqueness
of an invariant probability $\lambda$, and the convergence to equilibrium
property
\[
\bigl|\mathbf{E}^{x} \bigl[f(X_n) \bigr]-\lambda(f)\bigr|
\mathop{\longrightarrow}^{n\to\infty}0 %
\]
for a sufficiently large class of functions $f$ and initial conditions $x$.
Such questions are far from straightforward for Navier--Stokes equations
and have formed a very active area of research in recent years
(see, e.g., \cite{HM06,Mat07,KS12}). In contrast, we are interested
in the question of conditional ergodicity
\[
\bigl|\mathbf{E}^{x} \bigl[f(X_n)|\mathcal{F}^Y_{0,\infty}
\bigr]- \mathbf{E}^\lambda \bigl[f(X_n)|\mathcal{F}^Y_{0,\infty}
\bigr]\bigr| \mathop{\longrightarrow}^{n\to\infty}0 %
\]
(where $\mathcal{F}^Y_{m,n}=\sigma\{Y_m,\ldots,Y_n\}$),
or, more importantly, its causal counterpart
\[
\bigl|\mathbf{E}^{x} \bigl[f(X_n)|\mathcal{F}^Y_{0,n}
\bigr]- \mathbf{E}^\lambda \bigl[f(X_n)\bigr|\mathcal{F}^Y_{0,n}
\bigr]\bigr| \mathop{\longrightarrow}^{n\to\infty}0, %
\]
which corresponds to stability of the nonlinear filter
$\pi_n^\mu=\mathbf{P}^\mu[X_n\in\cdot |\mathcal{F}^Y_{0,n}]$. In
contrast to convergence to equilibrium of the underlying model,
conditional ergodicity properties yield convergence to equilibrium
of the estimation error of the model given the observations \cite{Kun71}
or the long-term stability of the conditional distributions to
perturbations (such as those that arise in the investigation of
numerical filtering algorithms), cf. \cite{vH09spa}. The interplay
between ergodicity and conditioning is of intrinsic interest in
probability theory and in measurable dynamics, where it is closely
related to notions of relative mixing \cite{Rud04}, and lies at the
heart of stability problems that arise in data assimilation and
nonlinear filtering. The main results of this paper will allow us
to establish conditional ergodicity in a wide range of infinite-dimensional
models, including dissipative stochastic partial differential equations
such as the above Navier--Stokes model, stochastic spin systems,
and stochastic differential delay equations (detailed examples are given
in Section~\ref{sec:examples}).

One of the main difficulties in the investigation of conditional
ergodicity is that conditioning on an infinite observation sequence
$\mathcal{F}^Y_{0,\infty}$ is a very singular operation. Under the
conditional distribution, the unobserved process $(X_n)_{n\ge0}$
remains a Markov chain, albeit an inhomogeneous one with random
transition probabilities depending on the realized path of the
observations $(Y_n)_{n\ge0}$ (in the stationary case this is a Markov
chain in a random environment in the sense of Cogburn and Orey
\cite{Cog84,Ore91}). These conditional transition probabilities are
defined abstractly as regular conditional probabilities, but no explicit
equations are available even in the simplest examples. There is
therefore little hope of analyzing the properties of the conditional
chain ``by hand,'' and one must find a way to deduce the requisite
ergodic properties from their unconditional counterparts. On the other
hand, conditioning is an essentially measure-theoretic operation, and it
is unlikely that the most fruitful approaches to ergodic theory in
infinite dimension, such as topological properties or functional
inequalities, are preserved by the conditional distributions. To move
beyond the setting of Harris chains, we therefore aim to find a way to
encode such weak ergodic properties in a measure-theoretic fashion that
can be shown to be preserved under conditioning.

A central insight of this paper is that certain basic elements of the
classical theory admit \emph{local} formulations
that do not rely on the Markov property. The simplest of these is a
local zero--two law (Section~\ref{sec:loc02})
that characterizes, for a given $E$-valued Markov chain
$(X_n)_{n\ge0}$ and measurable map $\iota\dvtx E\to E'$ to another
space~$E'$, the following total variation ergodic property:
\[
\bigl\|\mathbf{P}^x \bigl[ \bigl(\iota(X_k)
\bigr)_{k\ge n}\in\cdot \bigr]- \mathbf{P}^{x'} \bigl[ \bigl(
\iota(X_k) \bigr)_{k\ge n}\in\cdot \bigr] \bigr\|\mathop{\longrightarrow}^{n\to
\infty}0 \qquad\mbox{for all }x,x'\in E. %
\]
If $\iota$ is injective, then this reduces to the ergodic property of
a Harris chain. By choosing different functions $\iota$,
however, we will find that such results are applicable far beyond the
setting of Harris chains. Let us emphasize that when $\iota$ is not
injective the process $(\iota(X_n))_{n\ge0}$ is generally not Markov,
so that our local ergodic theorems are fundamentally non-Markovian in nature.

In certain cases, this local notion of ergodicity can be applied
directly to infinite-dimensional Markov chains. When the entire chain
does not converge to equilibrium in total variation, it may still be the
case that each finite-dimensional projection of the chain converges in
the above sense. To our knowledge, this \emph{local mixing} property
was first proposed by F\"ollmer \cite{Fol79} in the context of
interacting particle systems; a similar idea appears in \cite{Mat02} for
stochastic Navier--Stokes equations with sufficiently nondegenerate
forcing. By choosing $\iota$ to be a finite-dimensional projection, we
obtain a very useful characterization of the local mixing property
(Section~\ref{sec:locmix}). Our results can also be applied directly to
non-Markovian processes, for example, we will obtain a non-Markovian
ergodic theorem that provides an apparently new characterization of
stationary absolutely regular sequences (Section~\ref{sec:ergnonm}).

While local mixing can be verified in various infinite-dimensional
models, this generally requires a fair amount of nondegeneracy. In truly
degenerate situations, we introduce another idea that exploits
topological properties of the model (Section~\ref{sec:ergwk}). In
dissipative models and in many other Markov chains that converge weakly
to equilibrium, it is possible to construct a coupling of two copies
$X_n,X_n'$ of the chain such that $d(X_n,X_n')\to0$ (cf. \cite{HMS11}). Of course, this need not imply any form of total
variation convergence. Consider, however, the perturbed process
$f(X_n)+\eta_n$ where $f\dvtx E\to\mathbb{R}$ is a Lipschitz function and
$(\eta_n)_{n\ge0}$ is an i.i.d. sequence of auxiliary Gaussian random
variables. When the asymptotic coupling converges sufficiently rapidly,
the process $(f(X_n)+\eta_n)_{n\ge0}$ will be ergodic in the above
total variation sense by the Kakutani theorem. We have thus transformed
a topological property into a measure-theoretic one, which is amenable to
our local ergodic theorems by considering the augmented Markov chain
$(X_n,\eta_n)_{n\ge0}$ with $\iota(x,\eta)=f(x)+\eta$. The added noise
can ultimately be deconvolved, which yields weak-* ergodic theorems for
the original chain $(X_n)_{n\ge0}$ by purely measure-theoretic means.

The local ergodic theorems developed in Section~\ref{sec:zerotwo} are of
independent interest. However, the full benefit of our approach emerges
in the development of the conditional ergodic theory that is undertaken
in Sections \ref{sec:conderg} and \ref{sec:filter}. First, we develop in
Section~\ref{sec:cond02} a conditional counterpart to the local zero--two
law that characterizes the conditional absolute regularity property of a
stationary (non-Markovian) sequence. The remainder of Section~\ref{sec:conderg} is devoted to the inheritance problem. In short, we
show that under a generalization of the nondegeneracy assumption on the
observations
that was introduced in \cite{vH09,TvH12}, the local ergodicity property
is inherited when we condition on the observed component of the model.
Together with the ideas developed in Section~\ref{sec:zerotwo}, this
allows us to obtain various filter stability results in Section~\ref{sec:filter}. After introducing the relevant setting and notations
in Section~\ref{sec:fsetting}, we first develop a general local filter
stability theorem in Section~\ref{sec:flocal}. In Section~\ref{sec:fstab}, we give concrete filter stability theorems for Markov
chains that are locally mixing or that admit asymptotic couplings. We
also investigate unique ergodicity of the filtering process in the
spirit of \cite{Kun71}. Finally, in Section~\ref{sec:conttime}, we
extend our main results to Markov processes in continuous time.
Our general approach in these sections is inspired by the ideas
developed in \cite{vH09} in the Harris setting. However, as is
explained in Section~\ref{sec:conderg}, the approach used in
\cite{vH09,TvH12} relies crucially on the Markov property, and the same
method can therefore not be used in the local setting. Instead, we
develop here a new (and in fact somewhat more direct) method for
establishing the inheritance property that does not rely on
Markov-specific arguments.

To illustrate the wide applicability of our results, we develop in
Section~\ref{sec:examples} several infinite-dimensional examples that
were already mentioned above. Our aim is to demonstrate that the
assumptions of our main results can be verified in several quite
distinct settings. In order not to unduly lengthen the paper, we have
restricted attention to a number of examples whose ergodic properties
are readily verified using existing results in the literature.

Let us conclude the \hyperref[sec:intro]{Introduction} by briefly highlighting two directions
that are not addressed in this paper. First, we emphasize that all the
results in this paper, which rely at their core on martingale convergence
arguments, are qualitative in nature. The development of quantitative
filter stability results is an interesting problem, and this remains
challenging even in finite-dimensional models (cf. \cite{CR11} and the
references therein). Second, let us note that while our theory allows the
unobserved process $X_n$ to be infinite-dimensional under mild conditions,
the main regularity assumptions of this paper (Assumptions
\ref{aspt:locerg} and \ref{aspt:nondeg} below) typically require in
practice that the observations $Y_n$ are ``effectively''
finite-dimensional, for reasons that are discussed in Remark~\ref{rem:idimobs} below. As is illustrated by the examples in
Section~\ref{sec:examples}, our general setting covers a wide range of
models of
practical interest. Nonetheless, conditional ergodicity problems with
degenerate infinite-dimensional observations are of significant interest
in their own right and require separate consideration. In the latter
setting, new probabilistic phenomena can arise; such issues will be
discussed elsewhere.

\begin{rem}[(A note on terminology)]
Throughout this paper, we will use the term \emph{ergodicity} in a broad
sense to denote the asympotic insensitivity of a (possibly inhomogeneous
or random) Markov process to its initial condition. This differs from the
use of the term in the theory of measurable dynamics, where ergodicity
strictly refers to triviality of the invariant \mbox{$\sigma$-}field of a
dynamical system \cite{Wal82}. Unfortunately, no consistent usage of
these terms has emerged in the probabilistic literature. In the theory of
Markov chains, ergodicity is often used to denote either convergence to an
invariant probability \cite{Rev84,Num84,MT09}, or insensitivity to the
initial condition \cite{Haj58}, \cite{Kal02}, Theorem~20.10. In the
absence of a commonly accepted usage and as many different forms of such
properties will appear throughout this paper, we have chosen not to
introduce an overly precise terminology to distinguish between different
notions of ergodicity: to avoid any confusion, the specific ergodic
properties pertaining to each result will always be specified explicitly.
\end{rem}

\section{Local ergodic theorems}
\label{sec:zerotwo}

The goal of this section is to develop a number of simple but powerful
measure-theoretic ergodic theorems that are applicable beyond the
classical setting of Harris chains. Our main tools are the \emph{local}
zero--two laws developed in Section~\ref{sec:loc02}. In the following
subsections, it is shown how these results can be applied in various
different settings. In Section~\ref{sec:locmix}, we consider a notion of
local mixing for Markov chains, due to F\"ollmer \cite{Fol79}, that
provides a natural measure-theoretic generalization of Harris chains
\cite{Rev84} to the infinite-dimensional setting. In Section~\ref{sec:ergnonm},
we obtain an ergodic theorem for non-Markov processes
that yields a new characterization of stationary absolutely regular
sequences. Finally, in Section~\ref{sec:ergwk}, we show how these results
can be combined with the notion of asymptotic coupling (see, e.g.,
\cite{HMS11}) to obtain ergodic theorems in the weak convergence topology
by purely measure-theoretic means.

Throughout this section, we will work in the following canonical
setup. Let $(E,\mathcal{E})$ be a measurable space, and let
$(X_k)_{k\in\mathbb{Z}}$ be the $E$-valued coordinate process defined
on the canonical path space $(\Omega,\mathcal{F})$. That is, we define
$\Omega=E^\mathbb{Z}$, $\mathcal{F}=\mathcal{E}^\mathbb{Z}$, and
$X_k(\omega)=\omega(k)$. We define for $m<n$
\[
X_{m,n}=(X_k)_{m\le k\le n},\qquad \mathcal{F}_{m,n}=
\sigma\{X_{m,n}\}, \qquad\mathcal{F}_+ = \mathcal{F}_{0,\infty},\qquad
\mathcal{F}_- = \mathcal{F}_{-\infty,0}. %
\]
We also define the canonical shift
$\Theta\dvtx \Omega\to\Omega$ as $\Theta(\omega)(n)=\omega(n+1)$.

We will denote by $\mathcal{P}(Z)$ the set of probability measures on
a measurable space $(Z,\mathcal{Z})$, and for $\mu,\nu\in\mathcal{P}(Z)$
we denote by $\|\mu-\nu\|_{\mathcal{Z}_0}$
the total variation of the signed measure $\mu-\nu$ on the $\sigma$-field
$\mathcal{Z}_0\subseteq\mathcal{Z}$, that is,
\[
\|\mu-\nu\|_{\mathcal{Z}_0}= 2\sup_{A\in\mathcal{Z}_0}\bigl|\mu(A)-\nu(A)\bigr|.
\]
For simplicity, we will write $\|\mu-\nu\|=\|\mu-\nu\|_\mathcal{Z}$.
Let us recall that if $K,K'$ are finite kernels and if $\mathcal{Z}$ is
countably
generated, then $x\mapsto\|K(x,\cdot)-K'(x,\cdot)\|$ is measurable
(see, e.g., \cite{vH09}, Lemma~2.4). In this setting, we have
\[
\biggl\|\int \bigl\{K(x,\cdot)-K'(x,\cdot) \bigr\} \mu(dx) \biggr\|\le \int
\bigl\|K(x,\cdot)-K'(x,\cdot)\bigr\| \mu(dx) %
\]
by Jensen's inequality. Moreover, if
$\mathcal{Z}_n\downarrow\mathcal{Z}_\infty:= \bigcap_n\mathcal{Z}_n$
is a decreasing family of $\sigma$-fields, then a simple martingale
argument (e.g., \cite{Der76}, page 117) yields
\[
\|\mu-\nu\|_{\mathcal{Z}_n} \mathop{\longrightarrow}^{n\to\infty} \|\mu-\nu
\|_{\mathcal{Z}_\infty}. %
\]
These facts will be used repeatedly throughout the paper.

\subsection{Local zero--two laws}
\label{sec:loc02}

Let $P\dvtx E\times\mathcal{E}\to[0,1]$ be a transition kernel on
$(E,\mathcal{E})$, and denote by $\mathbf{P}^\mu$ be the probability
measure on $\mathcal{F}_+$ such that $(X_k)_{k\ge0}$ is Markov with
transition kernel $P$ and initial law $X_0\sim\mu\in\mathcal{P}(E)$.
If $P$ is Harris and aperiodic, the Markov chain is ergodic
in the sense that
\[
\bigl\|\mu P^n-\nu P^n\bigr\| = \bigl\|\mathbf{P}^\mu-
\mathbf{P}^\nu\bigr\|_{\mathcal{F}_{n,\infty}}
 \mathop{\longrightarrow}^{n\to\infty}0 \qquad\mbox{for
all }\mu,\nu\in\mathcal{P}(E) %
\]
(cf. \cite{Rev84}, Theorem~6.2.2). Unfortunately, this mode of
convergence can be restrictive in complex models. For example,
when the state space $E$ is infinite-dimensional, such strong convergence
will rarely hold: it is often the case in this setting that $\mu
P^n\perp\nu P^n$
for all $n\ge0$ (cf. Section~\ref{sec:locmix}).

At the heart of this paper lies a simple idea. When total
variation convergence of the full chain fails, it may still be the case
that total variation convergence holds when the chain restricted to a
smaller $\sigma$-field $\mathcal{E}^0\subset\mathcal{E}$: that is, we
intend to establish convergence of $\iota(X_k)$ where
$\iota\dvtx (E,\mathcal{E})\to(E,\mathcal{E}^0)$ is the identity map. As will
become clear in the sequel, such \emph{local} total
variation convergence is frequently sufficient to deduce convergence of
the full chain in a weaker probability distance, while at the same time
admitting a powerful measure-theoretic ergodic theory that will be
crucial for the study of conditional ergodicity in complex models.

The key results of this section are a pair of local zero--two laws that
characterize the local total variation convergence of Markov processes.
Let us fix $\mathcal{E}^0\subseteq\mathcal{E}$ throughout this section,
and define the $\sigma$-fields
\[
\mathcal{F}_{m,n}^0 = \bigvee
_{m\le k\le n} X_k^{-1} \bigl(
\mathcal{E}^0 \bigr),\qquad m<n. %
\]
A central role will be played by the local tail $\sigma$-field
\[
\mathcal{A}^0 = \bigcap_{n\ge0}
\mathcal{F}_{n,\infty}^0. %
\]
Finally, for $x\in E$ we will denote for simplicity
$\mathbf{P}^x=\mathbf{P}^{\delta_x}$.

It is important to note that the local process $\iota(X_k)$ is generally
not Markov, so that the marginal distribution at a fixed time does not
determine the future of this process. Thus, one cannot restrict attention
to the marginal distance $\|\mu P^n-\nu P^n\|_{\mathcal{E}^0}$, but one
must instead consider the entire infinite future
$\|\mathbf{P}^\mu-\mathbf{P}^\nu\|_{\mathcal{F}_{n,\infty}^0}$. Of
course, when $\mathcal{E}^0=\mathcal{E}$, these notions coincide.

\begin{thmm}[(Local zero--two law)]
\label{thmm:loc02}
The following are equivalent.
\begin{enumerate}[1.]
\item[1.] The Markov chain is locally ergodic:
\[
\bigl\|\mathbf{P}^\mu-\mathbf{P}^\nu\bigr\|_{\mathcal{F}_{n,\infty}^0}
\mathop{\longrightarrow}^{n\to\infty}0\qquad \mbox{for every }\mu,\nu\in\mathcal{P}(E). %
\]
\item[2.] The local tail $\sigma$-field is trivial:
\[
\mathbf{P}^\mu(A)\in\{0,1\} \qquad\mbox{for every }A\in
\mathcal{A}^0\mbox{ and } \mu\in\mathcal{P}(E). %
\]
\item[3.] The Markov chain is locally irreducible: there exists
$\alpha>0$ such that
\[
\forall x,x'\in E, \exists n\ge0 \mbox{ such that }\qquad \bigl\|
\mathbf{P}^x-\mathbf{P}^{x'}\bigr\|_{\mathcal{F}_{n,\infty}^0}\le2-\alpha.
\]
\end{enumerate}
\end{thmm}

Zero-two laws of this type appear naturally in the theory of Harris
chains \cite{Rev84,Der76,OS70}. It is somewhat surprising that the
Markov property proves to be inessential in the proof, which enables the
present local formulation.

\begin{pf*}{Proof of Theorem~\ref{thmm:loc02}}
We prove $2\Rightarrow1\Rightarrow3\Rightarrow2$.

($2\Rightarrow1$).
Assumption~2 implies that $\mathbf{P}^\mu(A)=\mathbf{P}^\nu(A)$
for all $A\in\mathcal{A}^0$
[if not, then $\mathbf{P}^\rho(A)=1/2$ for
$\rho=(\mu+\nu)/2$, a contradiction]. Therefore,
\[
\bigl\|\mathbf{P}^\mu-\mathbf{P}^\nu\bigr\|_{\mathcal{F}^0_{n,\infty}}
\mathop{\longrightarrow}^{n\to\infty} \bigl\|\mathbf{P}^\mu-\mathbf{P}^\nu
\bigr\|_{\mathcal{A}^0}=0. %
\]

($1\Rightarrow3$). This is obvious.

($3\Rightarrow2$).
Assume that condition $2$ does not hold.
Then there exists $A\in\mathcal{A}^0$ and $\mu\in\mathcal{P}(E)$ such that
$0<\mathbf{P}^\mu(A)<1$. Define $f=\mathbf{1}_A-\mathbf{1}_{A^c}$, and note
that
\[
\mathbf{E}^{X_n} \bigl[f\circ\Theta^{-n} \bigr] =
\mathbf{E}^\mu[f|\mathcal{F}_{0,n}]\mathop{\longrightarrow}
^{n\to\infty}f,
\qquad\mathbf{P}^\mu\mbox{-a.s.} %
\]
by the Markov property and the martingale convergence theorem.
(Recall that for any $\mathcal{A}^0$-measurable function
$f$, the function $f\circ\Theta^{-n}$ is unambiguously defined and
$\mathcal{A}^0$-measurable for every $n\in\mathbb{Z}$, cf. \cite{Rev84},
pages~186--187.)

Define the probability measure $\mathbf{Q}$
on $\Omega\times\Omega$ as $\mathbf{P}^\mu\otimes\mathbf{P}^\mu$, and
denote by $(X_n,X_n')_{n\ge0}$ the coordinate process on $\Omega\times
\Omega$.
Fix $\alpha>0$. Then
\[
\mathbf{Q} \bigl[\bigl|\mathbf{E}^{X_n} \bigl[f\circ\Theta^{-n}
\bigr]- \mathbf{E}^{X_n'} \bigl[f\circ\Theta^{-n} \bigr]\bigr|>2-
\alpha \bigr] \mathop{\longrightarrow}^{n\to\infty} 2 \mathbf{P}^\mu(A)
\mathbf{P}^\mu \bigl(A^c \bigr)>0. %
\]
Thus, there exist $N\ge0$ and $x,x'\in E$ such that
$|\mathbf{E}^{x}[f\circ\Theta^{-N}]-
\mathbf{E}^{x'}[f\circ\Theta^{-N}]|>2-\alpha$. But note
that $|f|\le1$ and $f\circ\Theta^{-N}$ is
$\mathcal{A}^0$-measurable. Therefore,
\[
\bigl\|\mathbf{P}^{x}-\mathbf{P}^{x'}\bigr\|_{\mathcal{F}_{n,\infty}^0} \ge \bigl\|
\mathbf{P}^{x}-\mathbf{P}^{x'}\bigr\|_{\mathcal{A}^0} \ge\bigl|
\mathbf{E}^{x} \bigl[f\circ\Theta^{-N} \bigr]-
\mathbf{E}^{x'} \bigl[f\circ\Theta^{-N} \bigr]\bigr|>2-\alpha
\]
for all $n\ge0$. As $\alpha>0$ is arbitrary, condition 3 is contradicted.
\end{pf*}

The characterization in Theorem~\ref{thmm:loc02} does not require the
existence of an invariant probability. However, when such a probability
exists, we can obtain a useful stationary variant of the local zero--two
law that will be proved next. The advantage of the stationary
zero--two law is that it does not require uniform control in condition 3.
On the other hand, the resulting convergence only holds for almost every
initial condition.

\begin{thmm}[(Local stationary zero--two law)]
\label{thmm:sloc02}
Suppose $\mathcal{E}^0$ is countably generated.
Given a $P$-invariant probability $\lambda$,
the following are equivalent:
\begin{enumerate}[1.]
\item[1.] The Markov chain is a.e. locally ergodic:
\[
\bigl\|\mathbf{P}^x-\mathbf{P}^\lambda\bigr\|_{\mathcal{F}_{n,\infty}^0}
\mathop{\longrightarrow}^{n\to\infty}0\qquad \mbox{for }\lambda\mbox{-a.e. } x, %
\]
or, equivalently,
\[
\bigl\|\mathbf{P}^x-\mathbf{P}^{x'}\bigr\|_{\mathcal{F}_{n,\infty}^0}
\mathop{\longrightarrow}^{n\to\infty}0 \qquad\mbox{for }\lambda\otimes\lambda\mbox{-a.e. }
\bigl(x,x' \bigr). %
\]
\item[2.] The local tail $\sigma$-field is a.e. trivial:
\[
\mathbf{P}^x(A)=\mathbf{P}^x(A)^2=
\mathbf{P}^{x'}(A) \qquad\forall A\in\mathcal{A}^0, \lambda\otimes
\lambda\mbox{-a.e. } \bigl(x,x' \bigr). %
\]
\item[3.] The Markov chain is a.e. locally irreducible:
\[
\mbox{for }\lambda\otimes\lambda\mbox{-a.e. } \bigl(x,x' \bigr),
\exists n\ge0 \mbox{ such that} \qquad\bigl\|\mathbf{P}^x-\mathbf{P}^{x'}
\bigr\|_{\mathcal{F}_{n,\infty}^0}<2. %
\]
\end{enumerate}
\end{thmm}

\begin{pf}
The equivalence of the two statements of condition 1 follows from
\begin{eqnarray*}
\bigl\|\mathbf{P}^x-\mathbf{P}^\lambda\bigr\|_{\mathcal{F}_{n,\infty}^0} &\le&
\int\bigl\|\mathbf{P}^x-\mathbf{P}^{x'}\bigr\|_{\mathcal{F}_{n,\infty}^0} \lambda
\bigl(dx' \bigr),
\\
\bigl\|\mathbf{P}^x-\mathbf{P}^{x'}\bigr\|_{\mathcal{F}_{n,\infty}^0} &\le& \bigl\|
\mathbf{P}^x-\mathbf{P}^\lambda\bigr\|_{\mathcal{F}_{n,\infty}^0}+ \bigl\|
\mathbf{P}^{x'}-\mathbf{P}^\lambda\bigr\|_{\mathcal{F}_{n,\infty}^0}.
\end{eqnarray*}
The proofs of $2\Rightarrow1\Rightarrow3$ are identical to the
corresponding proofs in Theorem~\ref{thmm:loc02}. It therefore remains to
prove $3\Rightarrow2$. To this end, define
\[
\beta_n \bigl(x,x' \bigr)= \bigl\|\mathbf{P}^x-
\mathbf{P}^{x'}\bigr\|_{\mathcal{F}_{n,\infty}^0},\qquad \beta \bigl(x,x'
\bigr)= \bigl\|\mathbf{P}^x-\mathbf{P}^{x'}\bigr\|_{\mathcal{A}^0}.
\]
As $\mathcal{E}^0$ is countably generated, the maps $\beta_n$ are
measurable. Moreover, as $\beta_n\downarrow\beta$ pointwise as
$n\to\infty$, the map $\beta$ is measurable also.

By the Markov property, we have $\mathbf{E}^x(\mathbf{1}_A\circ\Theta)=
\int P(x,dz) \mathbf{P}^{z}(A)$ for every $x$ and
$A\in\mathcal{F}_{n-1}^0$. Thus we obtain by Jensen's inequality
\[
\bigl\|\mathbf{P}^x-\mathbf{P}^{x'}\bigr\|_{\mathcal{F}_{n,\infty}^0} \le \int
P(x,dz) P \bigl(x',dz' \bigr) \bigl\|\mathbf{P}^z-
\mathbf{P}^{z'}\bigr\|_{\mathcal{F}_{n-1,\infty}^0}, %
\]
so $\beta_n\le(P\otimes P)\beta_{n-1}$.
Thus $\beta\le(P\otimes P)\beta$ by dominated convergence.

Define
$\mathbf{Q}=\mathbf{P}^\lambda\otimes\mathbf{P}^\lambda$ and
$\mathbf{Q}^{x,x'}=\mathbf{P}^x\otimes\mathbf{P}^{x'}$ on
$\Omega\times\Omega$. Then
\[
\mathbf{E}_{\mathbf{Q}} \bigl[\beta \bigl(X_{n+1},X_{n+1}'
\bigr)|X_{0,n},X_{0,n}' \bigr] = (P\otimes P)
\beta \bigl(X_{n},X_{n}' \bigr) \ge \beta
\bigl(X_{n},X_{n}' \bigr),\qquad \mathbf{Q}\mbox{-a.s.} %
\]
Thus, $\beta(X_n,X_n')$ is a bounded and stationary submartingale under
$\mathbf{Q}$, and
\[
\mathbf{E}_{\mathbf{Q}} \bigl[\bigl|\beta \bigl(X_0,X_0'
\bigr)-\beta \bigl(X_n,X_n' \bigr)\bigr| \bigr]
= \mathbf{E}_{\mathbf{Q}} \bigl[\bigl|\beta \bigl(X_k,X_k'
\bigr)-\beta \bigl(X_{n+k},X_{n+k}' \bigr)\bigr|
\bigr] \mathop{\longrightarrow}^{k\to\infty}0 %
\]
by stationarity and the martingale convergence theorem. It follows
that $\beta(X_0,X_0')=\beta(X_n,X_n')$ for all $n\ge0$,
$\mathbf{Q}$-a.s.
By disintegration, there is a measurable set $H'\subseteq\Omega\times
\Omega$
with $(\lambda\otimes\lambda)(H')=1$ such that
\[
\mathbf{Q}^{x,x'} \bigl[\beta \bigl(x,x' \bigr)=\beta
\bigl(X_n,X_n' \bigr) \mbox{ for all }n\ge0
\bigr]=1 \qquad\mbox{for all } \bigl(x,x' \bigr)\in H'.
\]

In the remainder of the proof, we assume that condition 3 holds, and
we fix a measurable set $H\subseteq H'$ with $(\lambda\otimes\lambda)(H)=1$
such that
\[
\forall \bigl(x,x' \bigr)\in H, \exists n\ge0 \mbox{ such that}\qquad
\beta_n \bigl(x,x' \bigr)<2. %
\]
Suppose condition 2 does not hold. Then there exist $A\in\mathcal{A}^0$
and $(x,x')\in H$ such that either $0<\mathbf{P}^x(A)<1$ or
$\mathbf{P}^x(A)\ne\mathbf{P}^{x'}(A)$. Define
$f=\mathbf{1}_A-\mathbf{1}_{A^c}$ and fix $\alpha>0$.
Proceeding as in the proof of Theorem~\ref{thmm:loc02}, we find that
\begin{eqnarray*}
&&\mathbf{Q}^{x,x'} \bigl[\bigl|\mathbf{E}^{X_n} \bigl[f\circ
\Theta^{-n} \bigr]- \mathbf{E}^{X_n'} \bigl[f\circ
\Theta^{-n} \bigr]\bigr|>2-\alpha \bigr]\\
&&\qquad \mathop{\longrightarrow}^{n\to\infty}
\mathbf{P}^x(A)\mathbf{P}^{x'} \bigl(A^c
\bigr)+ \mathbf{P}^x \bigl(A^c \bigr)\mathbf{P}^{x'}(A)>0.
\end{eqnarray*}
Note that as $|f|\le1$ and $f\circ\Theta^{-n}$ is
$\mathcal{A}^0$-measurable, we have
\[
\bigl|\mathbf{E}^{X_n} \bigl[f\circ\Theta^{-n} \bigr]-
\mathbf{E}^{X_n'} \bigl[f\circ\Theta^{-n} \bigr]\bigr| \le\beta
\bigl(X_n,X_n' \bigr) = \beta
\bigl(x,x' \bigr),\qquad \mathbf{Q}^{x,x'}\mbox{-a.s.} %
\]
It follows that $\beta(x,x')>2-\alpha$, and we therefore have
$\beta(x,x')=2$ as $\alpha>0$ was arbitrary.
But by construction there exists $n\ge0$ such that
$\beta(x,x')\le\beta_n(x,x')<2$, and we have the desired contradiction.
\end{pf}

Theorems \ref{thmm:loc02} and \ref{thmm:sloc02}, while elementary, play a
fundamental role in our theory. In the following subsections, we will
see that these results have a broad range of applicability that goes far
beyond the setting of Harris chains.

\subsection{Local mixing in infinite dimension}
\label{sec:locmix}

Markov chains in an infinite-dimensional state space are rarely amenable
to the classical theory of Harris chains. The key obstacle is that total
variation convergence requires nonsingularity of the transition
probabilities. This is not restrictive in finite dimension, but fails
in infinite dimension even in the most trivial examples.

\begin{example}
\label{ex:trivial}
Let $(X_k)_{k\ge0}$ be the Markov chain in
$\{-1,+1\}^\mathbb{N}$ such that each coordinate
$(X_k^i)_{k\ge0}$ is an independent Markov chain in $\{-1,+1\}$
with transition probabilities $0<p_{-1,+1}=p_{+1,-1}<1/2$.
Clearly, each coordinate is a Harris
chain, and the law of $X_n$ converges weakly as $n\to\infty$ to its
unique invariant measure $\lambda$ for any initial condition.
Nonetheless, $\delta_{1^{\mathbb{N}}}P^n$ and $\lambda$ are mutually
singular for all $n\ge0$ ($\delta_{1^{\mathbb{N}}}P^n$ and $\lambda$
possess i.i.d. coordinates with a different law), so
$X_n$ cannot converge in total variation.
\end{example}

As the classical measure-theoretic theory fails to yield satisfactory
results, the ergodic theory of infinite-dimensional Markov chains is
frequently approached by means of topological methods. A connection
between topological methods and local zero--two laws will be investigated
in Section~\ref{sec:ergwk} below. On the other hand, one may seek a
purely measure-theoretic counterpart of the notion of a Harris chain
that is adapted to the infinite-dimensional setting. We now describe
such a notion due to F\"ollmer \cite{Fol79}.

Throughout this section, we adopt the same setting as in Section~\ref{sec:loc02}. To formalize the notion of an infinite-dimensional
Markov chain, we assume that the state space $(E,\mathcal{E})$ is
contained in a countable product: that is, there exist a countable set~$I$ and measurable spaces $(E^i,\mathcal{E}^i)$ such that
\[
(E,\mathcal{E}) \subseteq\prod_{i\in I}
\bigl(E^i,\mathcal{E}^i \bigr). %
\]
Each $i\in I$ plays the role of a single dimension of the model.
We will write $x=(x^i)_{i\in I}$ for $x\in E$, and for $J\subseteq I$
we denote by $x^J=(x^i)_{i\in J}$ the natural projection of $x$ onto
$\prod_{i\in J}E^i$. For $m<n$, we define
the quantities $X_{m,n}^J$ and $\mathcal{F}_{m,n}^J$ in the obvious manner.
Moreover, we define the local tail $\sigma$-fields
\[
\mathcal{A}^J = \bigcap_{n\ge0}
\mathcal{F}^J_{n,\infty}, \qquad\mathcal{A}_{\mathrm{loc}} = \bigvee
_{|J|<\infty}\mathcal{A}^J. %
\]
That is, $\mathcal{A}_{\mathrm{loc}}$ is generated by the asymptotic events
associated to all finite-dimensional projections of the infinite-dimensional
chain.

We now introduce F\"ollmer's notion of local mixing, which states that
each finite-dimensional projection of the model converges in total
variation (the term ``local ergodicity'' would be more in line with the
terminology used in this paper, but we will conform to the definition
given in \cite{Fol79}). Let us emphasize, as in the previous section, that
the finite-dimensional projection of an infinite-dimensional Markov chain
is generally not Markov.

\begin{defn}[(Local mixing)]
\label{def:follmer}
A Markov chain $(X_k)_{k\ge0}$ taking values in the countable product
space $(E,\mathcal{E}) \subseteq\prod_{i\in I}(E^i,\mathcal{E}^i)$ is
\emph{locally mixing} if
\[
\bigl\|\mathbf{P}^\mu-\mathbf{P}^\nu\bigr\|_{\mathcal{F}_{n,\infty}^J}
\mathop{\longrightarrow}^{n\to\infty}0 \qquad\mbox{for all } \mu,\nu\in\mathcal{P}(E)\mbox{ and } J
\subseteq I, |J|<\infty. %
\]
\end{defn}

In the finite-dimensional case $|I|<\infty$, this definition reduces to
the ergodic property of Harris chains. Moreover, in the
infinite-dimensional setting, F\"ollmer~\cite{Fol79} proves a
characterization of local mixing in complete analogy with the
Blackwell--Orey equivalence in the theory of Harris chains \cite{Rev84}, Chapter~6.
It therefore appears that local mixing is the natural
measure-theoretic generalization of the Harris theory to the infinite
dimensional setting.

Unfortunately, the characterization given in \cite{Fol79} is of limited
use for the purpose of establishing the local mixing property of a given
Markov chain: only a very strong verifiable sufficient condition is
given there (in the spirit of the Dobrushin uniqueness condition for
Gibbs measures). The missing ingredient is a zero--two law,
which we can now give as a simple corollary of the results in
Section~\ref{sec:loc02}. This completes the characterization of local mixing
given in \cite{Fol79}, and provides a concrete tool to verify this
property.

\begin{cor}[(Local mixing theorem)]
\label{cor:locmix}
The following are equivalent.
\begin{enumerate}[1.]
\item[1.]$(X_k)_{k\ge0}$ is locally mixing.
\item[2.]$\mathcal{A}_{\mathrm{loc}}$ is $\mathbf{P}^\mu$-trivial
for every $\mu\in\mathcal{P}(E)$.
\item[3.] For every $J\subseteq I$, $|J|<\infty$, there exists $\alpha>0$
such that
\[
\forall x,x'\in E, \exists n\ge0 \mbox{ such that} \qquad\bigl\|
\mathbf{P}^x-\mathbf{P}^{x'}\bigr\|_{\mathcal{F}_{n,\infty}^J}\le2-\alpha.
\]
\end{enumerate}
\end{cor}

\begin{pf}
Note that $\mathcal{A}_{\mathrm{loc}}$ is $\mathbf{P}^\mu$-trivial if and
only if $\mathcal{A}^J$ is $\mathbf{P}^\mu$-trivial for all
$|J|<\infty$. Thus the result follows immediately from
Theorem~\ref{thmm:loc02}.
\end{pf}

Condition 3 of Corollary~\ref{cor:locmix} can be used directly to verify
the local mixing property in infinite-dimensional models that possess a
sufficient degree of nondegeneracy. For example, in the setting of
stochastic Navier--Stokes equations with additive noise (cf. Section~\ref{sec:exsns}), the approach developed in \cite{EMS01,Mat02}
can be
used to show that condition 3 holds under the assumption that every
Fourier mode is forced by an independent Brownian motion (in this
setting, each dimension $i\in I$ corresponds to a single Fourier mode of
the system).
However, in degenerate models (e.g., where some modes are
unforced or when the noise is not additive), local mixing may be
difficult or impossible to establish. In Section~\ref{sec:ergwk} below,
we will introduce a technique that will significantly extend the
applicability of our results.

\begin{rem}
One can of course also obtain a stationary counterpart of
Corollary~\ref{cor:locmix} by applying Theorem~\ref{thmm:sloc02} rather than
Theorem~\ref{thmm:loc02}. As the result
is essentially identical, we do not state it explicitly.
\end{rem}

\subsection{Ergodicity of non-Markov processes}
\label{sec:ergnonm}

As the local zero--two laws introduced in Section~\ref{sec:loc02} are
essentially non-Markovian, they can be used to investigate the ergodic
theory of non-Markov processes. Let us illustrate this idea by
developing a new (to the best of our knowledge) characterization of
stationary absolutely regular sequences.

In this section, we assume that $(E,\mathcal{E})$ is a Polish space
(the Polish assumption is made to ensure the existence of regular
conditional probabilities), and
let $\mathbf{P}$ be a stationary probability measure on
$(\Omega,\mathcal{F})$. Let us recall the well-known notion of absolute
regularity \cite{VR59} (sometimes called $\beta$-mixing).

\begin{defn}[(Absolute regularity)]
\label{def:absreg}
A stationary sequence $(X_k)_{k\in\mathbb{Z}}$ is said to be
\emph{absolutely regular} if the following holds:
\[
\bigl\|\mathbf{P}[X_{-\infty,0}, X_{k,\infty}\in\cdot ]-
\mathbf{P}[X_{-\infty,0}\in\cdot ]\otimes \mathbf{P}[X_{k,\infty}\in\cdot
]\bigr\| \mathop{\longrightarrow}^{k\to\infty}0. %
\]
\end{defn}

We obtain the following characterization.

\begin{cor}
\label{cor:absreg}
Let $(X_k)_{k\in\mathbb{Z}}$ be a stationary sequence. Choose any
version
$\mathbf{P}^{x_{-\infty,0}}$ of the regular conditional
probability $\mathbf{P}[ \cdot |\mathcal{F}_-]$, and
define the measure
$\mathbf{P}^-=\mathbf{P}[X_{-\infty,0}\in\cdot ]$.
The following are equivalent:
\begin{enumerate}[1.]
\item[1.]$(X_k)_{k\in\mathbb{Z}}$ is absolutely regular.
\item[2.]$\|\mathbf{P}^{x_{-\infty,0}}-\mathbf{P}\|_{\mathcal{F}_{k,\infty}}
\mathop{\longrightarrow}\limits^{k\to\infty}0$ for $\mathbf{P}^-$-a.e. $x_{-\infty,0}$.
\item[3.] For $\mathbf{P}^-\otimes\mathbf{P}^-$-a.e. $(x_{-\infty,0},
\tilde x_{-\infty,0})$, there exists $k\ge0$ such that
$\mathbf{P}^{x_{-\infty,0}}$ and
$\mathbf{P}^{\tilde x_{-\infty,0}}$ are not
mutually singular on $\mathcal{F}_{k,\infty}$.
\end{enumerate}
\end{cor}

We remark that the notation here was chosen in direct analogy with the
usual notation for Markov chains, and should be thought of in this spirit:
just as $\mathbf{P}^x$ denotes the law
of a Markov chain started at the point $x$,
$\mathbf{P}^{x_{-\infty,0}}$ denotes the law of a non-Markovian process
given the initial history $x_{-\infty,0}$.

The proof of Corollary~\ref{cor:absreg} requires a basic property of the
total variation distance that we state here as a lemma for future
reference.

\begin{lem}
\label{lem:margtv}
Let $H_1,H_2$ be Polish spaces, and let $X_i(x_1,x_2)=x_i$.
Let $\mathbf{R},\mathbf{R}'$ be probabilities on $H_1\times H_2$.
If $\mathbf{R}[X_1\in\cdot ]=\mathbf{R}'[X_1\in\cdot ]$, then
\[
\bigl\|\mathbf{R}-\mathbf{R}'\bigr\| = \mathbf{E}_{\mathbf{R}} \bigl[\bigl\|
\mathbf{R}[X_2\in\cdot |X_1]- \mathbf{R}'[X_2
\in\cdot |X_1]\bigr\| \bigr]. %
\]
\end{lem}

\begin{pf}
By the definition of the total variation distance, we have
\[
\bigl\|\mathbf{R}[X_2\in\cdot |X_1]- \mathbf{R}'[X_2
\in\cdot |X_1]\bigr\| \ge 2 \bigl\{\mathbf{R}[A|X_1]-
\mathbf{R}'[A|X_1] \bigr\} %
\]
for every measurable $A\subseteq H_1\times H_2$. Taking the expectation
on both sides and using
$\mathbf{R}[X_1\in\cdot ]=\mathbf{R}'[X_1\in\cdot ]$, we obtain the
lower bound
\[
\bigl\|\mathbf{R}-\mathbf{R}'\bigr\| = 2\sup_A \bigl
\{\mathbf{R}(A)-\mathbf{R}'(A) \bigr\} \le \mathbf{E}_{\mathbf{R}}
\bigl[\bigl\|\mathbf{R}[X_2\in\cdot |X_1]-
\mathbf{R}'[X_2\in\cdot |X_1]\bigr\| \bigr].
\]
But by the existence of
a measurable version of the Radon--Nikodym density
between kernels \cite{DM82}, Theorem~V.58, there exists a measurable
$B\subseteq H_1\times H_2$ such that
\[
\bigl\|\mathbf{R}[X_2\in\cdot |X_1]- \mathbf{R}'[X_2
\in\cdot |X_1]\bigr\| = 2 \bigl\{\mathbf{R}[B|X_1]-
\mathbf{R}'[B|X_1] \bigr\}. %
\]
Proceeding as above yields the converse inequality.
\end{pf}

\begin{pf*}{Proof of Corollary~\ref{cor:absreg}}
Applying Lemma~\ref{lem:margtv} above to the measures
$\mathbf{P}[X_{-\infty,0},X_{k,\infty}\in\cdot ]$ and
$\mathbf{P}[X_{-\infty,0}\in\cdot ]\otimes
\mathbf{P}[X_{k,\infty}\in\cdot ]$, it follows directly that
the process $(X_k)_{k\in\mathbb{Z}}$
is absolutely regular if and only if
\[
\mathbf{E} \bigl[\bigl\| \mathbf{P}[X_{k,\infty}\in\cdot |\mathcal{F}_-]-
\mathbf{P}[X_{k,\infty}\in\cdot ]\bigr\| \bigr] \mathop{\longrightarrow}^{k\to\infty} 0.
\]
But as $\|\mathbf{P}[X_{k,\infty}\in\cdot |\mathcal{F}_-]-
\mathbf{P}[X_{k,\infty}\in\cdot ]\|$ is pointwise decreasing in $k$,
the equivalence between conditions 1 and 2 follows immediately.

Now define the $E^{\mathbb{Z}_-}$-valued process $Z_k =
X_{-\infty,k}$. Then $(Z_k)_{k\in\mathbb{Z}}$ is clearly a Markov chain
with transition kernel $Q(z,A)=\mathbf{P}^z[(z,X_1)\in A]$ and
invariant probability $\mathbf{P}^-$.
We apply Theorem~\ref{thmm:sloc02} to the Markov chain
$(Z_k)_{k\in\mathbb{Z}}$, where $\mathcal{E}^0$ in Theorem~\ref{thmm:sloc02}
is the $\sigma$-field generated by the first coordinate of
$E^{\mathbb{Z}_-}$. This yields immediately the equivalence between
conditions 2 and 3.
\end{pf*}

While conditions 1 and 2 of Corollary~\ref{cor:absreg} are standard,
condition 3 appears at first sight to be substantially weaker: all that
is needed is that, for almost every pair of initial histories, we can
couple the future evolutions with nonzero success probability. This is
reminiscent to the corresponding result for Harris chains, and one could
argue that absolutely regular sequences provide a natural generalization
of the Harris theory to non-Markov processes. In this spirit, Berbee
\cite{Ber86} has shown that absolutely regular sequences admit a
decomposition into cyclic classes much like in the Markov setting.

The elementary observation used in the proof of Corollary~\ref{cor:absreg} is that any non-Markov process $X_k$ can be made Markov
by considering the history process $Z_k=X_{-\infty,k}$. However, the
process $Z_k$ is highly degenerate: its transition probabilities are
mutually singular for any distinct pair of initial conditions. For this
reason, the classical Harris theory is of no use in investigating the
ergodicity of non-Markov processes; the \emph{local} nature of the
zero--two laws developed in Section~\ref{sec:loc02} is the key to
obtaining nontrivial results.

\begin{rem}
Along the same lines, one can also obtain a counterpart of
Theorem~\ref{thmm:loc02} for non-Markov processes.
The latter is useful for the investigation of delay equations
with infinite memory. As no new ideas are involved, we leave
the formulation of such a result to the reader.
\end{rem}

\subsection{Weak convergence and asymptotic coupling}
\label{sec:ergwk}

In the previous sections, we have employed the local zero--two laws
directly to obtain ergodic properties in the total variation distance.
However, even local total variation convergence is still too strong a
requirement in many cases of interest. In this section, we introduce a
technique that allows us to deduce ergodic properties of weak
convergence type from the local zero--two laws. This significantly
extends the range of applicability of our techniques.

Throughout this section, we adopt the same setting as in Section~\ref{sec:loc02}. We will assume in addition that the state space $E$ is
Polish and is endowed with its Borel $\sigma$-field $\mathcal{E}$ and a
complete metric $d$. Denote by $U_b(E)$ the uniformly continuous and
bounded functions on $E$, and let $\mathrm{Lip}(E)$ be the
class of functions $f\in U_b(E)$ such that $\|f\|_\infty\le1$ and
$|f(x)-f(y)|\le d(x,y)$ for all $x,y\in E$. Let $\mathcal{M}(E)$ be the
space of signed finite measures on $E$, and define the
bounded-Lipschitz norm $\|\varrho\|_{\mathrm{BL}}=\sup_{f\in\mathrm{Lip}(E)}|\varrho f|$ for
$\varrho\in\mathcal{M}(E)$. We recall for future reference that
$x\mapsto\|K(x,\cdot)-K'(x,\cdot)\|_{\mathrm{BL}}$ is measurable when
$K,K'$ are finite kernels; see, for example, \cite{vH09uo}, Lemma A.1.

A \emph{coupling} of two probability measures
$\mathbf{P}_1,\mathbf{P}_2$ on $\Omega$ is a probability measure
$\mathbf{Q}$ on $\Omega\times\Omega$ such that the first marginal of
$\mathbf{Q}$ coincides with $\mathbf{P}_1$ and the second marginal
coincides with $\mathbf{P}_2$. Let us denote the family of all couplings
of $\mathbf{P}_1,\mathbf{P}_2$ by $\mathcal{C}(\mathbf{P}_1,\mathbf{P}_2)$.
To set the stage for our result, let us
recall the coupling characterization of the total variation distance
\cite{Lin02}, page 19:
\[
\|\mathbf{P}_1-\mathbf{P}_2\|_{\mathcal{F}_{n,\infty}} = 2\min
\bigl\{\mathbf{Q} \bigl[X_{n,\infty}\ne X_{n,\infty}'
\bigr] \dvtx \mathbf{Q}\in\mathcal{C}(\mathbf{P}_1,
\mathbf{P}_2) \bigr\}. %
\]
Consider for simplicity the classical zero--two law (Theorem~\ref{thmm:loc02} for $\mathcal{E}^0=\mathcal{E}$).
Its basic condition reads: there exists $\alpha>0$ such that
\[
\forall x,x'\in E, \exists n\ge0 \mbox{ such that} \qquad\bigl\|
\mathbf{P}^x-\mathbf{P}^{x'}\bigr\|_{\mathcal{F}_{n,\infty}}\le2-\alpha.
\]
By the coupling characterization of the total variation distance,
this condition can be equivalently stated as follows:
there exists $\alpha>0$ such that
\[
\forall x,x'\in E, \exists \mathbf{Q}\in\mathcal{C} \bigl(
\mathbf{P}^x,\mathbf{P}^{x'} \bigr) \mbox{ such that}\qquad
\mathbf{Q} \Biggl[\sum_{n=0}^\infty
\mathbf{1}_{X_n\ne X_n'}<\infty \Biggr]\ge\alpha. %
\]
The message of the following theorem is that if one replaces the
discrete distance $\mathbf{1}_{X_n\ne X_n'}$
by the topological distance $d(X_n,X_n')$,
one obtains an ergodic theorem with respect to the bounded-Lipschitz
(rather than total variation) distance. This is a much weaker assumption:
it is not necessary to construct an \emph{exact} coupling where
$X_n,X_n'$ eventually coincide with positive probability,
but only an \emph{asymptotic} coupling where $X_n,X_n'$ converge toward
each other. The latter can often be accomplished even in degenerate
situations.

\begin{thmm}[(Weak-* ergodicity)]
\label{thmm:wk02}
Suppose there exists $\alpha>0$ so that
\[
\forall x,x'\in E, \exists \mathbf{Q}\in\mathcal{C} \bigl(
\mathbf{P}^x,\mathbf{P}^{x'} \bigr) \mbox{ such that}\qquad
\mathbf{Q} \Biggl[\sum_{n=0}^\infty d
\bigl(X_n,X_n' \bigr)^2<\infty
\Biggr]\ge\alpha. %
\]
Then the Markov chain is weak-* ergodic in the sense that
\[
\bigl\|\mu P^n-\nu P^n\bigr\|_{\mathrm{BL}}\mathop{\longrightarrow}^{n\to
\infty}0 \qquad\mbox{for every }\mu,\nu\in\mathcal{P}(E). %
\]
\end{thmm}

It is interesting to compare Theorem~\ref{thmm:wk02} to the weak-* ergodic
theorems obtained in \cite{HMS11}, Section~2.2, in terms of asymptotic
coupling. In contrast to those results, Theorem~\ref{thmm:wk02} requires no specific recurrence structure, Markovian
couplings, control on the coupling probability $\alpha$ as a function
of $x,x'$ or even the existence of an invariant probability.
On the other hand, Theorem~\ref{thmm:wk02} requires the
asymptotic coupling to converge sufficiently rapidly so that $\sum d(X_n,X_n')^2<\infty$ (this is not a serious issue in most
applications), while the results in \cite{HMS11} are in principle
applicable to couplings with an arbitrarily slow convergence rate. These
results are therefore complementary.

However, it should be emphasized that the feature of Theorem~\ref{thmm:wk02} that is of key importance for our purposes is that its
proof reduces the problem to the local zero--two law of Section~\ref{sec:loc02}. Using this technique, we can therefore extend the
applicability of purely measure-theoretic results that are based on
zero--two laws to a wide class of weak-* ergodic Markov chains. This
idea will be crucial to establishing conditional ergodicity in
degenerate infinite-dimensional models (see Section~\ref{sec:examples}
for examples).

\begin{pf*}{Proof of Theorem~\ref{thmm:wk02}}
Let $(\bar E,\mathcal{\bar E})=(E\times\mathbb{R},\mathcal{E}\otimes
\mathcal{B}(\mathbb{R}))$. Consider the $\bar E$-valued process
$(Z_n)_{n\ge0}$ (defined on its canonical probability space) such that
$Z_n=(X_n,\xi_n)$, where $(\xi_n)_{n\ge0}$ is an i.i.d. sequence of
standard Gaussian random variables independent of the Markov chain
$(X_n)_{n\ge0}$.\break Clearly $(Z_n)_{n\ge0}$ is itself a Markov chain.
Given $f\in\mathrm{Lip}(E)$, we will apply Theorem~\ref{thmm:loc02} to $(Z_n)_{n\ge0}$ with the local $\sigma$-field
$\mathcal{\bar E}^0=\sigma\{g\}$, $g(x,y)=f(x)+y$.

We begin by noting a standard estimate.

\begin{lem}
\label{lem:hell}
Let $(\xi_n)_{n\ge0}$ be an i.i.d. sequence of standard Gaussian
random variables, and let $(a_n)_{n\ge0}$ and
$(b_n)_{n\ge0}$ be real-valued sequences. Then
\[
\bigl\|\mathbf{P} \bigl[(a_n+\xi_n)_{n\ge0}\in\cdot
\bigr]- \mathbf{P} \bigl[(b_n+\xi_n)_{n\ge0}\in
\cdot \bigr]\bigr\|^2 \le \sum_{n=0}^\infty(b_n-a_n)^2.
\]
\end{lem}

\begin{pf}
Denote by $H(\mu,\nu)=\int\sqrt{d\mu\, d\nu}$ the Kakutani--Hellinger
affinity between probability measures $\mu,\nu$.
We recall that \cite{Shi96}, Section III.9,
\[
\|\mu_1\otimes\cdots\otimes\mu_n- \nu_1
\otimes\cdots\otimes\nu_n\|^2 \le 8 \Biggl[1-\prod
_{k=1}^nH(\mu_k,
\nu_k) \Biggr]. %
\]
But a direct computation shows that $H(N(a,1),N(b,1))=
\exp(-(b-a)^2/8)$.
The result now follows directly using $1-e^{-x}\le x$ and
$n\to\infty$.
\end{pf}

Fix $x,x'\in E$ and $f\in\mathrm{Lip}(E)$, and choose $\mathbf{Q}\in
\mathcal{C}(\mathbf{P}^x,\mathbf{P}^{x'})$ as in the statement of the
theorem. By assumption, we can choose
$n\ge0$ such that
\[
\mathbf{Q} \Biggl[\sum_{k=n}^\infty d
\bigl(X_k,X_k' \bigr)^2\le
\frac{\alpha^2}{4} \Biggr]\ge\frac{3\alpha}{4}. %
\]
Let $F_n = f(X_n)+\xi_n$, and define for every real-valued sequence
$\mathbf{a}=(a_n)_{n\ge0}$ the measure
$\mu_{\mathbf{a}}=\mathbf{P}[(a_n+\xi_n)_{n\ge0}\in\cdot ]$. Then we
have for every $A\in\mathcal{B}(\mathbb{R})^{\mathbb{Z}_+}$
\[
\mathbf{P}^x[F_{n,\infty}\in A]- \mathbf{P}^{x'}[F_{n,\infty}
\in A] = \mathbf{E}_{\mathbf{Q}} \bigl[ \mu_{(f(X_k))_{k\ge n}}(A)-
\mu_{(f(X_k'))_{k\ge n}}(A) \bigr]. %
\]
Therefore, we obtain by Jensen's inequality and Lemma~\ref{lem:hell}
\begin{eqnarray*}
\bigl\|\mathbf{P}^x[F_{n,\infty}\in\cdot ]- \mathbf{P}^{x'}[F_{n,\infty}
\in\cdot ]\bigr\| &\le& \mathbf{E}_{\mathbf{Q}} \Biggl[ \Biggl(\sum
_{k=n}^\infty \bigl\{f(X_k)-f
\bigl(X_k' \bigr) \bigr\}^2
\Biggr)^{1/2} \wedge2 \Biggr]
\\
&\le& \mathbf{E}_{\mathbf{Q}} \Biggl[ \Biggl(\sum
_{k=n}^\infty d \bigl(X_k,X_k'
\bigr)^2 \Biggr)^{1/2}\wedge2 \Biggr]
\\
&\le& \frac{\alpha}{2} + 2 \biggl[1-\frac{3\alpha}{4} \biggr] = 2 - \alpha,
\end{eqnarray*}
where we have used the Lipschitz property of $f$.
Applying Theorem~\ref{thmm:loc02} as indicated at the beginning
of the proof, it follows that
\[
\bigl\|\mathbf{P}^\mu[F_{n,\infty}\in\cdot ]- \mathbf{P}^\nu[F_{n,\infty}
\in\cdot ]\bigr\|\mathop{\longrightarrow}^{n\to\infty}0 \qquad\mbox{for all }\mu,\nu\in\mathcal{P}(E).
\]
In particular, if we denote by $\xi\in\mathcal{P}(\mathbb{R})$ the
standard Gaussian measure, then
\[
\bigl\|\mu P^n f^{-1} * \xi- \nu P^n
f^{-1} * \xi\bigr\| \mathop{\longrightarrow}^{n\to\infty}0 \qquad\mbox{for all }\mu,\nu\in
\mathcal{P}(E) %
\]
(here $*$ denotes convolution).
We claim that this implies
\[
\bigl|\mu P^n f - \nu P^n f\bigr| \mathop{\longrightarrow}^{n\to\infty}0\qquad
\mbox{for all }\mu,\nu\in\mathcal{P}(E)\mbox{ and } f\in\mathrm{Lip}(E).
\]
Indeed, if we assume the contrary, then there exists
for some
$f\in\mathrm{Lip}(E)$ and $\mu,\nu\in\mathcal{P}(E)$ a
subsequence $m_n\uparrow\infty$ so that
$\inf_n|\mu P^{m_n} f - \nu P^{m_n} f|>0$. As $f$ takes values
in the compact interval $[-1,1]$, we can extract a further subsequence
$k_n\uparrow\infty$ so that $\mu P^{k_n}f^{-1}\to\mu_\infty$
and $\nu P^{k_n}f^{-1}\to\nu_\infty$ in the weak convergence topology
for some
$\mu_\infty,\nu_\infty\in\mathcal{P}([-1,1])$, and clearly
$\mu_\infty\ne\nu_\infty$ by construction. On the other hand, as
$\|\mu P^n f^{-1} * \xi- \nu P^n f^{-1} * \xi\|\to0$, we must have
$\mu_\infty*\xi=\nu_\infty*\xi$. This entails a contradiction, as the
Fourier transform of $\xi$ vanishes nowhere [so convolution by
$\xi$ is injective on $\mathcal{P}(\mathbb{R})$].

We finally claim that in fact
\[
\bigl\|\mu P^n - \nu P^n\bigr\|_{\mathrm{BL}} \mathop{\longrightarrow}^{n\to
\infty}0 \qquad\mbox{for all }\mu,\nu\in\mathcal{P}(E). %
\]
Indeed, we have shown above that the signed measure
$\varrho_n = \mu P^n-\nu P^n$ converges to zero pointwise on
$\mathrm{Lip}(E)$. As any function in $U_b(E)$ can be approximated
uniformly by bounded Lipschitz functions, this implies that
$\varrho_n\to0$ in the $\sigma(\mathcal{M}(E),U_b(E))$-topology.
A result of Pachl \cite{Pac79}, Theorem~3.2, now implies that
$\|\varrho_n\|_{\mathrm{BL}}\to0$, which concludes the proof of the
theorem.
\end{pf*}

For simplicity, we have stated the assumption of Theorem~\ref{thmm:wk02}
so that an asymptotic coupling of the entire state $X_k$ of the Markov
chain is required. The reader may easily adapt the proof of Theorem~\ref{thmm:wk02} to require only asymptotic couplings of
finite-dimensional projections as in the local mixing setting (Corollary~\ref{cor:locmix}), or to deduce a variant of this result in the setting
of non-Markov processes. However, let us emphasize that even in the
setting of Theorem~\ref{thmm:wk02}, where the asymptotic coupling is at
the level of the Markov process $X_k$, the ``smoothed'' process
$F_k=f(X_k)+\xi_k$ that appears in the proof is non-Markovian. Therefore,
the \emph{local} zero--two law is essential in order to obtain weak-*
ergodicity results from the total variation theory.

The stationary counterpart to Theorem~\ref{thmm:wk02} also follows along
the same lines. However, here a small modification is needed at
the end of the proof.

\begin{thmm}[(Stationary weak-* ergodicity)]
\label{thmm:swk02}
Let $\lambda$ be a $P$-invariant probability. Suppose that
for $\lambda\otimes\lambda$-a.e. $(x,x')\in E\times E$,
\[
\exists \mathbf{Q}\in\mathcal{C} \bigl(\mathbf{P}^x,
\mathbf{P}^{x'} \bigr) \mbox{ such that}\qquad \mathbf{Q} \Biggl[\sum
_{n=0}^\infty d \bigl(X_n,X_n'
\bigr)^2<\infty \Biggr]>0. %
\]
Then the Markov chain is a.e. weak-* ergodic in the sense that
\[
\bigl\|P^n(x,\cdot)-\lambda\bigr\|_{\mathrm{BL}}\mathop{\longrightarrow}^{n\to\infty}0
\qquad\mbox{for }\lambda\mbox{-a.e. } x\in E. %
\]
\end{thmm}

\begin{pf}
Repeating the proof of Theorem~\ref{thmm:wk02} using
Theorem~\ref{thmm:sloc02} instead of Theorem~\ref{thmm:loc02} yields the
following: for every $f\in\mathrm{Lip}(E)$, we have
\[
\bigl|P^nf(x) - \lambda f\bigr| \mathop{\longrightarrow}^{n\to\infty}0\qquad \mbox{for }\lambda
\mbox{-a.e. }x. %
\]
We would like to extend this to convergence in the bounded-Lipschitz
norm. This does not follow immediately, however, as the $\lambda$-null
set of
$x\in E$ for which the convergence fails may depend on $f\in\mathrm{Lip}(E)$.

Fix $\varepsilon>0$. Let $K\subseteq E$ be a compact set such that
$\lambda(K)\ge1-\varepsilon$, and define
$\chi(x)=(1-\varepsilon^{-1}d(x,K))_+$. Then we can estimate
\begin{eqnarray*}
\bigl\|P^n(x,\cdot)-\lambda\bigr\|_{\mathrm{BL}} &\le& \sup
_{f\in\mathrm{Lip}(E)} \bigl|P^n(f\chi) (x)-\lambda(f\chi)\bigr| +
P^n(1-\chi) (x)+\lambda(1-\chi)
\\
&\le& \sup_{f\in\mathrm{Lip}(E)} \bigl|P^n(f\chi) (x)-\lambda(f\chi)\bigr| +
\bigl|P^n\chi(x)-\lambda\chi\bigr| + 2\varepsilon.
\end{eqnarray*}
By the Arzel\`a--Ascoli theorem, we can find a finite number
of functions $f_1,\ldots,f_k\in\mathrm{Lip}(E)$ such that
$\sup_{f\in\mathrm{Lip}(E)}\min_i\|f_i\mathbf{1}_K-f\mathbf{1}_K
\|_\infty\le\varepsilon$. But note that
$|f(x)-g(x)|\le2\varepsilon+
\|f\mathbf{1}_K-g\mathbf{1}_K\|_\infty$ whenever $d(x,K)\le\varepsilon$
and $f,g\in\mathrm{Lip}(E)$. Therefore,
$\sup_{f\in\mathrm{Lip}(E)}\min_i\|f_i\chi-f\chi\|_\infty\le
3\varepsilon$,
and we have
\[
\bigl\|P^n(x,\cdot)-\lambda\bigr\|_{\mathrm{BL}} \le \max
_{i=1,\ldots,k}\bigl |P^n(f_i\chi) (x)-
\lambda(f_i\chi)\bigr| + \bigl|P^n\chi(x)-\lambda\chi\bigr| + 8
\varepsilon. %
\]
As the quantity on the right-hand side depends only on a finite
number of bounded Lipschitz functions $\chi,f_1\chi,\ldots,f_k\chi$,
we certainly have
\[
\limsup_{n\to\infty} \bigl\|P^n(x,\cdot)-\lambda
\bigr\|_{\mathrm{BL}} \le8\varepsilon \qquad\mbox{for }\lambda\mbox{-a.e. }x. %
\]
But $\varepsilon>0$ was arbitrary, so the proof is complete.
\end{pf}

\begin{rem}
The tightness argument used here is in fact more elementary than the
result of Pachl \cite{Pac79} used in the proof of Theorem~\ref{thmm:wk02}. Note, however, that Theorem~\ref{thmm:wk02} does not even
require the existence of an invariant probability, so that tightness is
not guaranteed in that setting.
\end{rem}

\section{Conditional ergodicity}
\label{sec:conderg}

In the previous section, we have developed various measure-theoretic
ergodic theorems that are applicable in infinite-dimensional or
non-Markov settings. The goal of the present section is to develop a
\emph{conditional} variant of these ideas: given a stationary process
$(Z_k,Y_k)_{k\in\mathbb{Z}}$, we aim to understand when
$(Z_k)_{k\in\mathbb{Z}}$ is ergodic conditionally on
$(Y_k)_{k\in\mathbb{Z}}$. The conditional ergodic theory developed in
this section will be used in Section~\ref{sec:filter} below to prove
stability and ergodicity of nonlinear filters.

In Section~\ref{sec:cond02}, we first develop a conditional variant of
the zero--two laws of the previous section. In principle, this result
completely characterizes the conditional absolute regularity property of
$(Z_k)_{k\in\mathbb{Z}}$ given $(Y_k)_{k\in\mathbb{Z}}$. Unfortunately,
the equivalent conditions of the zero--two law are stated in terms of the
conditional distribution $\mathbf{P}[Z\in\cdot |Y]$: this
quantity is defined abstractly as a regular conditional probability, but
an explicit expression is almost never available. Therefore, in itself,
the conditional zero--two law is very difficult to use. In contrast, the
(unconditional) ergodic theory of $(Z_k,Y_k)_{k\in\mathbb{Z}}$ can typically
be studied by direct analysis of the underlying model.

The question that we aim to address is therefore that of inheritance: if
the unconditional process $(Z_k,Y_k)_{k\in\mathbb{Z}}$ is absolutely
regular, does this imply that the process is also conditionally
absolutely regular? In general, this is not the case even when the
process is Markov (cf. \cite{vH12}). However, under a suitable
\emph{nondegeneracy} requirement, we will be able to establish
inheritance of the absolute regularity property from the unconditional
to the conditional process (Section~\ref{sec:nondeg}). Using an
additional argument (Section~\ref{sec:weiz}), we will also deduce
conditional ergodicity given the one-sided process $(Y_k)_{k\ge0}$
rather than the two-sided process $(Y_k)_{k\in\mathbb{Z}}$, as will be
needed in Section~\ref{sec:filter} below.

The inheritance of the ergodicity property under conditioning was first
established in the Markov setting in \cite{vH09,TvH12}. For a Markov
process $(X_k)_{k\ge0}$, the condition of the zero--two law states that
for a.e. initial conditions $x,x'$, there exists $n\ge0$ such that
$P^n(x,\cdot)$ and $P^n(x',\cdot)$ are not mutually singular, while the
conditional zero--two law yields essentially the same condition where the
transition kernel $P(x,\cdot)=\mathbf{P}^x[X_1\in\cdot ]$ is replaced
by the conditional transition kernel $\tilde P(x,\cdot) =
\mathbf{P}^x[X_1\in\cdot |Y]$. The key idea in the proof
of inheritance was to show that the unconditional and conditional
transition kernels are equivalent $P(x,\cdot)\sim\tilde P(x,\cdot)$
a.e., which immediately yields equivalence of the conditions of the
unconditional and conditional zero--two laws. Unfortunately, such an
approach cannot work in the non-Markov setting, as here the corresponding
argument would require us to show the equivalence of
laws of the infinite future
$\mathbf{P}[Z_{k,\infty}\in\cdot |Z_{-\infty,0}]$ and
$\mathbf{P}[Z_{k,\infty}\in\cdot |Y,Z_{-\infty,0}]$.
Such an equivalence on the infinite time interval cannot hold except in
trivial cases (even in the Markov setting). We must therefore develop a
new method to establish inheritance of the conditions of the
unconditional and conditional zero--two laws that avoids the
Markov-specific arguments in~\cite{vH09,TvH12}.

Throughout this section, we adopt the same setting and notations as in
Section~\ref{sec:zerotwo}. Here, we will assume that $E=G\times F$ where
$G,F$ are Polish spaces, and we fix a stationary probability
$\mathbf{P}$ on $(\Omega,\mathcal{F})$. We denote the components of the
coordinate process as $X_n=(Z_n,Y_n)$. Thus
$(Z_n,Y_n)_{n\in\mathbb{Z}}$ is a stationary process in $G\times F$
defined on the canonical probability space $(\Omega,\mathcal{F},\mathbf{P})$.
Let
\[
\mathcal{F}^Z_{m,n} = \sigma\{Z_{m,n}\},\qquad
\mathcal{F}^Y_{m,n} = \sigma\{Y_{m,n}\},\qquad
\mathcal{F}^Z=\mathcal{F}^Z_{-\infty,\infty},\qquad
\mathcal{F}^Y=\mathcal{F}^Y_{-\infty,\infty}, %
\]
so that $\mathcal{F}_{m,n}=\mathcal{F}^Z_{m,n}\vee\mathcal{F}^Y_{m,n}$.
We also define the tail $\sigma$-field
\[
\mathcal{A}^Z = \bigcap_{n\ge0}
\mathcal{F}^Z_{n,\infty}. %
\]
For simplicity, we will write $Z=Z_{-\infty,\infty}$ and
$Y=Y_{-\infty,\infty}$. We introduce the convention that for a
sequence $z=(z_n)_{n\in\mathbb{Z}}$, we write $z_-=(z_n)_{n\le0}$.

\begin{rem}
In the remainder of the paper, all random variables that we
encounter will take values in Polish spaces. This ensures the existence
of regular conditional probabilities and the validity of the
disintegration theorem \cite{Kal02}, Chapter~6, which will be
exploited repeatedly in our proofs.
\end{rem}

\subsection{A conditional zero--two law}
\label{sec:cond02}

In this section, we establish a conditional counterpart to Corollary~\ref{cor:absreg} that characterizes the absolute regularity property of
$(Z_n)_{n\in\mathbb{Z}}$ conditionally on $(Y_n)_{n\in\mathbb{Z}}$.
While it is difficult to apply directly, this conditional zero--two law
plays a fundamental role in the study of conditional ergodicity to be
undertaken in the following sections.

In the following, we define the probability measure $\mathbf{Y}$
and fix versions of the regular conditional probabilities $\mathbf{P}_y$,
$\mathbf{P}_y^{z_-}$ as follows:
\[
\mathbf{Y}=\mathbf{P}[Y\in\cdot ],\qquad \mathbf{P}_Y = \mathbf{P} \bigl[
\cdot |\mathcal{F}^Y \bigr], \qquad\mathbf{P}_Y^{Z_-} =
\mathbf{P} \bigl[ \cdot |\mathcal{F}^Y\vee\mathcal{F}^Z_{-\infty,0}
\bigr]. %
\]
That is, $\mathbf{Y}$ denotes the law of the observed process,
$\mathbf{P}_y$ is the conditional law of the model given a fixed
observation sequence $y$, and $\mathbf{P}_y^{z-}$ is the conditional law
of the model given a fixed observation sequence $y$ and initial history
$z_-$.

We now define the probability $\mathbf{Q}$ on
$G^\mathbb{Z}\times G^\mathbb{Z}\times F^{\mathbb{Z}}$ as
\[
\int\mathbf{1}_A \bigl(z,z',y \bigr) \mathbf{Q}
\bigl(dz,dz',dy \bigr)= \int\mathbf{1}_A
\bigl(z,z',y \bigr) \mathbf{P}_y(Z\in dz)
\mathbf{P}_y \bigl(Z\in dz' \bigr)\mathbf{Y}(dy).
\]
Denote the coordinate process on
$G^\mathbb{Z}\times G^\mathbb{Z}\times F^{\mathbb{Z}}$
as $(Z_n,Z_n',Y_n)_{n\in\mathbb{Z}}$. Evidently $\mathbf{Q}$ is the
coupling of two copies of $\mathbf{P}$ such that the observations $Y$
coincide and $Z,Z'$ are conditionally independent given $Y$.

The following is the main result of this section.

\begin{thmm}[(Conditional 0--2 law)]
\label{thmm:cond02}
The following are equivalent:
\begin{enumerate}
\item For $\mathbf{Q}$-a.e. $(z,z',y)$, we have
\[
\bigl\|\mathbf{P}_y^{z_-}-\mathbf{P}_y^{z_-'}
\bigr\|_{\mathcal{F}^Z_{n,\infty}} \mathop{\longrightarrow}^{n\to\infty}0. %
\]
\item For $\mathbf{Q}$-a.e. $(z,z',y)$, we have
\[
\mathbf{P}_y^{z_-}(A)=\mathbf{P}_y^{z_-}(A)^2=
\mathbf{P}_y^{z_-'}(A) \qquad\mbox{for all }A\in
\mathcal{A}^Z. %
\]
\item For $\mathbf{Q}$-a.e. $(z,z',y)$, there exists $n\ge0$ such that
$\mathbf{P}_y^{z_-}$ and $\mathbf{P}_y^{z_-'}$ are not
mutually singular on $\mathcal{F}^Z_{n,\infty}$.
\end{enumerate}
\end{thmm}

\begin{rem}
Informally, the first condition of Theorem~\ref{thmm:cond02} states that
conditionally on the observation sequence $Y=y$, the future of the
unobserved process $Z$ after time $n$ becomes independent of its initial
history as $n\to\infty$.
Using Lemmas \ref{lem:margtv} and \ref{lem:weiz} below,
one can show that this condition
is also equivalent to the conditional absolute
regularity property
\[
\bigl\|\mathbf{P}[Z_{-\infty,0}, Z_{k,\infty}\in\cdot |Y]-
\mathbf{P}[Z_{-\infty,0}\in\cdot |Y]\otimes \mathbf{P}[Z_{k,\infty}\in
\cdot |Y]\bigr\| \mathop{\longrightarrow}^{k\to\infty}0, \qquad\mathbf{P}\mbox{-a.s.} %
\]
exactly as in Corollary~\ref{cor:absreg}. For our purposes,
however, the condition as stated in Theorem~\ref{thmm:cond02} will be most
convenient in the sequel.
\end{rem}

The proof of Theorem~\ref{thmm:cond02} is similar in spirit to that of
Theorem~\ref{thmm:sloc02}. However, care is needed in the handling
of regular conditional probabilities. We begin by establishing some
basic facts.

An elementary but important idea that will be used several times in the
sequel is the principle of repeated conditioning. This idea is trivial
in the setting of discrete random variables. Let $X_1,X_2,X_3$ be discrete
random variables under a probability $\mathbf{R}$, and define the
conditional probabilities $\mathbf{R}_{x_1}=\mathbf{R}[ \cdot |X_1=x_1]$
and $\mathbf{R}_{x_1,x_2}=\mathbf{R}[ \cdot |X_1=x_1,X_2=x_2]$. Then
\begin{eqnarray*}
\mathbf{R}_{x_1,x_2}[X_3=x_3] &=&
\frac{\mathbf{R}[X_1=x_1,X_2=x_2,X_3=x_3]}{
\mathbf{R}[X_1=x_1,X_2=x_2]}
\\
&= &\frac{\mathbf{R}[X_2=x_2,X_3=x_3|X_1=x_1]}{
\mathbf{R}[X_2=x_2|X_1=x_1]} \\
&=& \mathbf{R}_{x_1}[X_3=x_3|X_2=x_2].
\end{eqnarray*}
Thus conditioning on two variables can be achieved by conditioning first
on one variable, and then conditioning on the second variable under the
conditional distribution. The following lemma, taken from \cite{vW83}, pages~95--96,
extends this idea to the general setting of regular
conditional probabilities. As the proof is short but insightful,
we include it for completeness.

\begin{lem}
\label{lem:weiz}
Let $H_1,H_2,H_3$ be Polish spaces, let $\mathbf{R}$ be a probability
on $H_1\times H_2\times H_3$, and let
$X_i(x_1,x_2,x_3)=x_i$, $\mathcal{G}_i=\sigma\{X_i\}$. Choose
any versions of the regular conditional probabilities
$\mathbf{R}_{X_1}=\mathbf{R}[ \cdot |\mathcal{G}_1]$ and
$\mathbf{R}_{X_1,X_2}=\mathbf{R}[ \cdot |\mathcal{G}_1\vee
\mathcal{G}_2]$, and let $\mathbf{R}^1=\mathbf{R}[X_1\in\cdot ]$.
Then for $\mathbf{R}^1$-a.e. $x_1\in H_1$,
\[
\mathbf{R}_{x_1,X_2}(A) = \mathbf{R}_{x_1}[A|
\mathcal{G}_2] \qquad\mathbf{R}_{x_1}\mbox{-a.s.} \mbox{ for all }A
\in \mathcal{G}_1\vee\mathcal{G}_2\vee
\mathcal{G}_3. %
\]
\end{lem}

\begin{pf}
For given $A\in\mathcal{G}_1\vee\mathcal{G}_2\vee\mathcal{G}_3$ and
$B\in\mathcal{G}_2$, we have $\mathbf{R}$-a.s.
\[
\mathbf{R}_{X_1} \bigl[\mathbf{R}_{X_1,X_2}(A)
\mathbf{1}_B \bigr] = \mathbf{R} \bigl[\mathbf{R}[A|
\mathcal{G}_1\vee\mathcal{G}_2] \mathbf{1}_B|
\mathcal{G}_1 \bigr]= \mathbf{R}[A\cap B|\mathcal{G}_1]=
\mathbf{R}_{X_1}[A\cap B]. %
\]
By disintegration, we obtain
\[
\mathbf{R}_{x_1} \bigl[\mathbf{R}_{x_1,X_2}(A)
\mathbf{1}_B \bigr]= \mathbf{R}_{x_1}[A\cap B] \qquad\mbox{for }
\mathbf{R}^1\mbox{-a.e. } x_1. %
\]
But as $\mathcal{G}_1,\mathcal{G}_2,\mathcal{G}_3$ are
countably generated, we can ensure that for $\mathbf{R}^1$-a.e. $x_1$,
this equality holds simultaneously for all sets $A$ and $B$ in a countable
generating algebra for $\mathcal{G}_1\vee\mathcal{G}_2\vee\mathcal{G}_3$
and $\mathcal{G}_2$, respectively. By the monotone class theorem, it
follows that for $\mathbf{R}^1$-a.e. $x_1$ the equality holds for all
$A\in\mathcal{G}_1\vee\mathcal{G}_2\vee\mathcal{G}_3$ and
$B\in\mathcal{G}_2$ simultaneously, which yields the claim.
\end{pf}

Commencing with the proof of Theorem~\ref{thmm:cond02}, we begin by
obtaining a conditional counterpart of the Markov property
$\mathbf{P}^{X_n}(A)=\mathbf{E}[\mathbf{1}_A\circ\Theta^n|\mathcal{F}_{0,n}]$
that is adapted to the present setting. While the following lemma could
be deduced from Lemma~\ref{lem:weiz}, we give a direct proof along the
same lines.

\begin{lem}
\label{lem:c02-i}
For $\mathbf{P}$-a.e. $(z,y)$, we have
\[
\mathbf{P}_{\Theta^n y}^{Z_-\circ\Theta^n}(A) = \mathbf{E}^{z_-}_{y}
\bigl[\mathbf{1}_A\circ\Theta^n| \mathcal{F}^Z_{-\infty,n}
\bigr] \qquad\mathbf{P}^{z_-}_{y}\mbox{-a.s.} \mbox{ for every }A\in
\mathcal{F}^Z. %
\]
\end{lem}

\begin{pf}
Fix $A\in\mathcal{F}^Z$, $B\in\mathcal{F}^Z_{-\infty,n}$, and
$C\in\mathcal{F}^Y\vee\mathcal{F}^Z_{-\infty,0}$. Then
\begin{eqnarray*}
\mathbf{E} \bigl[ \mathbf{E}^{Z_-}_Y \bigl[
\mathbf{P}^{Z_-\circ\Theta^n}_{\Theta^nY}(A) \mathbf{1}_B \bigr]
\mathbf{1}_C \bigr] &=& \mathbf{E} \bigl[ \mathbf{P}^{Z_-\circ\Theta^n}_{\Theta^nY}(A)
\mathbf{1}_{B\cap C} \bigr]
\\
&=& \mathbf{E} \bigl[ \mathbf{P}^{Z_-}_{Y}(A) \bigl\{
\mathbf{1}_{B\cap C} \circ\Theta^{-n} \bigr\} \bigr]
\\
&=& \mathbf{E} \bigl[\mathbf{1}_A \bigl\{\mathbf{1}_{B\cap C}
\circ\Theta^{-n} \bigr\} \bigr]
\\
&=& \mathbf{E} \bigl[ \bigl\{\mathbf{1}_A\circ\Theta^n
\bigr\} \mathbf{1}_{B\cap C} \bigr]
\\
&=& \mathbf{E} \bigl[\mathbf{E}_Y^{Z_-} \bigl[ \bigl\{
\mathbf{1}_A\circ\Theta^n \bigr\} \mathbf{1}_B
\bigr] \mathbf{1}_{C} \bigr],
\end{eqnarray*}
where we have used the stationarity of $\mathbf{P}$ and the definition
of the regular conditional probability $\mathbf{P}_y^{z_-}$.
As this holds for all
$C\in\mathcal{F}^Y\vee\mathcal{F}^Z_{-\infty,0}$, we have
\[
\mathbf{E}^{z_-}_y \bigl[ \mathbf{P}^{Z_-\circ\Theta^n}_{\Theta^ny}(A)
\mathbf{1}_B \bigr] = \mathbf{E}_y^{z_-}
\bigl[ \bigl\{\mathbf{1}_A\circ\Theta^n \bigr\}
\mathbf{1}_B \bigr] \qquad\mbox{for }\mathbf{P}\mbox{-a.e. }(z,y).
\]
We conclude by a monotone class argument
as in the proof of Lemma~\ref{lem:weiz}.
\end{pf}

In the following, we will require some basic properties of the coupled
measure~$\mathbf{Q}$. First, notice that as $\mathbf{Y}$ is stationary
and as
\[
\mathbf{P}_{Y\circ\Theta}(A)=\mathbf{P} \bigl[A|\mathcal{F}^Y
\bigr] \circ\Theta =\mathbf{E} \bigl[\mathbf{1}_A\circ\Theta|
\mathcal{F}^Y \bigr]= \mathbf{E}_Y(\mathbf{1}_A
\circ\Theta), \qquad\mathbf{P}\mbox{-a.s.}, %
\]
it is easily seen that $\mathbf{Q}$ is a stationary measure.
Moreover, defining
\[
\mathbf{Q}_y^{z_-,z_-'} = \mathbf{P}_y^{z_-}(Z
\in \cdot )\otimes \mathbf{P}_y^{z_-'}(Z\in \cdot )\otimes
\delta_y \qquad\mbox{on } G^\mathbb{Z}\times G^\mathbb{Z}
\times F^{\mathbb{Z}}, %
\]
it follows directly that $\mathbf{Q}_Y^{Z_-,Z_-'}$ is a version of the
regular conditional probability $\mathbf{Q}[ \cdot |\mathcal{F}^Y
\vee\mathcal{F}^Z_{-\infty,0}\vee\mathcal{F}^{Z'}_{-\infty,0}]$,
where $\mathcal{F}^{Z'}_{m,n}=\sigma\{Z_{m,n}'\}$.

To proceed, we define the maps
\[
\beta_n \bigl(z_-,z_-',y \bigr) = \bigl\|
\mathbf{P}^{z_-}_y -\mathbf{P}^{z_-'}_y
\bigr\|_{\mathcal{F}^Z_{n,\infty}}, \qquad\beta \bigl(z_-,z_-',y \bigr)= \bigl\|
\mathbf{P}^{z_-}_y -\mathbf{P}^{z_-'}_y
\bigr\|_{\mathcal{A}^Z}. %
\]
As $\mathcal{F}^Z_{n,\infty}$ is countably generated, the maps
$\beta_n$ are measurable. Moreover, as $\beta_n\downarrow\beta$
pointwise as $n\to\infty$, the map $\beta$ is measurable also.

The following lemma establishes the invariance of $\beta$.

\begin{lem}
\label{lem:c02-ii}
For $\mathbf{Q}$-a.e. $(z_-,z_-',y)$, we have
\[
\mathbf{Q}_y^{z_-,z_-'} \bigl[ \beta \bigl(z_-,z_-',y
\bigr)=\beta \bigl(Z_-\circ\Theta^n, Z_-'\circ
\Theta^n,\Theta^ny \bigr)\mbox{ for all }n\ge0 \bigr]=1.
\]
\end{lem}

\begin{pf}
Define for simplicity $\mathcal{G}_n =
\mathcal{F}^Y\vee\mathcal{F}^Z_{-\infty,n}\vee
\mathcal{F}^{Z'}_{-\infty,n}$.
As $Z,Z'$ are conditionally independent given $Y$ under $\mathbf{Q}$,
we have
\[
\mathbf{P}^{Z_-}_Y(Z\in \cdot ) = \mathbf{Q}[Z\in\cdot |
\mathcal{G}_0]\quad \mbox{and}\quad \mathbf{P}^{Z_-'}_Y(Z\in
\cdot )= \mathbf{Q} \bigl[Z'\in\cdot |\mathcal{G}_0
\bigr],\qquad \mathbf{Q}\mbox{-a.s.} %
\]
Using Jensen's inequality and stationarity of $\mathbf{Q}$ yields
\begin{eqnarray*}
\beta_k \bigl(Z_-,Z_-',Y \bigr) &=& \bigl\|
\mathbf{Q}[Z_{k,\infty}\in\cdot |\mathcal{G}_0]- \mathbf{Q}
\bigl[Z'_{k,\infty}\in\cdot |\mathcal{G}_0 \bigr]
\bigr\|
\\
&\le& \mathbf{E}_{\mathbf{Q}} \bigl[\bigl \|\mathbf{Q}[Z_{k,\infty}\in\cdot |
\mathcal{G}_1]- \mathbf{Q} \bigl[Z'_{k,\infty}\in
\cdot |\mathcal{G}_1 \bigr] \bigr\| |\mathcal{G}_0 \bigr]
\\
&=& \mathbf{E}_{\mathbf{Q}} \bigl[ \bigl\|\mathbf{Q}[Z_{k-1,\infty}\in\cdot |
\mathcal{G}_0]- \mathbf{Q} \bigl[Z'_{k-1,\infty}\in
\cdot |\mathcal{G}_0 \bigr] \bigr\|\circ\Theta |\mathcal{G}_0
\bigr]
\\
&=& \mathbf{E}_{\mathbf{Q}} \bigl[ \beta_{k-1} \bigl(Z_-,Z_-',Y
\bigr)\circ\Theta |\mathcal{G}_0 \bigr].
\end{eqnarray*}
Letting $k\to\infty$ and using stationarity, we find that
\[
\beta \bigl(Z_-,Z_-',Y \bigr)\circ\Theta^{n} \le
\mathbf{E}_{\mathbf{Q}} \bigl[\beta \bigl(Z_-,Z_-',Y \bigr)\circ
\Theta^{n+1}|\mathcal{G}_n \bigr], %
\]
so $M_n=\beta(Z_-,Z_-',Y)\circ\Theta^n$ is a bounded submartingale under
$\mathbf{Q}$. In particular, by stationarity and the martingale
convergence theorem,
\begin{eqnarray*}
&&\mathbf{E}_{\mathbf{Q}} \bigl[\bigl|\beta \bigl(Z_-,Z_-',Y \bigr)-
\beta \bigl(Z_-,Z_-',Y \bigr)\circ\Theta^n\bigr| \bigr]\\
&&\qquad =
\mathbf{E}_{\mathbf{Q}} \bigl[\bigl|\beta \bigl(Z_-,Z_-',Y \bigr)
\circ \Theta^k- \beta \bigl(Z_-,Z_-',Y \bigr)\circ
\Theta^{n+k}\bigr| \bigr] \mathop{\longrightarrow}^{k\to\infty}0.
\end{eqnarray*}
It follows
that $\beta(Z_-,Z_-',Y)=\beta(Z_-,Z_-',Y)\circ\Theta^n$ for all $n\ge0$,
$\mathbf{Q}$-a.s., and the result follows readily by disintegration.
\end{pf}

We are now ready to complete the proof of Theorem~\ref{thmm:cond02}.

\begin{pf*}{Proof of Theorem~\ref{thmm:cond02}}
The proofs of $2\Rightarrow1\Rightarrow3$ are identical to the
corresponding proofs in Theorem~\ref{thmm:loc02}.
It remains to prove $3\Rightarrow2$.

Let us assume that condition 3 holds.
Then, using Lemmas \ref{lem:c02-i} and
\ref{lem:c02-ii}, we can find a measurable subset
$H\subset G^\mathbb{Z}\times G^\mathbb{Z}\times F^\mathbb{Z}$
of full probability $\mathbf{Q}(H)=1$ such that
the following hold for every $(z,z',y)\in H$:
\begin{longlist}[(a)]
\item[(a)]
$\mathbf{P}_y^{z_-}$ and $\mathbf{P}_y^{z_-'}$ are not
mutually singular on $\mathcal{F}^Z_{n,\infty}$ for some $n\ge0$.
\item[(b)]
$\mathbf{P}_{\Theta^n y}^{Z_-\circ\Theta^n}(A) =
\mathbf{E}^{z_-}_{y}[\mathbf{1}_A\circ\Theta^n|
\mathcal{F}^Z_{-\infty,n}]$,
$\mathbf{P}^{z_-}_{y}$-a.s. for all $A\in\mathcal{F}^Z$, $n\ge0$.
\item[(c)]
$\mathbf{P}_{\Theta^n y}^{Z_-\circ\Theta^n}(A) =
\mathbf{E}^{z_-'}_{y}[\mathbf{1}_A\circ\Theta^n|
\mathcal{F}^Z_{-\infty,n}]$,
$\mathbf{P}^{z_-'}_{y}$-a.s. for all $A\in\mathcal{F}^Z$, $n\ge0$.
\item[(d)]
$\mathbf{Q}_y^{z_-,z_-'}[
\beta(z_-,z_-',y)=\beta(Z_-\circ\Theta^n,
Z_-'\circ\Theta^n,\Theta^ny)\mbox{ for all }n\ge0]=1$.
\end{longlist}
Now suppose that condition 2 does not hold. Then we can choose
a path
$(z,z',y)\in H$ and a tail event $A\in\mathcal{A}^Z$ such that
\[
\mbox{either } 0<\mathbf{P}^{z_-}_y(A)<1\quad\mbox{or}\quad
\mathbf{P}^{z_-}_y(A)\ne\mathbf{P}^{z_-'}_y(A).
\]
Define $f=\mathbf{1}_A-\mathbf{1}_{A^c}$ and fix $\alpha>0$.
By the martingale convergence theorem,
\begin{eqnarray*}
&&\mathbf{Q}_y^{z_-,z_-'} \bigl[ \bigl| \mathbf{E}_{\Theta^n y}^{Z_-\circ\Theta^n}
\bigl(f\circ\Theta^{-n} \bigr)- \mathbf{E}_{\Theta^n y}^{Z_-'\circ\Theta^n}
\bigl(f\circ\Theta^{-n} \bigr)\bigr |>2-\alpha \bigr] \\
&&\qquad\mathop{\longrightarrow}^{n\to
\infty} {}
\mathbf{P}_y^{z_-}(A) \mathbf{P}_y^{z_-'}
\bigl(A^c \bigr)+ \mathbf{P}_y^{z_-}
\bigl(A^c \bigr) \mathbf{P}_y^{z_-'}(A)>0.
\end{eqnarray*}
Note that as $|f|\le1$ and $f\circ\Theta^{-n}$ is
$\mathcal{A}^Z$-measurable, we have
\begin{eqnarray*}
\bigl| \mathbf{E}_{\Theta^n y}^{Z_-\circ\Theta^n} \bigl(f\circ\Theta^{-n}
\bigr)- \mathbf{E}_{\Theta^n y}^{Z_-'\circ\Theta^n} \bigl(f\circ
\Theta^{-n} \bigr)\bigr | &\le& \beta \bigl(Z_-\circ\Theta^n,Z_-'
\circ \Theta^n,\Theta^n y \bigr)
\\
&=&\beta \bigl(z_-,z_-',y \bigr), \qquad\mathbf{Q}^{z_-,z_-'}_y
\mbox{-a.s.}
\end{eqnarray*}
It follows that $\beta(z_-,z_-',y)>2-\alpha$, and we therefore have
$\beta(z_-,z_-',y)=2$ as $\alpha>0$ was arbitrary.
But by construction there exists $n\ge0$ such that
$\beta(z_-,z_-',y)\le\beta_n(z_-,z_-',y)<2$, which entails the desired
contradiction.
\end{pf*}

\subsection{Nondegeneracy and inheritance}
\label{sec:nondeg}

While Theorem~\ref{thmm:cond02} in principle characterizes completely the
conditional absolute regularity property, this result is difficult to
apply directly as an explicit description of $\mathbf{P}_y^{z_-}$
(beyond its existence as a regular conditional probability) is typically
not available. On the other hand, many methods are available to
establish the absolute regularity property of the unconditional model
$(Z,Y)$. We will presently develop a technique that allows us to deduce
the conditional absolute regularity property from absolute regularity of
the unconditional model.

The essential assumption that will be needed for inheritance of the
absolute regularity property is nondegeneracy of the observations,
Definition~\ref{defn:nondeg}, which states that the conditional
distribution of any finite number of observations can be made independent
of the unobserved component by a change of measure. The precise form of
the following definition is adapted to what is needed in the proofs in the
present section, and its interpretation may not appear entirely obvious at
first sight. However, we will see in Section~\ref{sec:filter} that the
nondegeneracy assumption takes a very natural form in the Markov setting
and is typically easy to verify from the model description.

\begin{defn}[(Nondegeneracy)]
\label{defn:nondeg}
The process $(Z_k,Y_k)_{k\in\mathbb{Z}}$
is said to be \emph{nondegenerate} if for every $-\infty<m\le n<\infty$
we have
\[
\mathbf{P}[Y_{m,n}\in\cdot | \mathcal{F}_{-\infty,m-1}\vee
\mathcal{F}_{n+1,\infty}] \sim \mathbf{P} \bigl[Y_{m,n}\in\cdot |
\mathcal{F}_{-\infty,m-1}^Y\vee\mathcal{F}_{n+1,\infty}^Y
\bigr],\qquad\mathbf{P}\mbox{-a.s.} %
\]
\end{defn}

We now state the main result of this section (recall that the definition
of absolute regularity was given as Definition~\ref{def:absreg} above).

\begin{thmm}[(Inheritance of absolute regularity)]
\label{thmm:inherit}
Suppose that the stationary process $(Z_k,Y_k)_{k\in\mathbb{Z}}$ is
absolutely regular and nondegenerate. Then any (hence all) of the
conditions of Theorem~\ref{thmm:cond02} hold true.
\end{thmm}

To gain some intuition for the role played by the nondegeneracy property,
let us briefly outline the idea behind the proof. We will use
nondegeneracy to remove a finite number of observations from our
conditional distributions:
\[
\mathbf{P}^{Z_-}_Y(Z_{n,\infty}\in\cdot ) =
\mathbf{P}[Z_{n,\infty}\in\cdot |Y,Z_{-\infty,0}]\sim
\mathbf{P}[Z_{n,\infty}\in\cdot |Y_{n,\infty},Y_{-\infty,0},Z_{-\infty,0}].
\]
The right-hand side depends only on
$(Z_{-\infty,0},Y_{-\infty,0})$, $(Z_{n,\infty},Y_{n,\infty})$ which
are nearly independent for large $n$ by the absolute regularity property
(Definition~\ref{def:absreg}). Therefore, for large $n$, the right-hand
side becomes nearly independent of the initial history $Z_{-\infty,0}$ in
total variation. The above equivalence of conditional distributions
then ensures the nonsingularity of $\mathbf{P}^{z_-}_y$
given two initial histories $z_-$, which yields the third condition of
Theorem~\ref{thmm:cond02}.

To implement this approach, we need three lemmas. First, in order to use
the nondegeneracy property in the manner outlined above, we must be able
to exchange the roles of conditioning and conditioned variables. To this
end, we use an elementary property that is trivial in the discrete
setting: if $X_1,X_2,X_3$ are discrete random variables under a
probability $R$, then
\[
\frac{\mathbf{R}[X_1=x_1|X_2=x_2,X_3=x_3]} {
\mathbf{R}[X_1=x_1|X_3=x_3]} = \frac{\mathbf{R}[X_2=x_2|X_1=x_1,X_3=x_3]} {
\mathbf{R}[X_2=x_2|X_3=x_3]} %
\]
whenever one side makes sense. Thus
$\mathbf{R}[X_1\in\cdot |X_2,X_3]\sim\mathbf{R}[X_1\in\cdot |X_3]$
if and only if
$\mathbf{R}[X_2\in\cdot |X_1,X_3]\sim\mathbf{R}[X_2\in\cdot |X_3]$.
The following lemma \cite{vH09} extends this idea to the
general setting of regular conditional probabilities.

\begin{lem}
\label{lem:3.6}
Let $H_1,H_2,H_3$ be Polish spaces, and let $\mathbf{R}$ be a probability
on $H_1\times H_2\times H_3$. Define
$X_i(x_1,x_2,x_3)=x_i$ and $\mathcal{G}_i=\sigma\{X_i\}$. Then
\[
\mathbf{R}[X_1\in\cdot |\mathcal{G}_2\vee
\mathcal{G}_3] \sim \mathbf{R}[X_1\in\cdot |
\mathcal{G}_3],\qquad \mathbf{R}\mbox{-a.s.} %
\]
if and only if
\[
\mathbf{R}[X_2\in\cdot |\mathcal{G}_1\vee
\mathcal{G}_3] \sim \mathbf{R}[X_2\in\cdot |
\mathcal{G}_3],\qquad\mathbf{R}\mbox{-a.s.} %
\]
\end{lem}

\begin{pf}
This follows from \cite{vH09}, Lemma~3.6, and the existence of
a measurable version of the Radon--Nikodym density
between kernels \cite{DM82}, Theorem~V.58.
\end{pf}

The second lemma states that conditioning on less information preserves
existence of a conditional density. This follows easily from the tower
property of conditional expectations, but we formulate it precisely for
concreteness.

\begin{lem}
\label{lem:tower}
Let $X_1,X_2,X_3$ be random variables taking values in Polish spaces
$H_1,H_2,H_3$, respectively, and define the $\sigma$-fields
$\mathcal{G}_i=\sigma\{X_i\}$.
Moreover, let $K\dvtx H_2\times\mathcal{B}(H_1)\to[0,1]$ be a transition
kernel, and
suppose that
\[
\mathbf{P}[X_1\in\cdot |\mathcal{G}_2\vee
\mathcal{G}_3] \sim K(X_2, \cdot ),\qquad \mathbf{P}\mbox{-a.s.}
\]
Then we also have
\[
\mathbf{P}[X_1\in\cdot |\mathcal{G}_2] \sim
K(X_2, \cdot ),\qquad\mathbf{P}\mbox{-a.s.} %
\]
\end{lem}

\begin{pf}
Using existence of a measurable version of the Radon--Nikodym density
between kernels \cite{DM82}, Theorem~V.58, the assumption implies that
there is a measurable function
$h\dvtx H_1\times H_2\times H_3\to\mbox{}]0,\infty[\mbox{}$ such that
\[
\mathbf{P}[X_1\in A|\mathcal{G}_2\vee
\mathcal{G}_3] = \int\mathbf{1}_A(x)
h(x,X_2,X_3) K(X_2,dx),\qquad \mathbf{P}
\mbox{-a.s.} %
\]
for every $A\in\mathcal{B}(H_1)$. By the tower property
\[
\mathbf{P}[X_1\in A|\mathcal{G}_2] = \int
\mathbf{1}_A(x) \int h \bigl(x,X_2,x'
\bigr) \mathbf{P}_{X_2} \bigl(dx' \bigr)
K(X_2,dx),\qquad \mathbf{P}\mbox{-a.s.}, %
\]
where we fix a version of the conditional probability
$\mathbf{P}_{X_2}=\mathbf{P}[X_3\in\cdot |\mathcal{G}_2]$.
As $\mathcal{B}(H_1)$ is countably generated, the $\mathbf{P}$-exceptional
set can be chosen independent of $A$ by a monotone class argument.
This yields the claim.
\end{pf}

The third lemma will enable us to bound the total variation distance
between conditional distributions in terms of the total variation distance
between the underlying unconditional distributions.

\begin{lem}
\label{lem:cptv}
Let $H_1,H_2,H_3$ be Polish spaces, and let $\mathbf{R}$ be a probability
on $H_1\times H_2\times H_3$. Define
$X_i(x_1,x_2,x_3)=x_i$ and $\mathcal{G}_i=\sigma\{X_i\}$. Then
\begin{eqnarray*}
&&\mathbf{E}_{\mathbf{R}} \bigl[\bigl\|\mathbf{R}[X_1\in\cdot |
\mathcal{G}_2\vee\mathcal{G}_3] -\mathbf{R}[X_1
\in\cdot |\mathcal{G}_3]\bigr\| \bigr]
\\
&&\qquad\le 2 \mathbf{E}_{\mathbf{R}}\bigl[\bigl\|\mathbf{R}[X_1,X_3
\in\cdot |\mathcal{G}_2]- \mathbf{R}[X_1,X_3
\in\cdot ]\bigr\|\bigr].
\end{eqnarray*}
\end{lem}

\begin{pf}
Fix versions of the regular conditional probabilities
$\mathbf{R}_{X_2,X_3}=\mathbf{R}[X_1\in\cdot |\mathcal{G}_2\vee\mathcal{G}_3]$,
$\mathbf{R}_{X_3}=\mathbf{R}[X_1\in\cdot |\mathcal{G}_3]$, and
$\mathbf{R}^{X_2}=\mathbf{R}[X_3\in\cdot |\mathcal{G}_2]$. Then
\[
\mathbf{E}_{\mathbf{R}} \bigl[ \bigl\|\mathbf{R}[X_1\in\cdot |
\mathcal{G}_2\vee\mathcal{G}_3] -\mathbf{R}[X_1
\in\cdot |\mathcal{G}_3]\bigr\| |\mathcal{G}_2 \bigr] = \int
\| \mathbf{R}_{X_2,x_3}- \mathbf{R}_{x_3}\| \mathbf{R}^{X_2}(dx_3).
\]
Define the kernels
\begin{eqnarray*}
\tilde{\mathbf R}^{x_2}(dx_1,dx_3) &=&
\mathbf{R}_{x_2,x_3}(dx_1) \mathbf{R}^{x_2}(dx_3),
\\
\tilde{\mathbf R}^{x_2}_0(dx_1,dx_3)
&=& \mathbf{R}_{x_3}(dx_1) \mathbf{R}^{x_2}(dx_3),
\\
\tilde{\mathbf R}(dx_1,dx_3) &=& \mathbf{R}_{x_3}(dx_1)
\mathbf{R}(X_3\in dx_3).
\end{eqnarray*}
By Lemma~\ref{lem:margtv}, we have
\[
\int\|\mathbf{R}_{x_2,x_3}- \mathbf{R}_{x_3}\|
\mathbf{R}^{x_2}(dx_3) =\bigl \|\tilde{\mathbf R}^{x_2}-
\tilde{\mathbf R}^{x_2}_0\bigr\| \le \bigl\|\tilde{\mathbf R}^{x_2}-\tilde{\mathbf R}\bigr\| + \bigl\|\tilde{\mathbf R}^{x_2}_0-
\tilde{\mathbf R}\bigr\|. %
\]
But $\tilde{\mathbf R}=
\mathbf{R}[X_1,X_3\in\cdot ]$ and $\tilde{\mathbf R}^{X_2}=
\mathbf{R}[X_1,X_3\in\cdot |\mathcal{G}_2]$ by disintegration, so
\begin{eqnarray*}
& &\mathbf{E}_{\mathbf{R}} \bigl[\bigl \|\mathbf{R}[X_1\in\cdot |
\mathcal{G}_2\vee\mathcal{G}_3] -\mathbf{R}[X_1
\in\cdot |\mathcal{G}_3]\bigr\| |\mathcal{G}_2 \bigr]
\\
&&\qquad\le \bigl\|\mathbf{R}[X_1,X_3\in\cdot |
\mathcal{G}_2]- \mathbf{R}[X_1,X_3\in\cdot ]
\bigr\| + \bigl\|\mathbf{R}[X_3\in\cdot |\mathcal{G}_2]-
\mathbf{R}[X_3\in\cdot ]\bigr\|
\\
&&\qquad\le2 \bigl\|\mathbf{R}[X_1,X_3\in\cdot |
\mathcal{G}_2]- \mathbf{R}[X_1,X_3\in\cdot ]
\bigr\|,
\end{eqnarray*}
and the proof is completed by taking the expectation.
\end{pf}

We can now complete the proof of Theorem~\ref{thmm:inherit}.

\begin{pf*}{Proof of Theorem~\ref{thmm:inherit}}
The nondegeneracy assumption states
\[
\mathbf{P}[Y_{1,n-1}\in\cdot | \mathcal{F}_{-\infty,0}\vee
\mathcal{F}_{n,\infty}]\sim \mathbf{P} \bigl[Y_{1,n-1}\in\cdot |
\mathcal{F}_{-\infty,0}^Y\vee\mathcal{F}_{n,\infty}^Y
\bigr],\qquad \mathbf{P}\mbox{-a.s.} %
\]
Therefore, by Lemma~\ref{lem:tower}, we obtain
\[
\mathbf{P} \bigl[Y_{1,n-1}\in\cdot | \mathcal{F}_{-\infty,0}\vee
\mathcal{F}_{n,\infty}^Y \bigr]\sim \mathbf{P}
\bigl[Y_{1,n-1}\in\cdot | \mathcal{F}_{-\infty,0}^Y\vee
\mathcal{F}_{n,\infty}^Y \bigr],\qquad \mathbf{P}\mbox{-a.s.}
\]
It follows that
\[
\mathbf{P}[Y_{1,n-1}\in\cdot | \mathcal{F}_{-\infty,0}\vee
\mathcal{F}_{n,\infty}]\sim \mathbf{P} \bigl[Y_{1,n-1}\in\cdot |
\mathcal{F}_{-\infty,0}\vee\mathcal{F}_{n,\infty}^Y \bigr],\qquad
\mathbf{P}\mbox{-a.s.}, %
\]
which yields using Lemma~\ref{lem:3.6}
\[
\mathbf{P} \bigl[Z_{n,\infty}\in\cdot |\mathcal{F}^Y\vee
\mathcal{F}_{-\infty,0}^Z \bigr]\sim \mathbf{P}
\bigl[Z_{n,\infty}\in\cdot | \mathcal{F}_{-\infty,0}\vee
\mathcal{F}_{n,\infty}^Y \bigr],\qquad \mathbf{P}\mbox{-a.s.}
\]
Therefore, if we choose any version of $\mathbf{P}_{Y,n}^{Z_-}=
\mathbf{P}[ \cdot |\mathcal{F}_{-\infty,0}\vee
\mathcal{F}^Y_{n,\infty}]$, then
\[
\mathbf{P}_y^{z_-}|_{\mathcal{F}^Z_{n,\infty}} \sim
\mathbf{P}_{y,n}^{z_-}|_{\mathcal{F}^Z_{n,\infty}} \qquad\forall n\ge1 \mbox{ for
}\mathbf{P}\mbox{-a.e. }(z,y) %
\]
(note that we define $\mathbf{P}_{y,n}^{z-}$ as a function of the
entire path $y$ for simplicity of notation; by construction,
$\mathbf{P}_{y,n}^{z-}$ depends on $y_{-\infty,0},y_{n,\infty},z_-$ only).
By condition 3 of Theorem~\ref{thmm:cond02}, to complete
the proof it suffices to show that
\[
\inf_n\bigl\|\mathbf{P}_{y,n}^{z_-}-
\mathbf{P}_{y,n}^{z_-'}\bigr\|_{
\mathcal{F}^Z_{n,\infty}}<2 \qquad\mbox{for }
\mathbf{Q}\mbox{-a.e. } \bigl(z,z',y \bigr). %
\]
We now proceed to establish this property.

By the triangle inequality and as $(Z,Y)$ and $(Z',Y)$
have the same law,
\[
\mathbf{E}_{\mathbf{Q}} \bigl[ \bigl\|\mathbf{P}_{Y,n}^{Z_-}-
\mathbf{P}_{Y,n}^{Z_-'}\bigr \|_{
\mathcal{F}^Z_{n,\infty}} \bigr] \le 2
\mathbf{E}_{\mathbf{Q}} \bigl[ \bigl\| \mathbf{P} \bigl[ \cdot |
\mathcal{F}_{-\infty,0} \vee \mathcal{F}^Y_{n,\infty} \bigr]-
\mathbf{P} \bigl[ \cdot | \mathcal{F}^Y_{n,\infty} \bigr]
\bigr\|_{\mathcal{F}^Z_{n,\infty}} \bigr]. %
\]
Therefore, by Lemma~\ref{lem:cptv}, we obtain
\[
\mathbf{E}_{\mathbf{Q}} \bigl[\bigl \|\mathbf{P}_{Y,n}^{Z_-}-
\mathbf{P}_{Y,n}^{Z_-'}\bigr \|_{
\mathcal{F}^Z_{n,\infty}} \bigr] \le 4
\mathbf{E}_{\mathbf{Q}} \bigl[ \bigl\| \mathbf{P}[ \cdot |\mathcal{F}_{-\infty,0}]-
\mathbf{P} \bigr\|_{\mathcal{F}_{n,\infty}} \bigr]. %
\]
As $(Z,Y)$ is absolutely regular, Corollary~\ref{cor:absreg} gives
\[
\mathbf{E}_{\mathbf{Q}} \Bigl[\inf_{n} \bigl\|
\mathbf{P}_{Y,n}^{Z_-}-\mathbf{P}_{Y,n}^{Z_-'}
\bigr\|_{
\mathcal{F}^Z_{n,\infty}} \Bigr] \le 4\inf_n \mathbf{E} \bigl[ \bigl\|
\mathbf{P}[ \cdot |\mathcal{F}_{-\infty,0}] -\mathbf{P}\bigr\|_{\mathcal{F}_{n,\infty}}
\bigr] = 0. %
\]
Thus, the requisite property is established.
\end{pf*}

\subsection{One-sided observations}
\label{sec:weiz}

Theorem~\ref{thmm:cond02} establishes conditional ergodicity of $Z$ given
the entire observation $\sigma$-field $\mathcal{F}^Y$. This allows us to
control the behavior of the conditional distributions
$\mathbf{P}[Z_{n,\infty}\in\cdot |\mathcal{F}^Y]$ as $n\to\infty$. In
contrast, the ergodic theory of nonlinear filters (Section~\ref{sec:filter}) is concerned with the ``causal'' setting where one
considers the conditional distributions
$\mathbf{P}[X_{n,\infty}\in\cdot |\mathcal{F}^Y_{0,n}]$ as
$n\to\infty$. The latter requires a one-sided version of our results
where we only condition on $\mathcal{F}^Y_+=\mathcal{F}^Y_{0,\infty}$.
Unfortunately, two-sided conditioning was essential to obtain a
conditional zero--two law: if we had replaced $\mathcal{F}^Y$ by
$\mathcal{F}^Y_+$ in Section~\ref{sec:cond02}, for example, then the
coupled measure $\mathbf{Q}$ would be nonstationary and the key Lemmas
\ref{lem:c02-i} and \ref{lem:c02-ii} would fail.

We must therefore develop an additional technique to deduce one-sided
results from their two-sided counterparts. To this end, we prove the
following result, which will suffice for our purposes (recall that
absolute regularity and nondegeneracy were defined in Definitions
\ref{def:absreg} and \ref{defn:nondeg}, resp.).

\begin{prop}[(One-sided conditioning)]
\label{prop:oneside}
Suppose that the stationary process $(Z_k,Y_k)_{k\in\mathbb{Z}}$
is absolutely regular and nondegenerate. Then
\[
\mathbf{P} \bigl[Z_{0,\infty}\in\cdot |\mathcal{F}^Y_+ \bigr]
\sim \mathbf{P} \bigl[Z_{0,\infty}\in\cdot |\mathcal{F}^Y \bigr],\qquad
\mathbf{P}\mbox{-a.s.} %
\]
\end{prop}

Before we prove this result, let us use it to establish the key
$\sigma$-field identity that will be needed in the ergodic theory of
nonlinear filters (Section~\ref{sec:filter}).

\begin{cor}
\label{cor:exchg}
Suppose that the stationary process $(Z_k,Y_k)_{k\in\mathbb{Z}}$
is absolutely regular and nondegenerate. Then the following holds:
\[
\bigcap_{n\ge0} \mathcal{F}^Y_+\vee
\mathcal{F}^Z_{n,\infty} = \mathcal{F}^Y_+ \qquad\mathop{
\mathrm{mod}}\mathbf{P}. %
\]
\end{cor}

\begin{pf}
We begin by noting that for $\mathbf{Y}$-a.e. $y$
\[
\int\mathbf{P}_{y}^{z_-}(A) \mathbf{P}_{y}(dz) =
\mathbf{P}_{y}(A) \qquad\mbox{for all }A\in\mathcal{F}^Z.
\]
Indeed, as
$\mathbf{E}[\mathbf{P}[A|\mathcal{F}^Y\vee\mathcal{F}^Z_{-\infty,0}]|
\mathcal{F}^Y]=\mathbf{P}[A|\mathcal{F}^Y]$,
the equality holds $\mathbf{Y}$-a.e. for all $A$ in a countable
generating algebra for $\mathcal{F}^Z$, and the claim follows by the
monotone class theorem. Using Jensen's inequality,
Theorems \ref{thmm:inherit} and \ref{thmm:cond02},
\[
\bigl\|\mathbf{P}_y^{z_-}-\mathbf{P}_y
\bigr\|_{\mathcal{F}^Z_{n,\infty}} \le\int \bigl\|\mathbf{P}_y^{z_-}-
\mathbf{P}_y^{z_-'}\bigr\|_{\mathcal{F}^Z_{n,\infty}} \mathbf{P}_y
\bigl(dz' \bigr) \mathop{\longrightarrow}^{n\to\infty}0 %
\]
for $\mathbf{P}$-a.e. $(z,y)$. As $\|\mathbf{P}_y^{z_-}-\mathbf{P}_y\|
_{\mathcal{F}^Z_{n,\infty}}
\to\|\mathbf{P}_y^{z_-}-\mathbf{P}_y\|_{\mathcal{A}^Z}$ as
$n\to\infty$, applying again Theorems \ref{thmm:inherit} and \ref{thmm:cond02}
shows that $\mathcal{A}^Z$ is $\mathbf{P}_y$-trivial for
$\mathbf{Y}$-a.e. $y$.

Fix a version of the regular conditional probability
$\mathbf{P}_{Y_+}=\mathbf{P}[ \cdot |\mathcal{F}^Y_+]$, where
$Y_+=(Y_k)_{k\ge0}$. By Proposition~\ref{prop:oneside},
$\mathbf{P}_y$ and $\mathbf{P}_{y_+}$ are equivalent on
$\mathcal{F}^Z_{0,\infty}$ for $\mathbf{Y}$-a.e. $y$.
It follows that $\mathcal{A}^Z$ is also $\mathbf{P}_{y_+}$-trivial for
$\mathbf{Y}$-a.e. $y$. In particular,
\[
\mathbf{P}_{y_+} \bigl[A|\mathcal{F}^Z_{n,\infty}
\bigr] \mathop{\longrightarrow}^{n\to\infty}\mathbf{P}_{y_+}[A],
\qquad\mathbf{P}_{y_+}
\mbox{-a.s.} \mbox{ for }\mathbf{Y}\mbox{-a.e. }y %
\]
holds for any $A\in\mathcal{F}$ by martingale convergence.
By Lemma~\ref{lem:weiz}, this implies
\[
\mathbf{P} \bigl[A|\mathcal{F}^Y_+\vee\mathcal{F}^Z_{n,\infty}
\bigr] \mathop{\longrightarrow}^{n\to\infty}\mathbf{P} \bigl[A|\mathcal{F}^Y_+
\bigr],\qquad \mathbf{P}\mbox{-a.s.} \mbox{ for every }A\in\mathcal{F}. %
\]
This evidently establishes the claim.
\end{pf}

We now turn to the proof of Proposition~\ref{prop:oneside}.
The essential difficulty is that we must show equivalence of
two measures on an infinite time interval. The following lemma
provides a simple tool for this purpose. Recall that given two
probability measures $\mu$ and $\nu$, the measure $\mu\wedge\nu$ is
defined as
\[
d(\mu\wedge\nu) = \biggl(\frac{d\mu}{d(\mu+\nu)}\wedge \frac{d\nu}{d(\mu+\nu)} \biggr) d(
\mu+\nu), %
\]
so that $\|\mu-\nu\|=2(1-\|\mu\wedge\nu\|)$ \cite{Lin02}, page 99. By a
slight abuse of notation, we write $\mathbf{E}_{\mu\wedge\nu}[f] =
\int f\, d(\mu\wedge\nu)$ despite that $\mu\wedge\nu$ is only a
subprobability.

\begin{lem}
\label{lem:abscont}
Let $H$ be a Polish space, and let $\mu,\nu$ be probability measures
on $H^\mathbb{N}$. Denote by $X_i\dvtx H^\mathbb{N}\to H$ the coordinate
projections $X_i(x)=x_i$, and define the $\sigma$-fields
$\mathcal{G}_{m,n}=\sigma\{X_{m,n}\}$
for $m\le n$. If we have
\[
\mu[X_{1,n-1}\in\cdot |\mathcal{G}_{n,\infty}] \ll
\nu[X_{1,n-1}\in\cdot |\mathcal{G}_{n,\infty}],\qquad \mu|_{\mathcal{G}_{n,\infty}}
\wedge \nu|_{\mathcal{G}_{n,\infty}}\mbox{-a.s.} %
\]
for all $n<\infty$, and if in addition
\[
\|\mu-\nu\|_{\mathcal{G}_{n,\infty}}\mathop{\longrightarrow}^{n\to\infty}0, %
\]
then $\mu\ll\nu$ on $\mathcal{G}_{1,\infty}$.
\end{lem}

\begin{pf}
Let $\mu_n = \mu|_{\mathcal{G}_{n,\infty}}$ and
$\nu_n=\nu|_{\mathcal{G}_{n,\infty}}$. Choose any
$A\in\mathcal{G}_{1,\infty}$ such that $\nu(A)=0$. Then
$\nu[A|\mathcal{G}_{n,\infty}]=0$, $\nu_n$-a.s. and, therefore,
$\mu[A|\mathcal{G}_{n,\infty}]=0$, $\mu_n\wedge\nu_n$-a.s. by the first
assumption. But using the second assumption
\[
\mu(A) = \mathbf{E}_{\mu_n\wedge\nu_n} \bigl[\mu[A|\mathcal{G}_{n,\infty}]
\bigr]+ \mathbf{E}_{\mu_n-\mu_n\wedge\nu_n} \bigl[\mu[A|\mathcal{G}_{n,\infty}]
\bigr]
\le1-\|\mu_n\wedge\nu_n\| \mathop{\longrightarrow}^{n\to\infty}0,
\]
where we used $\|\mu-\nu\|_{\mathcal{G}_{n,\infty}}=
2(1-\|\mu_n\wedge\nu_n\|)$.
Thus, $\mu\ll\nu$ on $\mathcal{G}_{1,\infty}$.
\end{pf}

We can now complete the proof of Proposition~\ref{prop:oneside}.

\begin{pf*}{Proof of Proposition~\ref{prop:oneside}}
By Lemma~\ref{lem:3.6}, it suffices to show that
\[
\mathbf{P} \bigl[Y_{-\infty,-1}\in\cdot |\mathcal{F}^Y_+ \bigr]
\sim \mathbf{P}[Y_{-\infty,-1}\in\cdot |\mathcal{F}_+],\qquad \mathbf{P}\mbox{-a.s.}
\]
Fix versions of the regular conditional probabilities
\begin{eqnarray*}
\mathbf{P}_{Y_{0,\infty}}&=&\mathbf{P} \bigl[ \cdot |\mathcal{F}^Y_+
\bigr], \qquad\mathbf{P}_{Y_{0,\infty},Y_{-\infty,-n}}=\mathbf{P} \bigl[ \cdot |\mathcal
{F}^Y_+\vee\mathcal{F}^Y_{-\infty,-n} \bigr],
\\
\mathbf{P}_{X_{0,\infty}}&=&\mathbf{P}[ \cdot |\mathcal{F}_+],\qquad
\mathbf{P}_{X_{0,\infty},Y_{-\infty,-n}}=\mathbf{P} \bigl[ \cdot |\mathcal {F}_+\vee
\mathcal{F}^Y_{-\infty,-n} \bigr].
\end{eqnarray*}
To complete the proof, we show that
$\mathbf{P}_{y_{0,\infty}}[Y_{-\infty,-1}\in\cdot ]\sim
\mathbf{P}_{x_{0,\infty}}[Y_{-\infty,-1}\in\cdot ]$
for $\mathbf{P}$-a.e. $x=(z,y)$.
To this end, we verify the conditions of Lemma~\ref{lem:abscont}.

First, we claim that
\[
\|\mathbf{P}_{y_{0,\infty}}- \mathbf{P}_{x_{0,\infty}}\|_{\mathcal{F}^Y_{-\infty,-m}}
\mathop{\longrightarrow}^{m\to\infty}0 \qquad\mbox{for }\mathbf{P}\mbox{-a.e. } x=(z,y). %
\]
Indeed, note that by the triangle inequality and Jensen's inequality
\begin{eqnarray*}
&&\bigl\|\mathbf{P} \bigl[ \cdot |\mathcal{F}^Y_+ \bigr] - \mathbf{P}[
\cdot |\mathcal{F}_+]\bigr\|_{\mathcal{F}^Y_{-\infty,-m}}
\\
&&\qquad\le \bigl\|\mathbf{P}[ \cdot
|\mathcal{F}_+] -\mathbf{P}\bigr\|_{\mathcal{F}_{-\infty,-m}^Y} + \bigl\|
\mathbf{P} \bigl[ \cdot |\mathcal{F}_+^Y \bigr] -\mathbf{P}
\bigr\|_{\mathcal{F}_{-\infty,-m}^Y}
\\
&&\qquad\le\bigl \|\mathbf{P}[ \cdot |\mathcal{F}_+]
-\mathbf{P}\bigr\|_{\mathcal{F}_{-\infty,-m}} +
\mathbf{E} \bigl[\bigl\|\mathbf{P}[ \cdot |\mathcal{F}_+] -\mathbf{P}
\bigr\|_{\mathcal{F}_{-\infty,-m}}| \mathcal{F}^Y_+ \bigr].
\end{eqnarray*}
By Corollary~\ref{cor:absreg}, it suffices to show that the time-reversed
process $(X_{-k})_{k\in\mathbb{Z}}$ is absolutely regular. But it is
clear from Definition~\ref{def:absreg} that the absolute regularity
property of a stationary sequence is invariant under time reversal.
As we assumed absolute regularity of
$(X_k)_{k\in\mathbb{Z}}$, the claim follows.

Next, we claim that for $\mathbf{P}$-a.e. $x=(z,y)$
\begin{eqnarray}
\mathbf{P}_{y_{0,\infty}} \bigl[Y_{-m+1,-1}\in\cdot |
\mathcal{F}^Y_{-\infty,-m} \bigr] \sim \mathbf{P}_{x_{0,\infty}}
\bigl[Y_{-m+1,-1}\in\cdot | \mathcal{F}^Y_{-\infty,-m} \bigr],
\nonumber\\
\eqntext{\mathbf{P}_{y_{0,\infty}}|_{\mathcal{F}^Y_{-\infty,-m}} \wedge
\mathbf{P}_{x_{0,\infty}}|_{\mathcal{F}^Y_{-\infty,-m}}
\mbox{-a.s.}}
\end{eqnarray}
Indeed, by Lemma~\ref{lem:weiz}, it suffices to show that
\begin{eqnarray}
\mathbf{P}_{y_{0,\infty},Y_{-\infty,-m}}[Y_{-m+1,-1}\in\cdot ] \sim
\mathbf{P}_{x_{0,\infty},Y_{-\infty,-m}}[Y_{-m+1,-1}\in\cdot ],
\nonumber\\
\eqntext{\mathbf{P}_{y_{0,\infty}}|_{\mathcal{F}^Y_{-\infty,-m}} \wedge \mathbf{P}_{x_{0,\infty}}|_{\mathcal{F}^Y_{-\infty,-m}}
\mbox{-a.s.}}
\end{eqnarray}
To see that this is the case, note that
the nondegeneracy assumption yields
\[
\mathbf{P} \bigl[Y_{-m+1,-1}\in\cdot |\mathcal{F}^Y_{-\infty,-m}
\vee \mathcal{F}^Y_+ \bigr] \sim \mathbf{P} \bigl[Y_{-m+1,-1}
\in \cdot |\mathcal{F}^Y_{-\infty,-m}\vee \mathcal{F}_+ \bigr],\qquad
\mathbf{P}\mbox{-a.s.} %
\]
as in the proof of Theorem~\ref{thmm:inherit}.
Thus there is a measurable set $H\subseteq(G\times F)^\mathbb{N}\times
F^\mathbb{N}$ with $\mathbf{P}[(X_{0,\infty},Y_{-\infty,-m})\in H]=1$ such
that
\[
\mathbf{P}_{y_{0,\infty},y_{-\infty,-m}}[Y_{-m+1,-1}\in\cdot ] \sim
\mathbf{P}_{x_{0,\infty},y_{-\infty,-m}}[Y_{-m+1,-1}\in\cdot ] %
\]
for all
$(x_{0,\infty},y_{-\infty,-m})\in H$. But as
$\mathbf{P}[(X_{0,\infty},Y_{-\infty,-m})\in H]=1$ implies\break
$\mathbf{P}_{x_{0,\infty}}[(x_{0,\infty},Y_{-\infty,-m})\in H]=1$ for
$\mathbf{P}$-a.e. $x$ by disintegration, we obtain
\[
\mathbf{P}_{y_{0,\infty},Y_{-\infty,-m}}[Y_{-m+1,-1}\in\cdot ] \sim
\mathbf{P}_{x_{0,\infty},Y_{-\infty,-m}}[Y_{-m+1,-1}\in\cdot ],\qquad\mathbf{P}_{x_{0,\infty}}
\mbox{-a.s.}, %
\]
and thus {a fortiori}
$\mathbf{P}_{y_{0,\infty}}|_{\mathcal{F}^Y_{-\infty,-m}}\wedge
\mathbf{P}_{x_{0,\infty}}|_{\mathcal{F}^Y_{-\infty,-m}}$-a.s.,
for $\mathbf{P}$-a.e. $x$.

The two claims above verify the assumptions of Lemma~\ref{lem:abscont}, which yields
$\mathbf{P}_{y_{0,\infty}}[Y_{-\infty,-1}\in\cdot ]\sim
\mathbf{P}_{x_{0,\infty}}[Y_{-\infty,-1}\in\cdot ]$
for $\mathbf{P}$-a.e. $x$ as was to be shown.
\end{pf*}

\section{Ergodicity of the filter}
\label{sec:filter}

Let $(X_n,Y_n)_{n\ge0}$ be a Markov chain. We interpret
$(X_n)_{n\ge0}$ as the unobserved component of the model, while
$(Y_n)_{n\ge0}$ is the observable process. In this setting, there are
two distinct levels on which the conditional ergodic theory of the model
can be investigated.

In the previous section, we investigated directly the ergodic properties
of the unobserved process $(X_n)_{n\ge0}$ conditionally on the
observations $(Y_n)_{n\ge0}$. This setting is of interest if the
entire observation sequence $(Y_n)_{n\ge0}$ is available {a
priori}. In contrast, it is often of interest to consider the setting of
\emph{causal} conditioning, where we wish to infer the current state
$X_n$ of the unobserved process given only the history of observations
to date $\mathcal{F}^Y_{0,n}$. An object of central importance in this
setting is the \emph{nonlinear filter}
\[
\pi_n = \mathbf{P} \bigl[X_n\in\cdot |
\mathcal{F}^Y_{0,n} \bigr]. %
\]
Evidently, the filtering process $(\pi_n)_{n\ge0}$ is a measure-valued
process that is adapted to the observation filtration $\mathcal{F}^Y_{0,n}$.
The goal of this section is to investigate the stability and ergodic properties
of the filter $(\pi_n)_{n\ge0}$.

In Section~\ref{sec:fsetting}, we will develop the basic setting and
notation to be used throughout this section. In Section~\ref{sec:flocal}, we develop a local stability result for the nonlinear
filter, which is in essence the filtering counterpart of the local
zero--two laws of Section~\ref{sec:loc02}. In Section~\ref{sec:fstab},
we apply the local stability result to
develop a number of general stability and ergodicity
results for the nonlinear filter that are applicable to
infinite-dimensional models. Finally, in Section~\ref{sec:conttime}
we will extend our results to the continuous time setting.

The filter stability and ergodicity results developed in this section
provided the main motivation for the theory developed in this paper;
their broad applicability will be illustrated in Section~\ref{sec:examples} below.

\subsection{Setting and notation}
\label{sec:fsetting}

\subsubsection{The canonical setup}

Throughout this section, we will consider the bivariate stochastic process
$(X_n,Y_n)_{n\in\mathbb{Z}}$, where $X_n$ takes values in the Polish space
$E$ and $Y_n$ takes values in the Polish space $F$.
We realize this process on the canonical path space
$\Omega=\Omega^X\times\Omega^Y$ with
$\Omega^X= E^{\mathbb{Z}}$ and $\Omega^Y=F^{\mathbb{Z}}$, such that
$X_n(x,y)=x(n)$ and $Y_n(x,y)=y(n)$.
Denote by $\mathcal{F}$ the Borel $\sigma$-field of $\Omega$, and define
$X_{m,n}=(X_k)_{m\le k\le n}$, $Y_{m,n}=(Y_k)_{m\le k\le n}$, and
\[
\mathcal{F}^X_{m,n}=\sigma\{X_{m,n}\},\qquad
\mathcal{F}^Y_{m,n}=\sigma\{Y_{m,n}\},\qquad
\mathcal{F}_{m,n}=\mathcal{F}^X_{m,n}\vee
\mathcal{F}^Y_{m,n} %
\]
for $m\le n$.
For simplicity of notation, we define the $\sigma$-fields
\begin{eqnarray*}
\mathcal{F}^X&=&\mathcal{F}^X_{-\infty,\infty},\qquad
\mathcal{F}^Y=\mathcal{F}^Y_{-\infty,\infty},\qquad
\mathcal{F}^X_+=\mathcal{F}^X_{0,\infty},\qquad
\mathcal{F}^Y_+=\mathcal{F}^Y_{0,\infty},\\
\mathcal{F}_+&=&\mathcal{F}_{0,\infty}. %
\end{eqnarray*}
Finally, we denote by $Y$ the $F^{\mathbb{Z}}$-valued random variable
$(Y_k)_{k\in\mathbb{Z}}$, and the canonical shift $\Theta\dvtx\Omega\to
\Omega$
is defined as $\Theta(x,y)(m) = (x(m+1),y(m+1))$.

For any Polish space $Z$, we denote by $\mathcal{B}(Z)$ its Borel
$\sigma$-field, and by $\mathcal{P}(Z)$ the space of all probability
measures on $Z$ endowed with the weak convergence topology [thus
$\mathcal{P}(Z)$ is again Polish]. Let us recall that any probability
kernel $\rho\dvtx Z\times\mathcal{B}(Z')\to[0,1]$ may be equivalently viewed
as a $\mathcal{P}(Z')$-valued random variable
$z\mapsto\rho(z, \cdot )$ on $(Z,\mathcal{B}(Z))$. For notational
convenience, we will implicitly identify probability kernels and random
probability measures in the sequel.
The notation for total variation distance is as in Section~\ref{sec:zerotwo}.

\subsubsection{The Markov model}

The basic model of this section is defined by a Markov transition kernel
$P\dvtx E\times F\times\mathcal{B}(E\times F)\to[0,1]$ on $E\times F$.
Denote by $\mathbf{P}^\mu$ the probability measure on $\mathcal{F}_+$
such that $(X_n,Y_n)_{n\ge0}$ is a Markov chain with transition kernel
$P$ and initial law $(X_0,Y_0)\sim\mu\in\mathcal{P}(E\times F)$.
For any $(x,y)\in E\times F$, we will denote for simplicity the law of the
Markov chain started at the point mass $(X_0,Y_0)=(x,y)$ as
$\mathbf{P}^{x,y}=\mathbf{P}^{\delta_x\otimes\delta_y}$.

We now impose the following standing assumption.

\begin{saspt*}
The Markov transition kernel $P$ admits an invariant probability
measure $\lambda\in\mathcal{P}(E\times F)$, that is, $\lambda P=\lambda$.
\end{saspt*}

Let us emphasize that we do not rule out
at this point the existence of more than one invariant probability;
we simply fix one invariant probability $\lambda$ in what follows.
Our results will be stated in terms of $\lambda$.

Note that by construction, $(X_n,Y_n)_{n\ge0}$ is a stationary Markov
chain under~$\mathbf{P}^\lambda$. We can therefore naturally extend
$\mathbf{P}^\lambda$ to $\mathcal{F}$ such that the two-sided process
$(X_n,Y_n)_{n\in\mathbb{Z}}$ is the stationary Markov chain with
invariant probability $\lambda$ under $\mathbf{P}^\lambda$.
For simplicity, we will frequently write $\mathbf{P}=\mathbf{P}^\lambda$.

\subsubsection{The nonlinear filter}

The Markov chain $(X_n,Y_n)_{n\ge0}$ consists of two components:
$(X_n)_{n\ge0}$ represents the unobservable component of the model,
while $(Y_n)_{n\ge0}$ represents the observable component. As
$(X_n)_{n\ge0}$ is presumed to be unobservable, we are interested at time
$n$ in the conditional distribution given the observation history to
date $Y_0,\ldots,Y_n$.
To this end, we will introduce for every $\mu\in\mathcal{P}(E\times F)$
the following random measures:
\[
\Pi_n^\mu= \mathbf{P}^\mu \bigl[ \cdot |
\mathcal{F}^Y_{0,n} \bigr], \qquad\pi_n^\mu=
\mathbf{P}^\mu \bigl[X_n\in\cdot |\mathcal{F}^Y_{0,n}
\bigr]. %
\]
The $\mathcal{P}(E)$-valued process $(\pi_n^\mu)_{n\ge0}$ is called the
\emph{nonlinear filter} started at $\mu$. This is ultimately our main
object of interest. However, we will find it convenient to investigate
the full conditional distributions $\Pi_n^\mu$. When $\mu=\lambda$ is
the invariant measure, we will write $\Pi_n=\Pi_n^\lambda$ and
$\pi_n=\pi_n^\lambda$.

\begin{rem}
Note that $\Pi_n^\mu,\pi_n^\mu$ are
$\mathcal{F}^Y_{0,n}$-measurable kernels. That is,
$\pi_n^\mu\dvtx \Omega\times\mathcal{B}(E)\to[0,1]$ can be written
as $\pi_n^\mu(A)=\pi_n^\mu[Y_{0,n};A]$ for $A\in\mathcal{B}(E)$.
We will mostly suppress the
dependence on $Y_{0,n}$ for notational convenience.
\end{rem}

\subsection{A local stability result}
\label{sec:flocal}

The main tool that we will develop to investigate the ergodic theory of
nonlinear filters is a \emph{local} stability result for the conditional
distributions $\Pi_n^\mu$. To this end, we fix in this
subsection a countably generated
local $\sigma$-field $\mathcal{E}^0\subseteq\mathcal{B}(E)$
as in Section~\ref{sec:loc02}, and define
\[
\mathcal{F}^0_{m,n}= \mathcal{F}^Y_{m,n}
\vee\bigvee_{m\le k\le n} X_k^{-1}
\bigl(\mathcal{E}^0 \bigr),\qquad m<n. %
\]
Let us emphasize that the localization pertains only to the unobserved
component~$X_k$: it is essential for our results that the full observation
variable $Y_k$ is included in the local filtration $\mathcal{F}^0_{m,n}$.
In practice, this requirement and the nondegeneracy assumption below are
easily satisfied when the observations are finite-dimensional, but place
restrictions on the applicability of our theory when both unobserved
and observed processes are infinite-dimensional; cf. Section~\ref{sec:examples} for examples and Remark~\ref{rem:idimobs} for further
discussion.

As in Section~\ref{sec:conderg}, we require two basic assumptions.
The first assumption states that the model $(X_n,Y_n)_{n\ge0}$ is locally
ergodic.

\begin{aspt}[(Local ergodicity)]
\label{aspt:locerg}
The following holds:
\[
\bigl\|\mathbf{P}^{x,y}-\mathbf{P}\bigr\|_{\mathcal{F}^0_{n,\infty}}
 \mathop{\longrightarrow}^{n\to
\infty}0 \qquad\mbox{for }\lambda\mbox{-a.e. }(x,y)\in E\times F. %
\]
\end{aspt}

The second assumption provides a notion of nondegeneracy that is adapted
to the present setting (cf. \cite{TvH12}): it states that the dynamics of
the unobserved and observed processes can be made independent on a finite
time interval by an equivalent change of measure. Roughly speaking, this
requirement ensures that we cannot infer with certainty the outcome of any
unobserved event given a finite number of observations; indeed, by the
Bayes formula, Assumption~\ref{aspt:nondeg} implies that the conditional
distribution given a finite number of observations is equivalent to the
unconditional distribution. Our results can certainly fail in
the absence of such a property, cf. \cite{vH12} for examples.

\begin{aspt}[(Nondegeneracy)]
\label{aspt:nondeg}
There exist Markov transition kernels $P_0\dvtx E\times\mathcal{B}(E)\to[0,1]$
and $Q\dvtx F\times\mathcal{B}(F)\to[0,1]$ such that
\[
P \bigl(x,y,dx',dy' \bigr) = g \bigl(x,y,x',y'
\bigr) P_0 \bigl(x,dx' \bigr) Q \bigl(y,dy'
\bigr) %
\]
for some strictly positive measurable function $g\dvtx E\times F\times
E\times F\to\mbox{}]0,\infty[$.
\end{aspt}

Note that Assumption~\ref{aspt:locerg} is characterized by Theorem~\ref{thmm:sloc02}, which yields a general tool to verify this assumption.
Assumption~\ref{aspt:nondeg} is easily verified in practice as it is
stated directly in terms of the underlying model.

The main result of this section is as follows.

\begin{thmm}[(Local filter stability)]
\label{thmm:lfstab}
Suppose that Assumptions \ref{aspt:locerg} and \ref{aspt:nondeg} hold.
Then for any initial probability $\mu\in\mathcal{P}(E\times F)$ such that
\[
\mu(E\times\cdot)\ll\lambda(E\times\cdot) \quad\mbox{and}\quad\bigl \|\Pi^\mu_0-
\Pi_0\bigr\|_{\mathcal{F}^0_{n,\infty}} \mathop{\longrightarrow}^{n\to\infty}0,\qquad
\mathbf{P}^\mu\mbox{-a.s.}, %
\]
we have
\[
\bigl\|\Pi_n^\mu-\Pi_n\bigr\|_{\mathcal{F}^0_{n-r,\infty}}
\mathop{\longrightarrow}^{n\to\infty}0,\qquad \mathbf{P}^\mu\mbox{-a.s.} \mbox{ for any }r
\in\mathbb{N}. %
\]
If $\mu(E\times\cdot)\sim\lambda(E\times\cdot)$, the convergence
holds also $\mathbf{P}$-a.s.
\end{thmm}

\begin{rem}
When interpreting this result, we must take care to ensure that the
relevant quantities are well defined. Recall that $\Pi_n^\mu$, as a
regular conditional probability, is defined uniquely up
to a $\mathbf{P}^\mu|_{\mathcal{F}^Y_{0,n}}$-null set only.
Thus in order that $\Pi_n$ is $\mathbf{P}^\mu$-a.s. is uniquely defined
we must at least have
$\mathbf{P}^\mu|_{\mathcal{F}^Y_{0,n}}\ll\mathbf{P}|_{\mathcal{F}^Y_{0,n}}$.
The assumption $\mu(E\times\cdot)\ll\lambda(E\times\cdot)$
is therefore necessary even for the statement of Theorem~\ref{thmm:lfstab}
to make sense.
As part of
the proof, we will in fact show that under the stated assumptions
$\mathbf{P}^\mu|_{\mathcal{F}^Y_+}\ll\mathbf{P}|_{\mathcal{F}^Y_+}$
[resp., $\mathbf{P}^\mu|_{\mathcal{F}^Y_+}\sim
\mathbf{P}|_{\mathcal{F}^Y_+}$ when
$\mu(E\times\cdot)\sim\lambda(E\times\cdot)$].
This ensures that $\Pi_n$ is $\mathbf{P}^\mu$-a.s. uniquely defined
(resp. $\Pi_n^\mu$ is $\mathbf{P}$-a.s. uniquely defined) for every
$n\ge0$.
\end{rem}

If we strengthen Assumption~\ref{aspt:locerg} as in Theorem~\ref{thmm:loc02}, we obtain the following.

\begin{cor}
\label{cor:lfverystab}
Suppose that Assumption~\ref{aspt:nondeg} holds and that
\[
\bigl\|\mathbf{P}^{x,y}-\mathbf{P}\bigr\|_{\mathcal{F}^0_{n,\infty}} \mathop{\longrightarrow}^{n\to
\infty}0 \qquad\mbox{for all }(x,y)\in E\times F. %
\]
Then for any $\mu\in\mathcal{P}(E\times F)$ such that
$\mu(E\times\cdot)\ll\lambda(E\times\cdot)$, we have
\[
\bigl\|\Pi_n^\mu-\Pi_n\bigr\|_{\mathcal{F}^0_{n-r,\infty}}
\mathop{\longrightarrow}^{n\to\infty}0,\qquad \mathbf{P}^\mu\mbox{-a.s. } \mbox{for any }r
\in\mathbb{N}. %
\]
If $\mu(E\times\cdot)\sim\lambda(E\times\cdot)$, the convergence
holds also $\mathbf{P}$-a.s.
\end{cor}

\begin{pf}
The assumption clearly implies Assumption~\ref{aspt:locerg}.
Moreover,
\begin{eqnarray*}
\bigl\|\Pi^\mu_0-\Pi_0\bigr\|_{\mathcal{F}^0_{n,\infty}} &\le& \bigl\|
\mathbf{P}^\mu[ \cdot |Y_0]-\mathbf{P}\bigr\|_{\mathcal{F}^0_{n,\infty
}} +
\bigl\|\mathbf{P}[ \cdot |Y_0]-\mathbf{P}\bigr\|_{\mathcal{F}^0_{n,\infty}}
\\
&\le& \mathbf{E}^\mu \bigl[ \bigl\|\mathbf{P}^{X_0,Y_0}-\mathbf{P}
\bigr\|_{\mathcal{F}^0_{n,\infty}} |Y_0 \bigr] + \mathbf{E} \bigl[ \bigl\|
\mathbf{P}^{X_0,Y_0}-\mathbf{P}\bigr\|_{\mathcal{F}^0_{n,\infty}} |Y_0 \bigr].
\end{eqnarray*}
Thus, $\|\Pi^\mu_0-\Pi_0\|_{\mathcal{F}^0_{n,\infty}}\to0$,
$\mathbf{P}^\mu$-a.s. It remains to apply Theorem~\ref{thmm:lfstab}.
\end{pf}

We now turn to the proof of Theorem~\ref{thmm:lfstab}. We begin with
a trivial consequence of the Bayes formula that we formulate for
completeness.

\begin{lem}
\label{lem:bayes}
Let $\mu,\nu$ be probability measures on a Polish space $H$, and let
$\mathcal{G}\subseteq\mathcal{B}(H)$ be a $\sigma$-field. If
$\mu\sim\nu$, then $\mu[ \cdot |\mathcal{G}]\sim
\nu[ \cdot |\mathcal{G}]$, $\mu$-a.s. and $\nu$-a.s.
\end{lem}

\begin{pf}
Let $\Lambda=d\mu/d\nu$. Then for any $A\in\mathcal{B}(H)$, we have
\[
\mu[A|\mathcal{G}] = \mathbf{E}_\nu \biggl[\mathbf{1}_A
\frac{\Lambda}{\mathbf{E}_\nu[\Lambda|\mathcal{G}]}\Big |\mathcal{G} \biggr]
\quad\mbox{and} \quad\frac{\Lambda}{\mathbf{E}_\nu[\Lambda|\mathcal{G}]}>0,\qquad \mu
\mbox{-a.s.} %
\]
by the Bayes formula and using $\mu\sim\nu$.
As $\mathcal{B}(H)$ is countably generated,
the $\mu$-exceptional set can be chosen independent of $A$ by a monotone
class argument. Thus, $\mu[ \cdot |\mathcal{G}]\sim
\nu[ \cdot |\mathcal{G}]$, $\mu$-a.s., and also $\nu$-a.s. as
$\mu\sim\nu$.
\end{pf}

To proceed, we first prove a key measure-theoretic identity that arises
from the conditional ergodic theory developed in Section~\ref{sec:conderg} above.

\begin{lem}
\label{lem:exchglocal}
Suppose that Assumptions \ref{aspt:locerg} and \ref{aspt:nondeg} hold.
Then
\[
\bigcap_{n\ge0}\mathcal{F}^Y_+\vee
\mathcal{F}^0_{n,\infty}=\mathcal{F}^Y_+\qquad \mathop{
\mathrm{mod}}\mathbf{P}. %
\]
\end{lem}

\begin{pf}
Let $(A_n)_{n\in\mathbb{N}}$ be a countable generating class for
$\mathcal{E}^0$, and define $Z_n=\iota(X_n)$ where
$\iota\dvtx E\to\{0,1\}^\mathbb{N}$ is given by
$\iota(x)=(\mathbf{1}_{A_n}(x))_{n\in\mathbb{N}}$.
Then $\mathcal{F}^0_{m,n}=
\mathcal{F}^Y_{m,n}\vee\mathcal{F}^Z_{m,n}$ by construction.
It now suffices to show that the
stationary process $(Z_n,Y_n)_{n\in\mathbb{Z}}$ satisfies the assumptions
of Corollary~\ref{cor:exchg}.

First, note that
by Assumption~\ref{aspt:nondeg} and the Markov property
\[
\bigl\|\mathbf{P} \bigl[ \cdot |\mathcal{F}^0_{-\infty,0} \bigr]-
\mathbf{P}\bigr\|_{
\mathcal{F}^0_{n,\infty}} \le \mathbf{E} \bigl[\bigl \|\mathbf{P}^{X_0,Y_0} -
\mathbf{P}\bigr\|_{
\mathcal{F}^0_{n,\infty}} |\mathcal{F}^0_{-\infty,0} \bigr]
\mathop{\longrightarrow}^{n\to\infty}0,\qquad\mathbf{P}\mbox{-a.s.} %
\]
Thus, Corollary~\ref{cor:absreg} yields absolute regularity of
$(Z_n,Y_n)_{n\in\mathbb{Z}}$.

It remains to prove nondegeneracy (in the sense of Definition~\ref{defn:nondeg}). To this end, we begin by noting that
by the Markov property of $(X_n,Y_n)_{n\in\mathbb{Z}}$
\[
\mathbf{P}[Y_{1,n}\in\cdot |\mathcal{F}_{-\infty,0}\vee
\mathcal{F}_{n+1,\infty}] = \mathbf{P}[Y_{1,n}\in\cdot
|X_0,Y_0, X_{n+1},Y_{n+1}]. %
\]
Let $\mathbf{R}$ be the probability on $\mathcal{F}_+$ under
which $(X_k,Y_k)_{k\ge0}$ is a Markov chain with transition kernel
$P_0\otimes Q$ and initial law $\lambda$. Then
$\mathbf{P}|_{\mathcal{F}_{0,n+1}}\sim
\mathbf{R}|_{\mathcal{F}_{0,n+1}}$ by Assumption~\ref{aspt:nondeg}. Therefore, by Lemma~\ref{lem:bayes},
\[
\mathbf{P}[Y_{1,n}\in\cdot |\mathcal{F}_{-\infty,0}\vee
\mathcal{F}_{n+1,\infty}] \sim \mathbf{R}[Y_{1,n}\in\cdot
|X_0,Y_0, X_{n+1},Y_{n+1}],\qquad \mathbf{P}
\mbox{-a.s.} %
\]
But $X_{0,n+1}$ and $Y_{1,n+1}$ are conditionally independent under
$\mathbf{R}$ given $Y_0$, so
\[
\mathbf{R}[Y_{1,n}\in\cdot |X_0,Y_0,X_{n+1},Y_{n+1}]
= \mathbf{R}[Y_{1,n}\in\cdot |Y_0,Y_{n+1}],\qquad
\mathbf{P}\mbox{-a.s.} %
\]
Therefore, by Lemma~\ref{lem:tower},
\[
\mathbf{P} \bigl[Y_{1,n}\in\cdot |\mathcal{F}_{-\infty,0}^0
\vee \mathcal{F}_{n+1,\infty}^0 \bigr] \sim \mathbf{P}
\bigl[Y_{1,n}\in\cdot |\mathcal{F}_{-\infty,0}^Y\vee
\mathcal{F}_{n+1,\infty}^Y \bigr],\qquad \mathbf{P}\mbox{-a.s.}
\]
As this holds for any $n\in\mathbb{N}$, and using stationarity of
$\mathbf{P}$, it follows readily that the process
$(Z_k,Y_k)_{k\in\mathbb{Z}}$ is nondegenerate (Definition~\ref{defn:nondeg}).
\end{pf}

Armed with this result, we prove first a dominated stability lemma.

\begin{lem}
\label{lem:fsdom}
Suppose that Assumptions \ref{aspt:locerg} and \ref{aspt:nondeg} hold.
Let $\tilde{\mathbf P}$ be a probability measure on $\mathcal{F}_+$,
and define $\Sigma_n=\tilde{\mathbf P}[ \cdot |\mathcal{F}^Y_{0,n}]$.
Suppose that
\[
\tilde{\mathbf P}|_{\mathcal{F}^Y_+\vee\mathcal{F}^0_{m,\infty}} \ll
 \mathbf{P}|_{\mathcal{F}^Y_+\vee\mathcal{F}^0_{m,\infty}} \qquad\mbox{for
some }m\in\mathbb{N}. %
\]
Then $\|\Sigma_n-\Pi_n\|_{\mathcal{F}^0_{n-r,\infty}}\to0$,
$\tilde{\mathbf P}$-a.s. as $n\to\infty$ for any $r\in\mathbb{N}$.
\end{lem}

\begin{pf}
Fix any $r$ and $m$ as in the statement of the lemma,
and let $n\ge m+r$. By the Bayes formula, we obtain for any
set $A\in\mathcal{F}^0_{n-r,\infty}$
\[
\tilde{\mathbf P} \bigl[A|\mathcal{F}^Y_{0,n} \bigr] =
\frac{\mathbf{E}[\mathbf{1}_A
\mathbf{E}[\Lambda|\mathcal{F}^Y_+\vee\mathcal{F}^0_{n-r,\infty}]
|\mathcal{F}^Y_{0,n}]}{
\mathbf{E}[\Lambda|\mathcal{F}^Y_{0,n}]} \qquad\tilde{\mathbf P}\mbox{-a.s.}, \Lambda= \frac{d\tilde{\mathbf P}|_{\mathcal{F}^Y_+\vee\mathcal
{F}^0_{m,\infty}}}{
d\mathbf{P}|_{\mathcal{F}^Y_+\vee\mathcal{F}^0_{m,\infty}}}
\]
(we write $\mathbf{E}[X]=\mathbf{E}_P[X]$ for simplicity).
We therefore have for $A\in\mathcal{F}^0_{n-r,\infty}$
\[
\Sigma_n(A)-\Pi_n(A) = \int\mathbf{1}_A
\biggl\{ \frac{\mathbf{E}[\Lambda|\mathcal{F}^Y_+\vee\mathcal{F}^0_{n-r,\infty}]
}{\mathbf{E}[\Lambda|\mathcal{F}^Y_{0,n}]}-1 \biggr\} \,d\Pi_n,\qquad\tilde{\mathbf P}
\mbox{-a.s.} %
\]
As $\mathcal{F}^0_{n-r,\infty}$ is countably generated, the
$\tilde{\mathbf P}$-exceptional set can be chosen independent of $A$
using a monotone class argument. It follows that
\[
\|\Sigma_n-\Pi_n\|_{\mathcal{F}^0_{n-r,\infty}} \le \int \biggl|
\frac{\mathbf{E}[\Lambda|\mathcal{F}^Y_+\vee\mathcal{F}^0_{n-r,\infty}]
}{\mathbf{E}[\Lambda|\mathcal{F}^Y_{0,n}]}-1 \biggr| \,d\Pi_n = \frac{\mathbf{E}[\Delta_n|\mathcal{F}^Y_{0,n}]
}{\mathbf{E}[\Lambda|\mathcal{F}^Y_{0,n}]},\qquad \tilde{\mathbf P}
\mbox{-a.s.}, %
\]
where we have defined
\[
\Delta_n = \bigl|\mathbf{E} \bigl[\Lambda|\mathcal{F}^Y_+\vee
\mathcal {F}^0_{n-r,\infty} \bigr] - \mathbf{E} \bigl[\Lambda|
\mathcal{F}^Y_{0,n} \bigr]\bigr|. %
\]
We now estimate
\[
\mathbf{E} \bigl[\Delta_n|\mathcal{F}^Y_{0,n}
\bigr] \le \mathbf{E} \bigl[\Delta_n^u|
\mathcal{F}^Y_{0,n} \bigr] + 2 \mathbf{E} \bigl[\Lambda
\mathbf{1}_{\Lambda> u}|\mathcal{F}^Y_{0,n} \bigr],
\]
where
\[
\Delta_n^u=\bigl|\mathbf{E} \bigl[\Lambda
\mathbf{1}_{\Lambda\le u}|\mathcal {F}^Y_+\vee\mathcal{F}^0_{n-r,\infty}
\bigr] - \mathbf{E} \bigl[\Lambda\mathbf{1}_{\Lambda\le u}|\mathcal{F}^Y_{0,n}
\bigr]\bigr|. %
\]
By Lemma~\ref{lem:exchglocal} and Hunt's lemma \cite{DM82}, Theorem~V.45,
we obtain
$\mathbf{E}[\Delta_n^u|\mathcal{F}^Y_{0,n}]\to0$, $\mathbf{P}$-a.s. as $n\to\infty$.
Moreover, as
$\mathbf{E}[\Lambda|\mathcal{F}^Y_+]>0$, $\tilde{\mathbf P}$-a.s., it
follows that
\[
\limsup_{n\to\infty} \|\Sigma_n-\Pi_n
\|_{\mathcal{F}^0_{n-r,\infty}} \le \frac{
2 \mathbf{E}[\Lambda\mathbf{1}_{\Lambda> u}|\mathcal{F}^Y_+]} {
\mathbf{E}[\Lambda|\mathcal{F}^Y_+]},\qquad\tilde{\mathbf P}\mbox{-a.s.}
\]
Letting $u\to\infty$ completes the proof.
\end{pf}

To use the previous lemma, we will decompose $\mathbf{P}^\mu$ on
$\mathcal{F}^Y_+\vee\mathcal{F}^0_{m,\infty}$ into an absolutely
continuous component with respect to $\mathbf{P}$ (to which Lemma~\ref{lem:fsdom} can be applied) and a remainder that is negligible
as $m\to\infty$. The following lemma ensures that this can be done
under the assumptions of Theorem~\ref{thmm:lfstab}.

\begin{lem}
\label{lem:lebesguedec}
Suppose Assumption~\ref{aspt:nondeg} holds. If
$\mu\in\mathcal{P}(E\times F)$ satisfies
\[
\mu(E\times\cdot)\ll\lambda(E\times\cdot) \quad\mbox{and} \quad\bigl\|\Pi^\mu_0-
\Pi_0\bigr\|_{\mathcal{F}^0_{n,\infty}} \mathop{\longrightarrow}^{n\to\infty}0,\qquad
\mathbf{P}^\mu\mbox{-a.s.}, %
\]
then for every $m\in\mathbb{N}$, we can choose a set
$C_m\in\mathcal{F}^Y_+\vee\mathcal{F}^0_{m,\infty}$
such that $\mathbf{P}^\mu[ \cdot \cap C_m]\ll\mathbf{P}$ on
$\mathcal{F}^Y_+\vee\mathcal{F}^0_{m,\infty}$ and
$\mathbf{P}^\mu(C_m^c)\to0$ as $m\to\infty$.
\end{lem}

\begin{pf}
By the Lebesgue decomposition theorem, we can choose for every
$m\in\mathbb{N}$ a set
$C_m\in\sigma\{Y_0\}\vee\mathcal{F}^0_{m,\infty}$ such that the following
holds:
\[
\mathbf{P}^\mu[ \cdot \cap C_m]\ll\mathbf{P}\quad \mbox{on }
\sigma\{Y_0\}\vee\mathcal{F}^0_{m,\infty} \mbox{ and }
\mathbf{P}(C_m)=1. %
\]
Now note that $\Pi_0(C_m^c)=0$, $\mathbf{P}$-a.s. and, therefore, also
$\mathbf{P}^\mu$-a.s. as we have assumed that
$\mu(E\times\cdot)\ll\lambda(E\times\cdot)$.
Thus, we obtain $\mathbf{P}^\mu$-a.s.
\[
\Pi^\mu_0 \bigl(C_m^c \bigr) =
\Pi^\mu_0 \bigl(C_m^c \bigr) -
\Pi_0 \bigl(C_m^c \bigr) \le \bigl\|
\Pi_0^\mu-\Pi_0\bigr\|_{\sigma\{Y_0\}\vee\mathcal{F}^0_{m,\infty}} = \bigl\|
\Pi_0^\mu-\Pi_0\bigr\|_{\mathcal{F}^0_{m,\infty}}. %
\]
Taking the expectation with respect to $\mathbf{P}^\mu$ and letting
$m\to\infty$, it follows using dominated convergence that
$\mathbf{P}^\mu(C_m^c)\to0$ as $m\to\infty$.

It remains to show that
$\mathbf{P}^\mu[ \cdot \cap C_m]\ll\mathbf{P}$ on the larger
$\sigma$-field $\mathcal{F}^Y_+\vee\mathcal{F}^0_{m,\infty}$.
To this end, we will establish below the following claim:
$\mathbf{P}^\mu[ \cdot \cap C_m]$-a.s.
\[
\mathbf{P}^\mu \bigl[Y_{1,m-1}\in\cdot |\sigma
\{Y_0\}\vee\mathcal {F}^0_{m,\infty} \bigr] \sim
\mathbf{P} \bigl[Y_{1,m-1}\in\cdot |\sigma\{Y_0\}\vee
\mathcal{F}^0_{m,\infty} \bigr]. %
\]
Let us first complete the proof assuming the claim.
Let $A\in\mathcal{F}^Y_+\vee\mathcal{F}^0_{m,\infty}$ such that
$\mathbf{P}(A)=0$. Then
$\mathbf{P}[A|\sigma\{Y_0\}\vee\mathcal{F}^0_{m,\infty}]=0$,
$\mathbf{P}$-a.s. and, therefore, also $\mathbf{P}^\mu[ \cdot \cap
C_m]$-a.s. But then the claim implies that
$\mathbf{P}^\mu[A|\sigma\{Y_0\}\vee\mathcal{F}^0_{m,\infty}]=0$,
$\mathbf{P}^\mu[ \cdot \cap C_m]$-a.s. by disintegration, which yields
$\mathbf{P}^\mu(A\cap C_m) = 0$ as required.

We now proceed to prove the claim. Let $\mathbf{R}^\mu_m$ be the
probability on $\mathcal{F}_+$ under which $(X_k,Y_k)_{k\ge0}$ is an
inhomogeneous Markov chain with initial law $(X_0,Y_0)\sim\mu$,
whose transition kernel is given by $P_0\otimes Q$ up to time $m$
and by $P$ after time $m$. Assumption~\ref{aspt:nondeg} evidently
implies that $\mathbf{R}^\mu_m\sim\mathbf{P}^\mu$, so
\[
\mathbf{P}^\mu \bigl[Y_{1,m-1}\in\cdot |\sigma
\{Y_0\}\vee \mathcal{F}^0_{m,\infty} \bigr]\sim
\mathbf{R}^\mu_m \bigl[Y_{1,m-1}\in\cdot |\sigma
\{Y_0\}\vee \mathcal{F}^0_{m,\infty} \bigr],\qquad
\mathbf{P}^\mu\mbox{-a.s.} %
\]
by Lemma~\ref{lem:bayes}. Now note that by the Markov property
\[
\mathbf{R}^\mu_m \bigl[Y_{1,m-1}\in\cdot |
\sigma \{Y_0\}\vee\mathcal {F}_{m,\infty} \bigr]=
\mathbf{R}^\mu_m[Y_{1,m-1} \in\cdot
|Y_0,X_m,Y_m], %
\]
while
\[
\mathbf{R}^\mu_m[Y_{1,m-1}\in\cdot
|Y_0,X_m,Y_m]= \mathbf{R}^\mu_m[Y_{1,m-1}
\in\cdot |Y_0,Y_m] %
\]
as $Y_{1,m}$ and $X_{m}$ are conditionally independent given $Y_0$
under $\mathbf{R}^\mu_m$. Therefore, we obtain using the
tower property
\[
\mathbf{R}^\mu_m \bigl[Y_{1,m-1}\in\cdot |
\sigma \{Y_0\}\vee \mathcal{F}^0_{m,\infty} \bigr]=
\mathbf{R}^\mu_m[Y_{1,m-1}\in\cdot
|Y_0,Y_m],\qquad \mathbf{P}^\mu\mbox{-a.s.}
\]
Proceeding in exactly the same manner for $\mathbf{P}$, we obtain
\begin{eqnarray*}
\mathbf{P}^\mu \bigl[Y_{1,m-1}\in\cdot |\sigma
\{Y_0\}\vee \mathcal{F}^0_{m,\infty} \bigr]&\sim&
\mathbf{R}^\mu_m[Y_{1,m-1}\in\cdot
|Y_0,Y_m],\qquad \mathbf{P}^\mu[ \cdot \cap
C_m]\mbox{-a.s.},
\\
\mathbf{P} \bigl[Y_{1,m-1}\in\cdot |\sigma\{Y_0\}\vee
\mathcal{F}^0_{m,\infty} \bigr]&\sim& \mathbf{R}^\lambda_m[Y_{1,m-1}
\in\cdot |Y_0,Y_m],\qquad \mathbf{P}^\mu[ \cdot
\cap C_m]\mbox{-a.s.}
\end{eqnarray*}
It remains to note that $\mathbf{R}^\mu_m[Y_{1,m-1}\in\cdot |Y_0,Y_m]
=\mathbf{R}^\lambda_m[Y_{1,m-1}\in\cdot |Y_0,Y_m]$,
$\mathbf{P}^\mu[ \cdot \cap C_m]$-a.s., as $Y_{0,m}$ is (the initial
segment of) a Markov chain
with transition kernel $Q$ under both $\mathbf{R}^\mu_m$ and
$\mathbf{R}^\lambda_m$. Thus, the proof is complete.
\end{pf}

The following corollary will be used a number of times.

\begin{cor}
\label{cor:oequiv}
Suppose Assumption~\ref{aspt:nondeg} holds. If
$\mu$ satisfies
\[
\mu(E\times\cdot)\ll\lambda(E\times\cdot) \quad\mbox{and\quad} \bigl\|\Pi^\mu_0-
\Pi_0\bigr\|_{\mathcal{F}^0_{n,\infty}} \mathop{\longrightarrow}^{n\to\infty}0,\qquad
\mathbf{P}^\mu\mbox{-a.s.}, %
\]
then
$\mathbf{P}^\mu|_{\mathcal{F}^Y_+}\ll\mathbf{P}|_{\mathcal{F}^Y_+}$.
If also $\lambda(E\times\cdot)\sim\mu(E\times\cdot)$, then
$\mathbf{P}^\mu|_{\mathcal{F}^Y_+}\sim\mathbf{P}|_{\mathcal{F}^Y_+}$.
\end{cor}

\begin{pf}
Define the sets $C_m$ as in Lemma~\ref{lem:lebesguedec}, and choose any
$A\in\mathcal{F}^Y_+$ with $\mathbf{P}(A)=0$. Then
$\mathbf{P}^\mu(A)=\mathbf{P}^\mu(A\cap C_m^c)\le
\mathbf{P}^\mu(C_m^c)\to0$ as $m\to\infty$. This yields the
first claim.
On the other hand, note that the proof of Lemma~\ref{lem:lebesguedec} does
not use the invariance of $\lambda$.
Thus, we may exchange the roles of $\mu$ and $\lambda$
in Lemma~\ref{lem:lebesguedec} to obtain also
$\mathbf{P}|_{\mathcal{F}^Y_+}\ll\mathbf{P}^\mu|_{\mathcal{F}^Y_+}$
when $\lambda(E\times\cdot)\sim\mu(E\times\cdot)$.
\end{pf}

We can now complete the proof of Theorem~\ref{thmm:lfstab}.

\begin{pf*}{Proof of Theorem~\ref{thmm:lfstab}}
Fix $\mu$ as in the statement of the theorem, and define the
corresponding sets $C_m$ as in Lemma~\ref{lem:lebesguedec}.
Let $\mathbf{P}^\mu_m=\mathbf{P}^\mu[ \cdot |C_m]$ and
$\mathbf{P}^{\mu\perp}_m=\mathbf{P}^\mu[ \cdot |C_m^c]$.
By the Bayes formula, we can write for any $A\in\mathcal{F}^+$
\[
\mathbf{P}^\mu \bigl[A|\mathcal{F}^Y_{0,n}
\bigr] = \mathbf{P}^\mu_m \bigl[A|\mathcal{F}^Y_{0,n}
\bigr] \mathbf{P}^\mu \bigl[C_m|\mathcal{F}^Y_{0,n}
\bigr] + \mathbf{P}^{\mu\perp}_m \bigl[A|\mathcal{F}^Y_{0,n}
\bigr] \mathbf{P}^\mu \bigl[C_m^c|
\mathcal{F}^Y_{0,n} \bigr],\qquad\mathbf{P}^\mu
\mbox{-a.s.} %
\]
In particular, if we define
$\Sigma_n^m=\mathbf{P}^\mu_m[ \cdot |\mathcal{F}^Y_{0,n}]$, we can write
\[
\bigl|\Pi^\mu_n(A)-\Pi_n(A)\bigr|\le \bigl|
\Sigma^m_n(A)-\Pi_n(A)\bigr| \Pi^\mu_n(C_m)
+ \Pi^\mu_n \bigl(C_m^c \bigr),\qquad
\mathbf{P}^\mu\mbox{-a.s.} %
\]
As $\mathcal{F}_+$ is countably generated,
the $\mathbf{P}^\mu$-exceptional set can be chosen independent of
$A$ using a monotone class argument. We therefore obtain
\[
\mathbf{E}^\mu \Bigl[ \limsup_{n\to\infty} \bigl\|
\Pi^\mu_n-\Pi_n\bigr\|_{\mathcal{F}^0_{n-r,\infty}} \Bigr] \le
\mathbf{E}^\mu_m \Bigl[ \limsup_{n\to\infty}
\bigl\|\Sigma^m_n-\Pi_n\bigr\|_{\mathcal{F}^0_{n-r,\infty}} \Bigr] +
\mathbf{P}^\mu \bigl(C_m^c \bigr).
\]
But the first term on the right vanishes by Lemma~\ref{lem:fsdom}.
Therefore, using that $\mathbf{P}^\mu(C_m^c)\to0$ as $m\to\infty$,
we find that $\|\Pi^\mu_n-\Pi_n\|_{\mathcal{F}^0_{n-r,\infty}}\to0$,
$\mathbf{P}^\mu$-a.s.

This completes the proof when
$\lambda(E\times\cdot)\ll\mu(E\times\cdot)$. To conclude, note that
$\mathbf{P}$-a.s. convergence follows by Corollary~\ref{cor:oequiv}
when $\lambda(E\times\cdot)\sim\mu(E\times\cdot)$.
\end{pf*}

\subsection{Filter stability and ergodicity}
\label{sec:fstab}

Using the local stability Theorem~\ref{thmm:lfstab}, we can now proceed to
obtain filter stability results that are applicable to
infinite-dimensional or weak-* ergodic models, in analogy with the
ergodic results obtained in Section~\ref{sec:zerotwo}. While many
variations on these results are possible, we give two representative
results that suffice in all the examples that will be given in
Section~\ref{sec:examples} below. Beside stability, we will also consider
following Kunita \cite{Kun71} the ergodic properties of the filtering
process $(\pi_n)_{n\ge0}$ when it is considered as a measure-valued
Markov process.

\subsubsection{Filter stability and local mixing}
\label{sec:fslm}

In this short subsection, we assume that the state space $E$ of the unobserved
process is contained in a countable product $E\subseteq\prod_{i\in I}E^i$,
where each $E^i$ is Polish. We are therefore in the local mixing setting
of Section~\ref{sec:locmix}. In the present section, we define
\[
\mathcal{F}^J_{m,n} = \sigma \bigl\{X^J_{m,n},Y_{m,n}
\bigr\} \qquad\mbox{for }J\subseteq I, m\le n, %
\]
that is, we include the observations in the local filtration.
We also denote by $\mathcal{E}^J\subseteq\mathcal{B}(E)$ the cylinder
$\sigma$-field generated by the coordinates in $J\subseteq I$.

The bivariate Markov chain $(X_n,Y_n)_{n\ge0}$ is said to be
\emph{locally mixing} if
\[
\bigl\|\mathbf{P}^{x,y}-\mathbf{P}\bigr\|_{\mathcal{F}^J_{n,\infty}} \mathop{\longrightarrow}^{n\to
\infty}0\qquad \mbox{for all }(x,y)\in E\times F\mbox{ and } J\subseteq I, |J|<\infty.
\]
It is easily seen that this coincides with the notion introduced
in Section~\ref{sec:locmix}.
The following filter stability result follows trivially from
Corollary~\ref{cor:lfverystab}.

\begin{cor}
\label{cor:flocmix}
Suppose that the Markov chain $(X_n,Y_n)_{n\ge0}$ is locally mixing
and that
Assumption~\ref{aspt:nondeg} holds.
Then for any $\mu\in\mathcal{P}(E\times F)$ such that
$\mu(E\times\cdot)\ll\lambda(E\times\cdot)$, and
for any $r\in\mathbb{N}$, we have
\[
\bigl\|\Pi_n^\mu-\Pi_n\bigr\|_{\mathcal{F}^J_{n-r,\infty}}
\mathop{\longrightarrow}^{n\to\infty}0,\qquad \mathbf{P}^\mu\mbox{-a.s.} \mbox{ for all }J
\subseteq I, |J|<\infty. %
\]
In particular, the filter is stable in the sense
\[
\bigl\|\pi_n^\mu-\pi_n\bigr\|_{\mathcal{E}^J}
\mathop{\longrightarrow}^{n\to\infty}0,\qquad\mathbf{P}^\mu\mbox{-a.s.} \mbox{ for all }J
\subseteq I, |J|<\infty. %
\]
If $\mu(E\times\cdot)\sim\lambda(E\times\cdot)$, the convergence
holds also $\mathbf{P}$-a.s.
\end{cor}

\subsubsection{Filter stability and asymptotic coupling}
\label{sec:fswk}

The goal of this section is to develop a filter stability counterpart
of the weak-* ergodic theorem in Section~\ref{sec:ergwk}. For sake of
transparency, we will restrict attention to a special class of bivariate
Markov chains, known as hidden Markov models, that arise in many
settings (cf. Section~\ref{sec:examples}). While our method is
certainly also applicable in more general situations, the hidden Markov
assumption will allow us to state concrete and easily verifiable
conditions for weak-* filter stability.

A \emph{hidden Markov model} is a bivariate Markov chain
$(X_k,Y_k)_{k\ge0}$ (in the Polish state space $E\times F$)
whose transition kernel $P$ factorizes as
\[
P \bigl(x,y,dx',dy' \bigr) = P_0
\bigl(x,dx' \bigr) \Phi \bigl(x',dy'
\bigr) %
\]
for transition kernels $P_0\dvtx E\times\mathcal{B}(E)\to[0,1]$ and
$\Phi\dvtx E\times\mathcal{B}(F)\to[0,1]$. The special feature of such
models is that the unobserved process $(X_k)_{k\ge0}$ is a Markov chain
in its own right, and the observations $(Y_k)_{k\ge0}$ are
conditionally independent given $(X_k)_{k\ge0}$. This is a common
scenario when $(Y_k)_{k\ge0}$ represent noisy observations of an
underlying Markov chain $(X_k)_{k\ge0}$. In this setting, it is
natural to consider just initial conditions for $X_0$, rather than for the
pair $(X_0,Y_0)$. We therefore define $\mathbf{P}^x=
\mathbf{P}^{\delta_x\otimes\Phi(x,\cdot)}$ for $x\in E$ and
$\mathbf{P}^\mu=\int\mathbf{P}^x\mu(dx)$ for $\mu\in\mathcal{P}(E)$,
as well as the corresponding filters $\pi_n^\mu,\Pi_n^\mu$.
We will assume that $P_0$ admits an
invariant probability $\hat\lambda\in\mathcal{P}(E)$, so that
$\lambda=\hat\lambda\otimes\Phi$ is invariant for $(X_n,Y_n)_{n\ge0}$
[this entails no loss of generality if we assume, as we do, that
$(X_n,Y_n)_{n\ge0}$ admits an invariant probability].

A hidden Markov model is called nondegenerate if the observation kernel
$\Phi$ admits a positive density with respect to some reference measure
$\varphi$:
%
\begin{aspt}
\label{aspt:hmmnondeg}
A hidden Markov model is \emph{nondegenerate} if
\[
\Phi(x,dy) = g(x,y) \varphi(dy),\qquad g(x,y)>0\mbox{ for all }(x,y)\in E\times F
\]
for a $\sigma$-finite reference measure $\varphi$ on $F$.
\end{aspt}
As any $\sigma$-finite measure is equivalent to a probability measure,
nondegeneracy of the hidden Markov model evidently corresponds to the
validity of Assumption~\ref{aspt:nondeg} for the bivariate Markov
chain $(X_k,Y_k)_{k\ge0}$.

We can now state our weak-* stability result for the filter in the
hidden Markov model setting; compare with the weak-* ergodic Theorem~\ref{thmm:wk02}. In the following, we fix a complete metric $d$ for the
Polish space $E$. To allow for the case that the observation law is
discontinuous with respect to $d$ (see Section~\ref{sec:examples} for
examples), we introduce an auxiliary quantity $\tilde d$ that dominates
the metric $d$. Let us note that it is not necessary for $\tilde d$ to
be a metric.

\begin{thmm}[(Weak-* filter stability)]
\label{thmm:wfstab}
Let $(X_k,Y_k)_{k\ge0}$ be a hidden Markov model that admits an invariant
probability $\hat\lambda$ and satisfies Assumption~\ref{aspt:hmmnondeg}.
Let $\tilde d(x,y)\ge d(x,y)$ for all $x,y\in E$. Suppose the following hold:
\begin{longlist}[(a)]
\item[(a)]\textup{(Asymptotic coupling.)} There exists $\alpha>0$ such that
\[
\forall x,x'\in E, \exists \mathbf{Q}\in\mathcal{C} \bigl(
\mathbf{P}^x,\mathbf{P}^{x'} \bigr) \mbox{ s.t.} \qquad\mathbf{Q}
\Biggl[\sum_{n=1}^\infty\tilde d
\bigl(X_n,X_n' \bigr)^2<\infty
\Biggr]\ge\alpha. %
\]
\item[(b)]\textup{(Hellinger--Lipschitz observations.)} There exists $C<\infty$ such that
\[
\int \bigl\{{ \sqrt{g(x,y)}-\sqrt{g \bigl(x',y \bigr)}} \bigr
\}^2 \varphi(dy) \le C \tilde d \bigl(x,x'
\bigr)^2 \qquad\mbox{for all }x,x'\in E. %
\]
\end{longlist}
Then the filter is stable in the sense that
\[
\bigl|\pi_n^\mu(f)-\pi_n^\nu(f)\bigr|
\mathop{\longrightarrow}^{n\to\infty}0,\qquad \mathbf{P}^\gamma\mbox{-a.s.}
\mbox{ for all }f
\in\mathrm{Lip}(E), \mu,\nu,\gamma\in\mathcal{P}(E). %
\]
In particular, we obtain
\[
\bigl\|\pi_n^\mu-\pi_n^\nu
\bigr\|_{\mathrm{BL}} \mathop{\longrightarrow}^{n\to\infty}0 \qquad\mbox{in }\mathbf{P}^\gamma
\mbox{-probability} \mbox{ for all } \mu,\nu,\gamma\in\mathcal{P}(E). %
\]
\end{thmm}

\begin{pf}
The proof is similar to that of Theorem~\ref{thmm:wk02}.
Let $(\xi_n)_{n\ge0}$ be an i.i.d. sequence of standard
Gaussian random variables independent of $(X_n,Y_n)_{n\ge0}$,
so that we may consider in the following the extended Markov chain
$(X_n,\xi_n,Y_n)_{n\ge0}$. Fix $f\in\mathrm{Lip}(E)$, and
define $F_n=f(X_n)+\xi_n$. Conditionally on $\mathcal{F}^X_+$,
the process $(F_n,Y_n)_{n\ge0}$ is an independent sequence
with
\begin{eqnarray}
&&\mathbf{P}^\mu \bigl[(F_n,Y_n)_{n\ge k}
\in A|\mathcal{F}^X_+ \bigr]\nonumber\\
&&\qquad= \mathbf{R}^{X_{k,\infty}}(A):=
\int\mathbf{1}_A(r,y) \prod_{n=k}^\infty
\frac{e^{-(r_n-f(X_n))^2/2}}{\sqrt{2\pi}} \,dr_n g(X_n,y_n)
\varphi(dy_n), \nonumber\\
\eqntext{A\in\mathcal{B}(\mathbb{R}\times F)^{\otimes\mathbb{N}}}
\end{eqnarray}
for all $\mu\in\mathcal{P}(E)$.
We can now estimate as in
the proof of Lemma~\ref{lem:hell}
\begin{eqnarray*}
\bigl\|\mathbf{R}^{x_{0,\infty}}-\mathbf{R}^{x'_{0,\infty}}\bigr\|^2 &\le& \sum
_{n=0}^\infty \bigl[2 \bigl
\{f(x_n)-f \bigl(x_n' \bigr) \bigr
\}^2 + 8C \tilde d \bigl(x_n,x_n'
\bigr)^2 \bigr]
\\
&\le& (8C+2) \sum_{n=0}^\infty\tilde d
\bigl(x_n,x_n' \bigr)^2,
\end{eqnarray*}
where we have used that $1-\prod_n(1-p_n)\le\sum_np_n$
when $0\le p_n\le1$ for all $n$.
Therefore, proceeding exactly as in the
proof of Theorem~\ref{thmm:wk02}, we find that for every $x,x'\in E$,
there exists $n\ge1$ such that we have
\[
\bigl\| \mathbf{P}^x[F_{n,\infty},Y_{n,\infty}\in\cdot ]-
\mathbf{P}^{x'}[F_{n,\infty},Y_{n,\infty}\in\cdot ]\bigr \|\le2-
\alpha. %
\]
Now note that the law
$\mathbf{P}^x[F_{n,\infty},Y_{n,\infty}\in\cdot ]$ does not
change if we condition additionally on $\xi_0,Y_0$ (as
$\xi_0,Y_0$ are independent of $X_{1,\infty},\xi_{1,\infty},Y_{1,\infty}$
under $\mathbf{P}^x$).
We can therefore apply Theorem~\ref{thmm:loc02} to conclude that
\[
\bigl\| \mathbf{P}^\mu[F_{n,\infty},Y_{n,\infty}\in\cdot ]-
\mathbf{P}^\nu[F_{n,\infty},Y_{n,\infty}\in\cdot ]\bigr \|
\mathop{\longrightarrow}^{n\to\infty}0 \qquad\mbox{for all }\mu,\nu \in\mathcal{P}(E). %
\]
Moreover, note that $\mathbf{P}^\mu[Y_0\in\cdot ]\sim\varphi$ for all
$\mu\in\mathcal{P}(E)$. Thus, we can apply Corollary~\ref{cor:lfverystab}
to conclude [here $\xi\in\mathcal{P}(\mathbb{R})$
is the standard Gaussian measure]
\[
\bigl\|\pi_n^\mu f^{-1} * \xi-\pi_n
f^{-1} * \xi\bigr\| \mathop{\longrightarrow}^{n\to\infty}0,\qquad\mathbf{P}^\mu
\mbox{-a.s. and }\mathbf{P}\mbox{-a.s.} %
\]
Applying the argument in the proof of Theorem~\ref{thmm:wk02} pathwise,
we obtain
\[
\bigl|\pi_n^\mu(f)-\pi_n(f)\bigr| \mathop{\longrightarrow}^{n\to
\infty}0,\qquad\mathbf{P}^\mu\mbox{-a.s. and }\mathbf{P}\mbox{-a.s.}
\]
Thus, by the triangle inequality, we have
\[
\bigl|\pi_n^\mu(f)-\pi_n^\nu(f)\bigr| \le
\bigl|\pi_n^\mu(f)-\pi_n(f)\bigr| + \bigl|
\pi_n^\nu(f)-\pi_n(f)\bigr| \mathop{\longrightarrow}^{n\to
\infty}0,\qquad\mathbf{P}\mbox{-a.s.} %
\]
Finally, note that the assumptions of Corollary~\ref{cor:oequiv}
are satisfied for any initial measure, so that
$\mathbf{P}^\gamma|_{\mathcal{F}^Y_+}\ll
\mathbf{P}|_{\mathcal{F}^Y_+}$ for any $\gamma\in\mathcal{P}(E)$.
Thus, the above $\mathbf{P}$-a.s. convergence also holds $\mathbf
{P}^\gamma$-a.s.,
which yields the first conclusion.

To obtain the second conclusion, we argue as in the proof of
Theorem~\ref{thmm:swk02}.
Fix $\varepsilon>0$. Let $K\subseteq E$ be a compact set such that
$\lambda(K)\ge1-\varepsilon$, and define
$\chi(x)=(1-\varepsilon^{-1}d(x,K))_+$.
Following the argument used in the proof of Theorem~\ref{thmm:swk02},
we can find functions $f_1,\ldots,f_k\in\mathrm{Lip}(E)$ such that
\[
\bigl\|\pi_n^\mu-\pi_n\bigr\|_{\mathrm{BL}} \le \max
_{i=1,\ldots,k} \bigl|\pi_n^\mu(f_i\chi)-
\pi_n(f_i\chi)\bigr| + \bigl|\pi_n^\mu(
\chi)-\pi_n(\chi)\bigr| + 2\pi_n \bigl(K^c \bigr) +
6 \varepsilon %
\]
for all $n\ge0$. Taking the expectation and letting $n\to\infty$, we obtain
\[
\limsup_{n\to\infty} \mathbf{E} \bigl[\bigl\|\pi_n^\mu-
\pi_n\bigr\|_{\mathrm{BL}} \bigr] \le8\varepsilon. %
\]
As $\varepsilon>0$ is arbitrary, this implies that
\[
\bigl\|\pi_n^\mu-\pi_n\bigr\|_{\mathrm{BL}}
\mathop{\longrightarrow}^{n\to\infty}0 \qquad\mbox{in }\mathbf{P}\mbox{-probability}. %
\]
But $\mathbf{P}^\gamma|_{\mathcal{F}^Y_+}\ll
\mathbf{P}|_{\mathcal{F}^Y_+}$, so the convergence is also in
$\mathbf{P}^\gamma$-probability. Applying the triangle
inequality and dominated convergence completes the proof.
\end{pf}

\begin{rem}
In Theorem~\ref{thmm:wfstab}, we obtain a.s. stability of
the filter for individual Lipschitz functions, but only stability in
probability for the $\|\cdot\|_{\mathrm{BL}}$ norm. It is not clear whether
the latter could be improved to a.s. convergence (except when $E$ is
compact, in which case the compactness argument used in the proof above
directly yields a.s. convergence). The problem is that we do not know
whether the null set in the a.s. stability result can be made
independent of the choice of Lipschitz function; if this were the case,
the method used in Theorem~\ref{thmm:wk02} could be used to obtain
a.s. convergence.
\end{rem}

\subsubsection{Ergodicity of the filter}
\label{sec:kunita}

We developed above a number of filter stability results that ensure
convergence of conditional expectations of the form
$|\mathbf{E}^\mu[f(X_n)|\mathcal{F}^Y_{0,n}]-
\mathbf{E}[f(X_n)|\mathcal{F}^Y_{0,n}]|\to0$. This can evidently be
viewed as a natural conditional counterpart to the classical ergodic
theory of Markov chains, which ensures that $|\mathbf{E}^\mu[f(X_n)]-
\mathbf{E}[f(X_n)]|\to0$. In this section, following Kunita
\cite{Kun71}, we develop a different ergodic property of the filter.

It is well known---and a simple exercise using the Bayes formula---that
the filter $\pi_n^\mu$ can be computed in a recursive fashion.
In particular, under the nondegeneracy Assumption~\ref{aspt:nondeg}, we
have $\pi_{n+1}^\mu=U(\pi_n^\mu,Y_n,Y_{n+1})$ with
\[
U \bigl(\nu,y,y' \bigr) (A) = \frac{\int\mathbf{1}_A(x')
g(x,y,x',y') P_0(x,dx') \nu(dx)}{
\int g(x,y,x',y') P_0(x,dx') \nu(dx)}. %
\]
Let $H\dvtx \mathcal{P}(E)\times F\to\mathbb{R}$ be a bounded measurable function.
Then
\[
\mathbf{E}^\mu \bigl[H \bigl(\pi_{n+1}^\mu,Y_{n+1}
\bigr)|\mathcal{F}^Y_{0,n} \bigr] = \mathbf{E}^\mu
\bigl[H \bigl(U \bigl(\pi_n^\mu,Y_n,Y_{n+1}
\bigr),Y_{n+1} \bigr)|\mathcal {F}^Y_{0,n} \bigr]
= \mathsf{\Gamma}H \bigl(\pi_n^\mu,Y_n
\bigr), %
\]
where the kernel
$\mathsf{\Gamma}\dvtx \mathcal{P}(E)\times F\times\mathcal{B}(\mathcal
{P}(E)\times F)\to[0,1]$
is given by
\[
\mathsf{\Gamma}(\nu,y,A) = \int \mathbf{1}_A \bigl(U \bigl(
\nu,y,y' \bigr),y' \bigr) P \bigl(x,y,dx',dy'
\bigr) \nu(dx). %
\]
Thus, we see that the process $(\pi_n^\mu,Y_n)_{n\ge0}$ is itself a
$(\mathcal{P}(E)\times F)$-valued Markov chain under $\mathbf{P}^\mu$
with transition kernel $\mathsf{\Gamma}$.
In the hidden Markov model setting of Section~\ref{sec:fswk},
$\mathsf{\Gamma}(\nu,y,A)$ does not depend on $y$, so that in this special
case even the filter $(\pi_n^\mu)_{n\ge0}$ itself is a $\mathcal{P}(E)$-valued
Markov chain.

In view of this Markov property of the filter, it is now natural to
ask about the ergodic properties of the filtering process itself.
Generally speaking, we would like to know whether the ergodic properties
of the underlying Markov chain $(X_n,Y_n)_{n\ge0}$ with transition\vspace*{1pt} kernel
$P$ are ``lifted'' to the measure-valued Markov chain
$(\pi_n^\mu,Y_n)_{n\ge0}$ with transition kernel $\mathsf{\Gamma}$.
Following the classical work of Kunita \cite{Kun71}, such questions
have been considered by a number of authors
\cite{Ste89,Bud03,CvH10,vH12,TvH12}. We will focus here on the question
of unique ergodicity, where the following result
(essentially due to Kunita) is known.

\begin{thmm}
\label{thmm:kunita}
Suppose that Assumption~\ref{aspt:nondeg} holds, and that the transition
kernel $P$ admits a unique invariant measure $\lambda\in\mathcal
{P}(E\times F)$.
Then the transition kernel $\mathsf{\Gamma}$ admits a unique invariant
measure $\Lambda\in\mathcal{P}(\mathcal{P}(E)\times F)$ iff
\[
\bigcap_{n\ge0} \mathcal{F}^Y_-\vee
\mathcal{F}^X_{-\infty,-n} = \mathcal{F}^Y_- \qquad\mathop{
\mathrm{mod}}\mathbf{P}. %
\]
\end{thmm}

We refer to \cite{vH12} for a full proof in the hidden Markov model
setting, which is easily adapted to the more general setting considered
here (sufficiency is also shown in our setting in the proof of
\cite{TvH12}, Theorem~2.12).

To prove unique ergodicity of the filter, we must therefore establish
the measure-theoretic identity in Theorem~\ref{thmm:kunita}. The goal
of this section is to accomplish this task under the same assumptions we
have used for filter stability: local mixing or asymptotic couplings.
In fact, slightly weaker forms of the assumptions of Corollary~\ref{cor:flocmix} or Theorem~\ref{thmm:wfstab} will suffice. We refer
to Sections \ref{sec:fslm} and \ref{sec:fswk} for the notation
used in the following results.

\begin{thmm}
\label{thmm:ferglmix}
Suppose $(X_n,Y_n)_{n\ge0}$ with $E\subseteq\prod_{i\in I}E^i$
satisfies Assumption~\ref{aspt:nondeg}, admits a unique invariant measure
and is a.e. locally mixing:
\[
\bigl\|\mathbf{P}^{x,y}-\mathbf{P}\bigr\|_{\mathcal{F}^J_{n,\infty}} \mathop{\longrightarrow}^{n\to
\infty}0 \qquad\mbox{for }\lambda\mbox{-a.e. }(x,y) \mbox{ and all } J\subseteq I, |J|<
\infty. %
\]
Then the filter transition kernel $\mathsf{\Gamma}$ admits a unique
invariant measure.
\end{thmm}

\begin{thmm}
\label{thmm:fergwk}
Let $(X_k,Y_k)_{k\ge0}$ be a hidden Markov model that admits a unique
invariant probability $\hat\lambda$ and that satisfies Assumption~\ref{aspt:hmmnondeg}. Moreover, let $\tilde d(x,y)\ge d(x,y)$
for all $x,y\in E$, and suppose that the following hold:
\begin{longlist}[(a)]
\item[(a)]\textup{(Asymptotic coupling.)} For $\hat\lambda\otimes\hat\lambda$-a.e. $(x,x')\in E\times E$,
\[
\exists \mathbf{Q}\in\mathcal{C} \bigl(\mathbf{P}^x,
\mathbf{P}^{x'} \bigr) \mbox{ such that}\qquad \mathbf{Q} \Biggl[\sum
_{n=1}^\infty\tilde d \bigl(X_n,X_n'
\bigr)^2<\infty \Biggr]>0. %
\]
\item[(b)]\textup{(Hellinger--Lipschitz observations.)} There exists $C<\infty$ such that
\[
\int \bigl\{{ \sqrt{g(x,y)}-\sqrt{g \bigl(x',y \bigr)}} \bigr
\}^2 \varphi(dy) \le C \tilde d \bigl(x,x'
\bigr)^2 \qquad\mbox{for all }x,x'\in E. %
\]
\end{longlist}
Then the filter transition kernel $\mathsf{\Gamma}$ admits a unique
invariant measure.
\end{thmm}

The remainder of this section is devoted to the proof of these results.
We must begin by obtaining a local result in the setting of Section~\ref{sec:flocal}. To this end, we will need the following structural
result for the invariant measure $\lambda$.

\begin{lem}
\label{lem:tvh}
Suppose that Assumptions \ref{aspt:locerg} and \ref{aspt:nondeg} hold.
Then the invariant measure $\lambda$ satisfies
$\lambda(dx,dy)\sim\lambda(dx\times F)\otimes
\lambda(E\times dy)$.
\end{lem}

\begin{pf}
Assumption~\ref{aspt:locerg} and Jensen's inequality yield
$\mathbf{P}$-a.s.
\[
\mathbf{E} \bigl[\bigl\|\mathbf{P}[Y_n\in\cdot |Y_0]-
\lambda(E \times\cdot)\bigr\| \bigr] \le \mathbf{E} \bigl[ \bigl\|\mathbf{P}^{X_0,Y_0}[Y_n
\in\cdot ]-\lambda(E\times\cdot)\bigr\| \bigr] \mathop{\longrightarrow}^{n\to\infty}0. %
\]
The result follows by applying \cite{TvH12}, Proposition 3.3, to the
process $(Y_n,X_n)_{n\ge0}$.
\end{pf}

We can now prove the following local result.

\begin{cor}
\label{cor:lockun}
Suppose that Assumptions \ref{aspt:locerg} and \ref{aspt:nondeg} hold.
Then
\[
\mathbf{P} \biggl[A \Big| \bigcap_{n\ge0}
\mathcal{F}^Y_-\vee \mathcal{F}^X_{-\infty,-n} \biggr]
= \mathbf{P} \bigl[A|\mathcal{F}^Y_- \bigr],\qquad\mathbf{P}\mbox{-a.s.}
\mbox{ for every }A\in\mathcal{F}^0_{-\infty,\infty}. %
\]
\end{cor}

\begin{pf}
By Lemma~\ref{lem:tvh}, we have $\lambda(dx,dy) = h(x,y) \lambda
(dx\times F)
\lambda(E\times dy)$ for a strictly positive measurable function $h$.
Define the kernel
\[
\lambda_x(A) = \frac{\int\mathbf{1}_A(y) h(x,y) \lambda(E\times dy)}{
\int h(x,y) \lambda(E\times dy)}. %
\]
By the Bayes formula, $\lambda_{X_0}=\mathbf{P}[Y_0\in\cdot |X_0]$.
Assumption~\ref{aspt:locerg}, disintegration and Jensen's inequality
yield a set $H\subseteq E$ with $\lambda(H\times F)=1$ such that
\[
\mathbf{E}^{\delta_x\otimes\lambda_x} \bigl[ \bigl\|\mathbf{P}^{x,Y_0}-\mathbf{P}
\bigr\|_{\mathcal{F}^0_{n,\infty}} \bigr] \mathop{\longrightarrow}^{n\to\infty}0
\quad\mbox{and}\quad
\mathbf{E}^{\delta_x\otimes\lambda_x}\bigl[ \|\Pi_0-\mathbf{P}\|_{\mathcal{F}^0_{n,\infty}}\bigr]
\mathop{\longrightarrow}^{n\to\infty}0 %
\]
for all $x\in H$.
But note that $\Pi_0^{\delta_x\otimes\lambda_x}=
\mathbf{P}^{x,Y_0}$ holds $\mathbf{P}^{\delta_x\otimes\lambda_x}$-a.s. Thus,
\[
\lambda_x\sim\lambda(E\times\cdot) \quad\mbox{and}\quad\bigl \|
\Pi_0^{\delta_x\otimes\lambda_x}-\Pi_0\bigr\|_{\mathcal{F}^0_{n,\infty}}
\mathop{\longrightarrow}^{n\to\infty}0,\qquad \mathbf{P}^{\delta_x\otimes\lambda_x}\mbox{-a.s.} %
\]
for every $x\in H$ by the definition of $\lambda_x$ and the triangle
inequality (we have used that $\|\Pi_0^{\delta_x\otimes\lambda_x}-
\Pi_0\|_{\mathcal{F}^0_{n,\infty}}$
is pointwise decreasing to establish a.s. convergence).
Therefore, by Theorem~\ref{thmm:lfstab}, we obtain for every $x\in H$
\[
\bigl\|\Pi_n^{\delta_x\otimes\lambda_x}-\Pi_n\bigr\|_{
\mathcal{F}^0_{n-r,\infty}}
\mathop{\longrightarrow}^{n\to\infty}0,\qquad\mathbf{P}^{\delta_x\otimes\lambda_x}\mbox{-a.s.} \mbox{ for any }r
\in\mathbb{N}. %
\]
Now note that $\Pi_n^{\delta_{X_0}\otimes\lambda_{X_0}}(A)=
\mathbf{P}[A|\sigma{X_0}\vee\mathcal{F}^Y_{0,n}]$ for all $A\in\mathcal{F}_+$,
for example, by Lemma~\ref{lem:weiz}. It follows that for any
$A\in\mathcal{F}^0_{0,\infty}$, we have the convergence
\[
\mathbf{E} \bigl[ \bigl|\mathbf{E} \bigl[\mathbf{1}_A\circ
\Theta^{n-r}|\sigma{X_0}\vee\mathcal{F}^Y_{0,n}
\bigr] - \mathbf{E} \bigl[\mathbf{1}_A\circ\Theta^{n-r}|
\mathcal{F}^Y_{0,n} \bigr]\bigr| \bigr]\mathop{\longrightarrow}^{n\to\infty}0.
\]
But the Markov property of $(X_n,Y_n)_{n\in\mathbb{Z}}$ yields
$\mathbf{E}[\mathbf{1}_A\circ\Theta^{n-r}|\sigma{X_0}\vee\mathcal{F}^Y_{0,n}]=
\mathbf{E}[\mathbf{1}_A\circ\Theta^{n-r}|\mathcal{F}^X_{-\infty,0}\vee
\mathcal{F}^Y_{-\infty,n}]$. Therefore, using stationarity, we obtain
\[
\mathbf{E} \bigl[\bigl |\mathbf{E} \bigl[\mathbf{1}_A\circ
\Theta^{-r}|\mathcal{F}^Y_-\vee \mathcal{F}^X_{-\infty,-n}
\bigr] - \mathbf{E} \bigl[\mathbf{1}_A\circ\Theta^{-r}|
\mathcal{F}^Y_{-n,0} \bigr]\bigr| \bigr]\mathop{\longrightarrow}^{n\to\infty}0.
\]
The lemma now follows by the martingale convergence
theorem for $A\in\mathcal{F}^0_{-r,\infty}$. As
$r$ is arbitrary, a monotone
class argument concludes the proof.
\end{pf}

The proof of Theorem~\ref{thmm:ferglmix} is now essentially trivial.

\begin{pf*}{Proof of Theorem~\ref{thmm:ferglmix}}
By Corollary~\ref{cor:lockun}, we have
\[
\mathbf{P} \biggl[A \Big| \bigcap_{n\ge0}
\mathcal{F}^Y_-\vee \mathcal{F}^X_{-\infty,-n} \biggr]
= \mathbf{P} \bigl[A|\mathcal{F}^Y_- \bigr],\qquad \mathbf{P}\mbox{-a.s.}
\mbox{ for every }A\in\mathcal{F}^J_{-\infty,\infty} %
\]
whenever $J\subseteq I$, $|J|<\infty$. A monotone class argument yields
the conclusion for all $A\in\mathcal{F}$. Thus the $\sigma$-field
identity of Theorem~\ref{thmm:kunita} holds.
\end{pf*}

We now turn to the proof of Theorem~\ref{thmm:fergwk}.

\begin{pf*}{Proof of Theorem~\ref{thmm:fergwk}}
Proceeding precisely as in the proof of Theorem~\ref{thmm:wfstab}
(and adopting the same notation as is used there), we
find that for $\hat\lambda\otimes\hat\lambda$-a.e. $(x,x')\in E\times E$, there exists $n\ge1$ such that
\[
\bigl\| \mathbf{P}^x[F_{n,\infty},Y_{n,\infty}\in\cdot ]-
\mathbf{P}^{x'}[F_{n,\infty},Y_{n,\infty}\in\cdot ]\bigr \|<2.
\]
From Theorem~\ref{thmm:sloc02}, it follows that
\[
\bigl\| \mathbf{P}^{x,y}[F_{n,\infty},Y_{n,\infty}\in\cdot ]-
\mathbf{P}[F_{n,\infty},Y_{n,\infty}\in\cdot ] \bigr\|\mathop{\longrightarrow}^{n\to
\infty}0,\qquad\lambda\mbox{-a.e. }(x,y)\in E\times F. %
\]
Applying Corollary~\ref{cor:lockun}, we find that
\[
\mathbf{P} \biggl[A \Big| \bigcap_{n\ge0}
\mathcal{F}^Y_-\vee \mathcal{F}^{X,\xi}_{-\infty,-n} \biggr]
= \mathbf{P} \bigl[A|\mathcal{F}^Y_- \bigr],\qquad\mathbf{P}\mbox{-a.s.}
\mbox{ for every }A\in\mathcal{F}^{F,Y}_{-\infty,\infty}, %
\]
where $\mathcal{F}^{X,\xi}_{m,n}=\sigma\{X_{m,n},\xi_{m,n}\}$ and
$\mathcal{F}^{F,Y}_{m,n}=\sigma\{F_{m,n},Y_{m,n}\}$. In particular,
if
\[
G=g \bigl(f(X_{-m})+\xi_{-m},\ldots,f(X_m)+
\xi_m,Y_{-m},\ldots,Y_m \bigr) %
\]
for some bounded continuous function $g\dvtx \mathbb{R}^{2m+1}\times
F^{2m+1}\to\mathbb{R}$, then
\[
\mathbf{E} \biggl[G\Big | \bigcap_{n\ge0}
\mathcal{F}^Y_-\vee \mathcal{F}^{X}_{-\infty,-n} \biggr]
= \mathbf{E} \bigl[G|\mathcal{F}^Y_- \bigr],\qquad\mathbf{P}\mbox{-a.s.}
\]
as $\sigma\{\xi_{-\infty,-n}\}$ is independent of
$\mathcal{F}^X\vee\mathcal{F}^Y\vee\sigma\{\xi_{-m,m}\}$ for
$n\ge m$. Now note that nothing in the proof relied on the fact
that $\xi_k$ are Gaussian with unit variance; we can replace
$\xi_k$ by $\varepsilon\xi_k$ for any $\varepsilon>0$ and attain
the same conclusion. Letting $\varepsilon\to0$, we find that
for any $f\in\mathrm{Lip}(E)$, $m\ge0$, and bounded continuous
$g\dvtx \mathbb{R}^{2m+1}\times F^{2m+1}\to\mathbb{R}$,
the above identity holds for
\[
G = g \bigl(f(X_{-m}),\ldots,f(X_m),Y_{-m},
\ldots,Y_m \bigr). %
\]
The remainder of the proof is a routine approximation argument.
As $E$ is Polish, we can choose a countable dense subset $E'\subset E$.
Then the countable family of open balls
$\{B(x,\delta)\dvtx x\in E',\delta\in\mathbb{Q}_+\}$ [where
$B(x,\delta)$ is the open ball with center $x$ and radius $\delta$]
generate the Borel $\sigma$-field $\mathcal{B}(E)$. Arrange these
open balls arbitrarily as a sequence $(B_k)_{k\ge1}$, and define the
functions
\[
\iota(x) = \sum_{k=1}^\infty
\frac{\mathbf{1}_{B_k}(x)}{3^k}, \qquad\iota_r^\delta(x) = \sum
_{k=1}^r \frac{\delta^{-1}d(x,B_k^c)\wedge1}{3^k}\qquad (x\in E). %
\]
Then $\iota_r^\delta$ is bounded and Lipschitz for every
$r,\delta$, and $\iota_r^\delta\to\iota$ as $\delta\downarrow0$,
$r\uparrow\infty$. Choosing $f=\iota_r^\delta$
and taking limits, we obtain the above identity for
\[
G = g \bigl(\iota(X_{-m}),\ldots,\iota(X_m),Y_{-m},
\ldots,Y_m \bigr) %
\]
for any $m\ge0$ and bounded continuous
$g\dvtx \mathbb{R}^{2m+1}\times F^{2m+1}\to\mathbb{R}$. A monotone
class argument shows that we may choose $G$ to be any bounded
$\sigma\{\iota(X_k),\break Y_k\dvtx k\in\mathbb{Z}\}$-measurable function.
But $\mathcal{B}(E)=\sigma\{\iota\}$, so the proof
is complete.
\end{pf*}

\subsection{Continuous time}
\label{sec:conttime}

Up to this point, we have considered only discrete-time processes and
Markov chains. However, continuous time processes are of equal interest
in many applications: indeed, most of the examples that we will
consider in
Section~\ref{sec:examples} will be in continuous time.
The goal of this section is to extend our main filter stability
results to the continuous time setting.

In principle, we can view continuous time processes as a special case
of the discrete time setting.
If $(x_t,y_t)_{t\ge0}$ is a continuous time
Markov process with c\`adl\`ag paths, then we can define
the associated discrete-time Markov chain $(X_n,Y_n)_{n\ge0}$
with values in the Skorokhod space $D([0,1]; E\times F)$ by
setting $X_n=(x_t)_{t\in[n,n+1]}$ and
$Y_n=(y_t)_{t\in[n,n+1]}$. When we are dealing with the
(unconditional) ergodic theory of $(x_t,y_t)_{t\ge0}$, we
can obtain continuous-time ergodic results directly from the
corresponding results for the discrete-time chain
$(X_n,Y_n)_{n\ge0}$. However, in the conditional setting, two
issues arise.

First, note that in the unconditional setting, the marginal law
$\mathbf{P}[x_t\in\cdot ]$ for $t\in[n,n+1]$ is a
coordinate projection of $\mathbf{P}[X_n\in\cdot ]$. However, this is
not true for the filter: the projection of
$\mathbf{P}[X_n\in\cdot |\mathcal{F}^Y_{0,n}]$ gives
$\mathbf{P}[x_t\in\cdot |(y_s)_{s\in[0,n+1]}]$,
not the continuous time filter
$\pi_t=\mathbf{P}[x_t\in\cdot |(y_s)_{s\in[0,t]}]$.
We must therefore get rid of
the additional observation segment $(y_s)_{s\in[t, n+1]}$
that appears in the projection. This is precisely what will be done in
this section.

Second, in continuous time, numerous subtleties arise in defining the
filtering process $(\pi_t)_{t\ge0}$ as a stochastic process with
sufficiently regular sample paths. Such issues would have to be dealt
with carefully if we wanted to obtain, for example, almost sure filter
stability results in the continuous time setting. The structure of
nonlinear filters in continuous time is a classical topic in stochastic
analysis (see, e.g., \cite{LS01,Yor77,Kur98}) that provides the
necessary tools to address such problems. However, in the present
setting, such regularity issues are purely technical in nature and do not
introduce any new ideas in the ergodic theory of nonlinear filters.
We therefore choose to circumvent these issues by considering only
stability in probability in the continuous time setting, in which case
regularity issues can be avoided.

A final issue that arises in the continuous time setting is that, unlike
in discrete time, one must make a distinction between general bivariate
Markov processes and hidden Markov processes, as we will
presently explain.

Recall that a discrete-time
hidden Markov model is defined by the fact that $(X_n)_{n\ge0}$
is itself Markov and $(Y_n)_{n\ge0}$ are conditionally independent given
$(X_n)_{n\ge0}$. In continuous time, we cannot assign a
(conditionally) independent random variable to every time
$t\in\mathbb{R}_+$. Instead, we consider an integrated form of the
observations where $(x_t)_{t\ge0}$ is a Markov process
and $(y_t)_{t\ge0}$ has conditionally independent increments
given $(x_t)_{t\ge0}$. This is known as a
Markov additive process~\cite{Cin72}, and constitutes the natural
continuous-time counterpart to a hidden Markov model~\cite{Yor77}. For example, the most common observation model in
continuous time is the ``white noise'' model \cite{LS01}
\[
y_t = \int_0^t
h(x_s) \,ds + W_t, %
\]
where $(W_t)_{t\ge0}$ is a Brownian motion independent of
$(x_t)_{t\ge0}$. Formally, $dy_t/dt$ represents the
observation of $h(x_t)$ corrupted by white noise, but the
integrated form is used to define a mathematically sensible model.
In this example, the pair $(x_t,y_t)_{t\ge0}$ is evidently
a Markov additive process.

In principle, a continuous-time hidden Markov process is a special case
of a bivariate Markov process as in the discrete time setting.
Unfortunately, as $y_t$ is an additive process, it cannot be
positive recurrent except in trivial cases, so the pair
$(x_t,y_t)_{t\ge0}$
does not admit an invariant probability. We
must therefore take care to utilize explicitly the fact that it is the
increments of $y_t$, and not $y_t$ itself, that will be
stationary under the invariant distribution. This does not introduce
any complications into our theory: both the bivariate Markov setting and
the Markov additive setting can be treated in exactly the same manner.
However, two distinct sets of notation are required for these two
settings. In order to avoid notational confusion, we will develop our
continuous time results below in the hidden Markov process setting only
(all our examples in Section~\ref{sec:examples} will be of this form).
The same approach can however be adapted to the bivariate Markov setting
with minimal effort.

\subsubsection{The continuous time setting}

In the remainder of this section, we consider a continuous-time
process $(x_t,y_t)_{t\ge0}$ with
c\`adl\`ag paths, where $x_t$ takes values in a Polish space
$E$ and $y_t$ takes values in a Polish topological
vector space $F$.
We realize this process on the canonical path space
$\Omega=D(\mathbb{R}_+;E\times F)$ endowed with its
Borel $\sigma$-field $\mathcal{F}$, such that
$x_t(\xi,\eta)=\xi(t)$ and $y_t(\xi,\eta)=\eta(t)$.
We define for $s\le t$ the
$D([0,t-s];E)$-valued random variable $x_{s,t}=(x_r)_{r\in[s,t]}$ and the
$\sigma$-field $\mathcal{F}^{x}_{s,t}
=\sigma\{x_{s,t}\}$.
Moreover, we define the $D([0,t-s];F)$-valued
random variable $y_{s,t}$ and corresponding $\sigma$-fields
\[
y_{s,t}=(y_r-y_s)_{r\in[s,t]},\qquad
\mathcal{F}^{y}_{s,t}=\sigma\{y_{s,t}\},\qquad
\mathcal{F}_{s,t}= \mathcal{F}^{x}_{s,t}\vee
\mathcal{F}^{y}_{s,t}. %
\]
The shift $\Theta^t\dvtx \Omega\to\Omega$ is defined as
$\Theta^t(\xi,\eta)(s)=(\xi(s+t),\eta(s+t)-\eta(t))$. Let us emphasize that
the observation segment $y_{s,t}$ and the shift $\Theta^t$
are defined differently than in the discrete time setting: the present
choice accounts for the additivity of the observations, which we introduce
next.

In the continuous time setting, we will assume that the canonical
process is a hidden Markov process or
\emph{Markov additive process}: that is, $(x_t,y_t)_{t\ge0}$ is
a time-homogeneous Markov process such that
$\mathbf{E}[f(x_t,y_t-y_0)|x_0,y_0]$
does not depend on $y_0$ for any bounded measurable function $f$.
It is not difficult to verify that this assumption corresponds to the
following two properties: the process $(x_t)_{t\ge0}$ is Markov in
its own right, and the process $(y_t)_{t\ge0}$ has conditionally
independent increments given $(x_t)_{t\ge0}$ (see, e.g.,
\cite{Cin72}).

In the following, we define the probability $\mathbf{P}^{x}$
on $\mathcal{F}_{0,\infty}$ as the law of the Markov additive
process $(x_t,y_t-y_0)_{t\ge0}$ started at $x_0=x\in E$, and let
$\mathbf{P}^\mu=\int\mathbf{P}^{x}
\mu(dx)$ for $\mu\in\mathcal{P}(E)$.
We will assume the existence of an invariant probability $\lambda\in
\mathcal{P}(E)$ so that $\mathbf{P}^{\lambda}$ is invariant
under the shift $\Theta^t$ for all $t\ge0$. We define
$\mathbf{P}=\mathbf{P}^{\lambda}$, and introduce
the continuous-time nonlinear filters
\[
\pi_t^\mu=\mathbf{P}^\mu
\bigl[x_t\in\cdot | \mathcal{F}^{y}_{0,t}
\bigr],\qquad \pi_t=\mathbf{P} \bigl[x_t\in\cdot |
\mathcal{F}^{y}_{0,t} \bigr]. %
\]
As we will consider convergence in probability only, we will not worry about
the regularity of $\pi_t^\mu$ as a function of $t$ (i.e., for each
$t\ge0$, we may choose any version of the above regular conditional
probabilities).

The Markov additive process $(x_t,y_t)_{t\ge0}$ is said
to be \emph{nondegenerate} if for every $\delta\in\mbox{}]0,\infty[$, there
exists a $\sigma$-finite reference measure $\varphi_\delta$ on
$D([0,\delta];F)$ and a strictly positive function
$g_\delta:D([0,\delta];E\times F)\to\mbox{}]0,\infty[$ such that
\[
\mathbf{P}^{z} \bigl[y_{t,t+\delta}\in A| \mathcal{F}^{x}_{0,\infty}
\bigr] = \int\mathbf{1}_A(\eta) g_\delta(x_{t,t+\delta},
\eta) \varphi_\delta(d\eta),\qquad\mathbf{P}^{z}\mbox{-a.s.}
\]
for all $t\ge0$, $A\in\mathcal{B}(D([0,\delta];F))$ and $z\in E$.
This assumption is the direct counterpart of nondegeneracy
for discrete time hidden Markov models.

\subsubsection{Local mixing in continuous time}

The aim of this section is to obtain a continuous-time version of
Corollary~\ref{cor:flocmix} (in the setting of hidden Markov processes).
To this end, we assume that the state space $E$ of the unobserved
process $x_t$ is contained in a countable product
$E\subseteq\prod_{i\in I}E^i$, where each $E^i$ is Polish.
Let $x_t^J$ be the projection of $x_t$ on
$\prod_{i\in J}E^i$ and
\[
\mathcal{F}^J_{s,t}=\sigma \bigl\{x^J_{s,t},
y_{s,t} \bigr\} \qquad\mbox{for }J\subseteq I, s\le t. %
\]
Let $\mathcal{E}^J\subseteq\mathcal{B}(E)$ be the
cylinder $\sigma$-field generated by the coordinates $J\subseteq I$.

\begin{thmm}[(Continuous local mixing filter stability)]
\label{thmm:clmfs}
If the Markov additive process $(x_t,y_t)_{t\ge0}$
is nondegenerate and locally mixing in the sense
\[
\bigl\|\mathbf{P}^{x}-\mathbf{P}\bigr\|_{\mathcal{F}^J_{t,\infty}} \mathop{\longrightarrow}^{t\to
\infty}0 \qquad\mbox{for all }x\in E \mbox{ and }J\subseteq I, |J|<\infty, %
\]
then the filter is stable in the sense that
\[
\bigl\|\pi_t^\mu-\pi_t^\nu
\bigr\|_{\mathcal{E}^J} \mathop{\longrightarrow}^{t\to\infty}0\qquad \mbox{in }\mathbf{P}^\gamma
\mbox{-probability} \mbox{ for all }J\subseteq I, |J|<\infty %
\]
for every $\mu,\nu,\gamma\in\mathcal{P}(E)$.
\end{thmm}

We will reduce the proof to the discrete time case. The key to this
reduction is the following lemma, essentially due to Blackwell and
Dubins \cite{BD62}.

\begin{lem}
\label{lem:blackwell}
Let $\mathbf{R}$ and $\mathbf{R}'$ be probabilities on
$D(\mathbb{R}_+,F)$. Let $r_t$ the coordinate process
of $D(\mathbb{R}_+,F)$ and $\mathcal{G}_t=\sigma\{r_s\dvtx s\in[0,t]\}$. If $\mathbf{R}\ll\mathbf{R}'$, then
\[
\bigl\|\mathbf{R}[ \cdot |\mathcal{G}_t]- \mathbf{R}'[
\cdot |\mathcal{G}_t]\bigr\| \mathop{\longrightarrow}^{t\to\infty}0 \qquad\mbox{in }
\mathbf{R}\mbox{-probability}. %
\]
\end{lem}

\begin{pf}
Let $\Lambda=d\mathbf{R}/d\mathbf{R}'$. Then the Bayes formula yields
\[
\mathbf{R}[A|\mathcal{G}_t] - \mathbf{R'}[A|
\mathcal{G}_t] = \mathbf{E}_{\mathbf{R'}} \biggl[
\mathbf{1}_A \frac{\Lambda-\mathbf{E}_{\mathbf{R'}}[\Lambda|\mathcal{G}_t]
}{\mathbf{E}_{\mathbf{R'}}[\Lambda|\mathcal{G}_t]} \Big|\mathcal{G}_t \biggr],\qquad
\mathbf{R}\mbox{-a.s.} %
\]
As the Borel $\sigma$-field of $D(\mathbb{R}_+,F)$ is countably generated,
it follows that
\[
\bigl\|\mathbf{R}[ \cdot |\mathcal{G}_t] - \mathbf{R'}[
\cdot |\mathcal{G}_t]\bigr\| \le \frac{\mathbf{E}_{\mathbf{R'}}[\Delta_t|\mathcal{G}_t]}{
\mathbf{E}_{\mathbf{R'}}[\Lambda|\mathcal{G}_t]},\qquad \mathbf{R}
\mbox{-a.s.}, \Delta_t= \bigl|\Lambda-\mathbf{E}_{\mathbf{R'}}[\Lambda|
\mathcal{G}_t]\bigr|. %
\]
But note that $\mathbf{E}_{\mathbf{R'}}[\Lambda|\mathcal{G}_t]\to\Lambda$ and, therefore,
$\Delta_t\to0$ in $\mathbf{R}'$-probability by the
martingale convergence theorem (a right-continuous modification of the
martingale is not needed for convergence in probability), while
$\Lambda>0$, $\mathbf{R}$-a.s. The remaining steps of the proof follow
the proof of Lemma~\ref{lem:fsdom}.
\end{pf}

We now turn to the proof of Theorem~\ref{thmm:clmfs}.

\begin{pf*}{Proof of Theorem~\ref{thmm:clmfs}}
Let $\bar E=D([0,1];E)$, $\bar F=D([0,1];F)$,
$X_n=x_{n,n+1}$, and $Y_n = y_{n,n+1}$.
Then $(X_n,Y_n)_{n\ge0}$ is a nondegenerate hidden
Markov model in $\bar E\times\bar F$ under $\mathbf{P}^{\mu}$
for every $\mu\in\mathcal{P}(E)$: in particular,
$(X_n)_{n\ge0}$ is a Markov chain with initial measure $\bar\mu$
and transition kernel $P_0$ given by
\[
\bar\mu(d\xi)= \mathbf{P}^{\mu}[x_{0,1}\in d\xi],\qquad
P_0 \bigl(\xi,d\xi' \bigr) = \mathbf{P}^{\xi(1)}
\bigl[x_{0,1}\in d\xi' \bigr], %
\]
while, by the nondegeneracy assumption,
the observation kernel $\Phi$ is
\[
\Phi(\xi,d\eta) = g_1(\xi,\eta) \varphi_1(d\eta)
\]
[so that, as in Section~\ref{sec:fswk},
$(X_n,Y_n)_{n\ge0}$ is the Markov chain with transition kernel
$P(\xi,\eta,d\xi',d\eta')=P_0(\xi,d\xi')\Phi(\xi',d\eta')$]. Moreover,
$\bar\lambda=\mathbf{P}[x_{0,1}\in\cdot ]$ is an invariant
probability for the discrete-time model
$(X_n,Y_n)_{n\ge0}$.

We can now apply Corollary~\ref{cor:flocmix} to the discrete-time
model $(X_n,Y_n)_{n\ge0}$. Indeed, we can decompose
$\bar E\subseteq\prod_{i\in I}\bar E^i$ with
$\bar E^i=D([0,1];E^i)$, and our local
mixing assumption directly implies that the discrete model
$(X_n,Y_n)_{n\ge0}$ is locally mixing with respect to this
decomposition. It follows that
\[
\bigl\|\mathbf{P}^{\mu} \bigl[x_{n,n+1}^J\in\cdot |
\mathcal{F}^{y}_{0,n+1} \bigr]- \mathbf{P}
\bigl[x_{n,n+1}^J \in\cdot | \mathcal{F}^{y}_{0,n+1}
\bigr]\bigr\| \mathop{\longrightarrow}^{n\to \infty}0,\qquad\mathbf{P}\mbox{-a.s.} %
\]
for all $J\subseteq I$, $|J|<\infty$ and $\mu\in\mathcal{P}(E)$
by Corollary~\ref{cor:flocmix}, while Corollary~\ref{cor:oequiv} yields the
equivalence
$\mathbf{P}^{\mu}|_{\mathcal{F}^{y}_{0,\infty}}
\sim\mathbf{P}|_{\mathcal{F}^{y}_{0,\infty}}$.
The latter implies that
\[
\bigl\| \mathbf{P}^{\mu} \bigl[ y_{0,\infty}\in\cdot |
\mathcal{F}^{y}_{0,t} \bigr]- \mathbf{P} \bigl[
y_{0,\infty}\in\cdot | \mathcal{F}^{y}_{0,t} \bigr]\bigr \|
\mathop{\longrightarrow}^{t\to\infty}0 \qquad\mbox{in }\mathbf{P}\mbox{-probability} %
\]
by Lemma~\ref{lem:blackwell}, which we now proceed to exploit.

Let $t\in[n,n+1]$ for some $n\in\mathbb{N}$. Then we have
\[
\pi_t^{\mu} = \mathbf{E}^{\mu} \bigl[
\mathbf{P}^{\mu} \bigl[ x_t\in\cdot |\mathcal{F}^{y}_{0,n+1}
\bigr] |\mathcal{F}^{y}_{0,t} \bigr], \qquad\pi_t =
\mathbf{E} \bigl[ \mathbf{P} \bigl[ x_t\in\cdot |
\mathcal{F}^{y}_{0,n+1} \bigr] |\mathcal{F}^{y}_{0,t}
\bigr]. %
\]
We can therefore estimate
\begin{eqnarray*}
\bigl\|\pi_t^{\mu} - \pi_t\bigr\|_{\mathcal{E}^J}& \le&
\bigl\| \mathbf{P}^{\mu} \bigl[ y_{0,\infty}\in\cdot |
\mathcal{F}^{y}_{0,t} \bigr]- \mathbf{P} \bigl[
y_{0,\infty}\in\cdot | \mathcal{F}^{y}_{0,t} \bigr]\bigr\|
\\
&&{} + \mathbf{E} \bigl[\bigl\| \mathbf{P}^{\mu} \bigl[ x_t
\in \cdot |\mathcal{F}^{y}_{0,n+1} \bigr] - \mathbf{P} \bigl[
x_t\in\cdot |\mathcal{F}^{y}_{0,n+1} \bigr]
\bigr\|_{\mathcal{E}^J}| \mathcal{F}^{y}_{0,t} \bigr].
\end{eqnarray*}
It follows that $\|\pi_t^{\mu} -\pi_t\|_{\mathcal{E}^J}
\to0$ as $t\to\infty$ in $\mathbf{P}$-probability.
As $\mu$ was arbitrary, the proof is easily completed using the
triangle inequality and the equivalence of all observation laws
to $\mathbf{P}|_{\mathcal{F}^{y}_{0,\infty}}$
as established above.
\end{pf*}

\begin{rem}
As we have seen above, deducing filter stability in continuous time
from our discrete time results requires some additional arguments
(a slightly longer argument will be used below in the setting of
asymptotic coupling). Let us therefore note, for sake of completeness,
that the corresponding results on the ergodicity of the filtering
process $(\pi_t)_{t\ge0}$ as in Section~\ref{sec:kunita}
follow immediately from their discrete-time counterparts:
in fact, uniqueness of the invariant measure of any continuous-time
Markov process $(\pi_t)_{t\ge0}$ is evidently implied by
uniqueness for the discretely sampled process $(\pi_n)_{n\in\mathbb{N}}$.
There is therefore no need to consider this question separately.
\end{rem}

\subsubsection{Asymptotic coupling in continuous time}

We now turn to the problem of obtaining a continuous-time counterpart to
our asymptotic coupling filter stability Theorem~\ref{thmm:wfstab}. To
this end we will assume, as we have done throughout this section, that
$(x_t,y_t)_{t\ge0}$ is a Markov additive process in the
Polish state space $E\times F$. In addition, we will assume in
this subsection that the unobserved process $(x_t)_{t\ge0}$ has
continuous sample paths. While this is not absolutely essential, the
restriction to continuous processes facilitates the treatment of
asymptotic couplings in continuous time.

The following is the main result of this section. As in Theorem~\ref{thmm:wfstab}, we will fix in the following
a complete metric $d$ for the Polish space $E$.

\begin{thmm}[(Continuous weak-* filter stability)]
\label{thmm:contwfstab}
Let $(x_t,y_t)_{t\ge0}$ be a nondegenerate Markov additive
process that admits an invariant probability $\lambda$, and assume
that the unobserved process $(x_t)_{t\ge0}$ has continuous
sample paths. Moreover, let $\tilde d(x,y)\ge d(x,y)$ for all
$x,y\in E$, fix $\Delta>0$, and define the intervals $I_n=[n\Delta
,(n+1)\Delta]$.
Suppose that the following hold:
\begin{longlist}[(a)]
\item[(a)] There exists $\alpha>0$ such that
\[
\forall x,x'\in E, \exists \mathbf{Q}\in\mathcal{C} \bigl(
\mathbf{P}^{x}, \mathbf{P}^{x'} \bigr) \mbox{ s.t.}\qquad \mathbf{Q}
\Biggl[\sum_{n=1}^\infty \sup
_{t\in I_n} \tilde d \bigl(x_t,x_t'
\bigr)^2<\infty \Biggr]\ge\alpha. %
\]
\item[(b)] There exists $C<\infty$ such that for all
$\delta\le\Delta$ and
$\xi,\xi'\in C([0,\delta];E)$
\[
\int \bigl\{{ \sqrt{g_\delta(\xi,\eta)} -\sqrt{g_\delta \bigl(
\xi',\eta \bigr)}} \bigr\}^2 \varphi_\delta(d
\eta) \le C \sup_{t\in[0,\delta]}\tilde d \bigl(\xi(t),
\xi'(t) \bigr)^2. %
\]
\end{longlist}
Then the filter is stable in the sense that
\[
\bigl\|\pi_t^\mu-\pi_t^\nu
\bigr\|_{\mathrm{BL}} \mathop{\longrightarrow}^{t\to\infty}0\qquad \mbox{in }\mathbf{P}^\gamma
\mbox{-probability } \mbox{for all } \mu,\nu,\gamma\in\mathcal{P}(E). %
\]
\end{thmm}

\begin{pf}
We begin by noting that if assumption \emph{a}
holds for $\Delta$, then this assumption also holds if $\Delta$ is replaced
by $\Delta/r$ for some $r\in\mathbb{N}$. Indeed, as
\[
\sum_{n=1}^\infty \sup_{t\in I_n}
\tilde d \bigl(x_t,x_t'
\bigr)^2 \ge \frac{1}{r}\sum_{k=1}^r
\sum_{n=1}^\infty \sup_{t\in[(n+(k-1)/r)\Delta,(n+k/r)\Delta]}
\tilde d \bigl(x_t,x_t'
\bigr)^2, %
\]
the claim follows. Fix $r\in\mathbb{N}$ for the time being.
Define $\bar E=C([0,\Delta/r];E)$, $\bar F=D([0,\Delta/r];F)$,
$X_n=x_{t_n,t_{n+1}}$, and
$Y_n=y_{t_n,t_{n+1}}$, where $t_n=n\Delta/r$.
Then it follows as in the proof of Theorem~\ref{thmm:clmfs} that
$(X_n,Y_n)_{n\ge0}$ is a nondegenerate hidden Markov model
in $\bar E\times\bar F$ that admits an invariant probability
[note that the definition of $\bar E$ takes into account
that $(x_t)_{t\ge0}$ has continuous sample paths].
Moreover, if we endow $\bar E$ with the metric
$\bar d(\xi,\xi')=\sup_{t\in[0,\Delta/r]}d(\xi(t),\xi'(t))$, then evidently
the assumptions of Theorem~\ref{thmm:wfstab} are satisfied.
It follows that for any $f\in\mathrm{Lip}(\bar E)$ and
$\mu\in\mathcal{P}(E)$
\[
\bigl|\mathbf{E}^\mu \bigl[f(x_{t_{n-1},t_n})| \mathcal{F}^{y}_{0,t_n}
\bigr] - \mathbf{E} \bigl[f(x_{t_{n-1},t_n}) |\mathcal{F}^{y}_{0,t_n}
\bigr]\bigr| \mathop{\longrightarrow}^{n\to\infty}0 %
\]
$\mathbf{P}$-a.s. To proceed, let us fix
$g\in\mathrm{Lip}(E)$, and define
$G(\xi_{0,\Delta/r})=g(\xi(0))$ and
$\bar G(\xi_{0,\Delta/r})=\sup_{s\in[0,\Delta/r]}|g(\xi(0))-g(\xi(s))|$.
We can easily estimate
\begin{eqnarray*}
\hspace*{-3pt}&&\bigl| \mathbf{E}^\mu \bigl[g(x_t)| \mathcal{F}^{y}_{0,t_n}
\bigr] - \mathbf{E} \bigl[g(x_t) |\mathcal{F}^{y}_{0,t_n}
\bigr]\bigr|\\
\hspace*{-3pt}&&\qquad\le\bigl | \mathbf{E}^\mu \bigl[g(x_{t_{n-1}})|
\mathcal{F}^{y}_{0,t_n} \bigr] - \mathbf{E}
\bigl[g(x_{t_{n-1}}) |\mathcal{F}^{y}_{0,t_n} \bigr]\bigr|
\\
\hspace*{-3pt}&&\qquad\quad{} +\bigl|\mathbf{E}^\mu \bigl[ \bar G(x_{t_{n-1},t_n}) |
\mathcal{F}^{y}_{0,t_n} \bigr] - \mathbf{E} \bigl[ \bar
G(x_{t_{n-1},t_n}) |\mathcal{F}^{y}_{0,t_n} \bigr]\bigr| +2
\mathbf{E} \bigl[ \bar G(x_{t_{n-1},t_n}) |\mathcal{F}^{y}_{0,t_n}
\bigr]
\end{eqnarray*}
for any $t\in[t_{n-1},t_n]$. As $G$ and $\bar G$
are $\bar d$-Lipschitz, we obtain
\begin{eqnarray*}
&&\limsup_{n\to\infty} \sup_{t\in[t_{n-1},t_n]} \mathbf{E}
\bigl[\bigl| \mathbf{E}^\mu \bigl[g(x_t)|
\mathcal{F}^{y}_{0,t_n} \bigr] - \mathbf{E}
\bigl[g(x_t) |\mathcal{F}^{y}_{0,t_n} \bigr]\bigr|
\bigr]
\\
&&\qquad\le 2 \mathbf{E} \Bigl[\sup_{s\in[0,\Delta/r]}\bigl|g(x_0)-
g(x_s)\bigr| \Bigr].
\end{eqnarray*}
On the other hand, we can estimate as in the proof of
Theorem~\ref{thmm:clmfs}
\begin{eqnarray*}
\mathbf{E} \bigl[\bigl|\pi_t^{\mu}(g) - \pi_t(g)\bigr|
\bigr] &\le& \mathbf{E} \bigl[\bigl\| \mathbf{P}^{\mu} \bigl[ y_{0,\infty}
\in \cdot | \mathcal{F}^{y}_{0,t} \bigr]- \mathbf{P} \bigl[
y_{0,\infty}\in\cdot | \mathcal{F}^{y}_{0,t} \bigr]\bigr\|
\bigr]
\\
&&{} + \mathbf{E} \bigl[\bigl| \mathbf{E}^{\mu} \bigl[
g(x_t)| \mathcal{F}^{y}_{0,t_n} \bigr] -
\mathbf{E} \bigl[g(x_t) |\mathcal{F}^{y}_{0,t_n}
\bigr]\bigr| \bigr]
\end{eqnarray*}
for $t\in[t_{n-1},t_n]$. Applying
Lemma~\ref{lem:blackwell} as in Theorem~\ref{thmm:clmfs} yields
\[
\limsup_{t\to\infty}\mathbf{E} \bigl[\bigl|\pi_t^{\mu}(g)
- \pi_t(g)\bigr| \bigr] \le 2 \mathbf{E} \Bigl[\sup_{s\in[0,\Delta/r]}\bigl|g(x_0)-
g(x_s)\bigr| \Bigr]. %
\]
But note that this holds for any $r\in\mathbb{N}$.
Letting $r\to\infty$, we obtain
\[
\bigl|\pi_t^{\mu}(g) - \pi_t(g)\bigr| \mathop{\longrightarrow}^{t\to
\infty}0 \qquad\mbox{in }\mathbf{P}\mbox{-probability} %
\]
using the continuity of paths. Finally, note that $g\in\mathrm{Lip}(E)$
is arbitrary. We can therefore strengthen the convergence
for individual $g$ to $\|\cdot\|_{\mathrm{BL}}$-convergence as in the proof of
Theorem~\ref{thmm:wfstab}. The proof is now easily completed using
the triangle inequality and the equivalence of all observation laws
to $\mathbf{P}|_{\mathcal{F}^{y}_{0,\infty}}$
as established in the proof of Theorem~\ref{thmm:clmfs}.
\end{pf}

\section{Examples}
\label{sec:examples}

Infinite-dimensional Markov processes and filtering problems arise in a
diverse range of applications; see, for example, \cite{DZ96,Stu10}. The
aim of this section is to demonstrate that the abstract theory that we
have developed in the previous sections is directly applicable in several
different settings. In Section~\ref{sec:exheat}, we consider the simplest
possible example of an infinite-dimensional system: a~stochastic heat
equation with smooth forcing and point observations. While this example
is nearly trivial, it allows us to easily illustrate our results in the
simplest possible setting. In Section~\ref{sec:exsns}, we consider a
highly degenerate stochastic Navier--Stokes equation with Eulerian
observations. In Section~\ref{sec:exspin}, we consider stochastic spin
systems. Finally, in Section~\ref{sec:exdelay} we consider filtering
problems for stochastic delay equations.

\subsection{Stochastic heat equation}
\label{sec:exheat}

We investigate the following example from \cite{Mat07}. Consider the
stochastic heat equation on the unit interval $x\in[0,1]$:
\[
du(t,z) = \Delta u(t,z) \,dt + dw(t,z),\qquad u(t,0)=u(t,1)=0. %
\]
Here, $dw(t,z)$ is the white in time, smooth in space random forcing
\[
w(t,z) = \sum_{k=1}^\infty
\sigma_k\sqrt{2} \sin(\pi kz) W_t^k,
\]
$(W_t^k)_{t\ge0}$, $k\in\mathbb{N}$, are independent Brownian motions,
$\sum_{k=1}^\infty\sigma_k^2<\infty$,
and $\sigma_k>0$ for all $k\in\mathbb{N}$. We will assume that $u(t,z)$
is observed at the points $z_1,\ldots,z_n\in[0,1]$ and that the observations
are corrupted by independent white noise: that is, we introduce the
$\mathbb{R}^n$-valued observation model
\[
dy_t^i = u(t,z_i) \,dt +
dB_t^i, \qquad i=1,\ldots,n, %
\]
where $(B_t^i)_{t\ge0}$, $i=1,\ldots,n$, are independent Brownian motions
that are independent of $(W_t^k)_{t\ge0}$, $k\in\mathbb{N}$.
As we are working with Dirichlet boundary conditions, we view
$z\mapsto u(t,z)$ as taking values in the Hilbert subspace $H\subset L^2[0,1]$
spanned by the eigenfunctions $(e_k)_{k\in\mathbb{N}}$,
$e_k(z)=\sqrt{2}\sin(\pi kz)$.

\begin{lem}
\label{lem:exheat}
Let $x_t = u(t,\cdot)$. Then the pair $(x_t,y_t)_{t\ge0}$ defines a
nondegenerate Markov additive process in $H\times\mathbb{R}^n$ with
continuous paths. Moreover, the unobserved process $(x_t)_{t\ge0}$
admits a unique invariant probability~$\lambda$.
\end{lem}

\begin{pf}
It is easily seen that for any $u(0,\cdot)\in H$,
the equation for $u(t,\cdot)$ has a unique mild
solution in $H$ that has continuous paths and
satisfies the Markov property (cf. \cite{DZ96}).
If we expand $x_t = \sum_{k=1}^\infty x_t^k e_k$, then evidently
\[
dx_t^k = -\pi^2k^2x_t^k
\,dt + \sigma_k \,dW_t^k. %
\]
By It\^o's formula, we obtain
\[
\mathbf{E} \bigl[\|x_t\|_H^2 \bigr] +
\mathbf{E} \biggl[ \int_0^t 2\|x_s
\|_{H^1}^2 \,ds \biggr] = \|x_0
\|_H^2 + \|\sigma\|_H^2 t,
\]
where we defined the Sobolev norm
$\|x_t\|_{H^s}^2 = \sum_{k=1}^\infty(\pi k)^{2s}(x_t^k)^2$.
Note that $|u(t,z)|\le\sqrt{2}\sum_{k=1}^\infty|x_t^k|\le
3^{-1/2}\|x_t\|_{H^1}$ by Cauchy--Schwarz. Thus, $z\mapsto u(t,z)$ is
continuous for a.e. $t$, so the observation process $y_t$ is
well defined and the pair $(x_t,y_t)_{t\ge0}$ defines a
Markov additive process. Moreover,
\[
\mathbf{E} \biggl[\int_0^t
\bigl|u(s,z_i)\bigr|^2 \,ds \biggr] \le \frac{\|x_0\|_H^2 + \|\sigma\|_H^2  t}{6}<\infty.
\]
Therefore, by Girsanov's theorem,
the conditional law of $y_{0,t}$ given $x_{0,t}$ is equivalent to the
Wiener measure a.s. for any $t<\infty$ and $x_0\in H$ [as
$(u(t,z_i))_{t\ge0}$ and $(B_t^i)_{t\ge0}$ are independent,
Novikov's criterion can be applied conditionally]. This establishes
the nondegeneracy assumption. Finally, as each Fourier mode $x_t^k$ is
an independent Ornstein--Uhlenbeck process, it is easily seen by
explicit computation that the law of $x_t$ converges weakly as $t\to
\infty$
to a unique Gaussian product measure $\lambda$ for any $x_0\in H$.
\end{pf}

It is evident from Lemma~\ref{lem:exheat} that the ergodic theory of
$u(t,z)$ is quite trivial: each of the Fourier modes is an
independent ergodic one-dimensional Ornstein--Uhlenbeck process (recall
Example~\ref{ex:trivial}). Nonetheless, the reader may easily verify
using the Kakutani theorem \cite{Shi96}, page 531, that $(x_t)_{t\ge0}$
is not Harris when the forcing is sufficiently smooth (e.g.,
$\sigma_k = e^{-k^3}$). Moreover, the finite-dimensional projections
$(x_t^1,\ldots,x_t^k,y_t)$ are not Markovian. Thus, stability of the
corresponding nonlinear filter
does not follow from earlier results. While this example remains
essentially trivial, it is nonetheless instructive to illustrate our
results in this simplest possible setting.

\subsubsection{Local mixing}

To every $x=\sum_{k=1}^\infty x_ke_k\in H$, we identify a vector of
Fourier coefficients $(x_k)_{k\in\mathbb{N}}\in\mathbb{R}^\mathbb{N}$.
In order to apply our local mixing results, we can therefore
view $H$ as a subset of the product space $\mathbb{R}^\mathbb{N}$. Note
that $H$ is certainly not a topological subspace of~$\mathbb{R}^\mathbb{N}$ (pointwise convergence of the Fourier
coefficients does not imply convergence in $H$); however, $H$ is a
measurable subspace of~$\mathbb{R}^\mathbb{N}$, which is all that is
needed in the present setting.

For every $k\in\mathbb{N}$, define the local $\sigma$-fields
\[
\mathcal{F}_{s,t}^k = \sigma \bigl\{x_{s,t}^1,
\ldots,x_{s,t}^k,y_{s,t} \bigr\},\qquad s\le t.
\]
To apply Theorem~\ref{thmm:clmfs}, it suffices to establish the
local mixing property.

\begin{lem}
The Markov additive process $(x_t,y_t)_{t\ge0}$ is locally mixing:
\[
\bigl\|\mathbf{P}^x-\mathbf{P}\bigr\|_{\mathcal{F}_{t,\infty}^k} \mathop{\longrightarrow}^{t\to
\infty}0\qquad \mbox{for every }x\in H, k\in\mathbb{N}. %
\]
\end{lem}

\begin{pf}
Let $x,x'\in H$, and define $v\in H$ such that $\langle e_\ell,v\rangle=
\langle e_\ell,x\rangle$ for $1\le\ell\le k$, and
$\langle e_\ell,v\rangle=\langle e_\ell,x'\rangle$ for $\ell>k$.
It is easily seen that
\[
\bigl\|\mathbf{P}^x-\mathbf{P}^{x'}\bigr\|_{\mathcal{F}_{t,\infty}^k} \le\bigl \|
\mathbf{P}^x-\mathbf{P}^{v}\bigr\|_{\mathcal{F}_{t,\infty}^k} + \bigl\|
\mathbf{P}^{x'}-\mathbf{P}^{v}\bigr\|_{\sigma\{x_t\}} %
\]
by the Markov additive property.
As the Fourier modes are independent, we evidently have
$\|\mathbf{P}^{x'}-\mathbf{P}^{v}\|_{\sigma\{x_t\}}=
\|\mathbf{P}^{x'}-\mathbf{P}^{v}\|_{\sigma\{x_t^1,\ldots,x_t^k\}}\to0$
as $t\to\infty$ (e.g., by explicit computation of the law of
the $k$-dimensional Ornstein--Uhlenbeck process). It therefore remains
to consider the first term.

Construct on a larger probability space
$(\Omega',\mathcal{F}',\mathbf{Q})$ the triple
$(x_t,v_t,y_t)_{t\ge0}$ as follows. The processes $x_t$ and $v_t$ are
solutions to the stochastic heat equation driven by the same Brownian
motion realization, but with different initial conditions $x_0=x$ and
$v_0=v$, while $dy_t^i = x_t(z_i) \,dt+dB_t^i$ as above.
Now note that can show precisely as in the proof of
Lemma~\ref{lem:exheat} that
\[
\mathbf{E} \biggl[ \int_0^\infty\bigl|v_s(z_i)-x_s(z_i)\bigr|^2
\,ds \biggr] \le \mathbf{E} \biggl[ \frac{1}{3}\int_0^\infty
\|v_s-x_s\|_{H^1}^2 \,ds \biggr]
\le \frac{\|x-v\|_H^2}{6}<\infty. %
\]
As $(x_t,v_t)_{t\ge0}$ is independent of $(B_t)_{t\ge0}$, we can apply
Novikov's criterion conditionally to establish that
$\mathbf{E}[\Lambda_t]=1$ for any $t\ge0$, where we define
\[
\Lambda_{t} = \prod_{i=1}^n
\exp \biggl[ \int_t^\infty \bigl
\{v_s(z_i)-x_s(z_i) \bigr\}
\,dB_s^i -\frac{1}{2}\int_t^\infty
\bigl|v_s(z_i)-x_s(z_i)\bigr|^2
\,ds \biggr]. %
\]
Using Girsanov's theorem, we obtain for any $A\in\mathcal{F}^k_{t,\infty}$
\[
\mathbf{P}^z(A) = \mathbf{E}_{\mathbf{Q}} \bigl[
\mathbf{1}_A \bigl(v_{t,\infty}^1,\ldots
v_{t,\infty}^k,y_{t,\infty} \bigr) \Lambda_{t}
\bigr] = \mathbf{E}_{\mathbf{Q}} \bigl[\mathbf{1}_A
\bigl(x_{t,\infty}^1,\ldots, x_{t,\infty}^k,y_{t,\infty}
\bigr) \Lambda_{t} \bigr], %
\]
where we have used that $x_t^\ell=v_t^\ell$ for all $t\ge0$ when
$\ell\le k$. Moreover,
the law of $(x_t,y_t)_{t\ge0}$ under $\mathbf{Q}$ obviously coincides
with $\mathbf{P}^x$. We therefore conclude that
$\|\mathbf{P}^x-\mathbf{P}^z\|_{\mathcal{F}^k_{t,\infty}}\le
\mathbf{E}_{\mathbf{Q}}[|\Lambda_{t}-1|]\to0$ as $t\to\infty$ by Scheff\'e's
lemma.
\end{pf}

Let $\mathcal{E}^k=\sigma\{P_k\}$, where $P_k\dvtx H\to H$ is the projection
onto the first $k$ Fourier modes. Theorem~\ref{thmm:clmfs} immediately
yields the filter stability result
\[
\bigl\|\pi_t^\mu-\pi_t^\nu
\bigr\|_{\mathcal{E}^k}\mathop{\longrightarrow}^{t\to\infty}0 \qquad\mbox{in }\mathbf{P}^\gamma
\mbox{-probability} \mbox{ for all }k\in\mathbb{N}, \mu,\nu,\gamma\in
\mathcal{P}(H). %
\]
A simple tightness argument can be used to deduce also filter stability in
the bounded-Lipschitz norm from this statement. However,
let us demonstrate instead how the latter can be obtained
directly from Theorem~\ref{thmm:contwfstab}.

\subsubsection{Asymptotic coupling}

It was shown in Lemma~\ref{lem:exheat} that
$(x_t,y_t)_{t\ge0}$ is nondegenerate. It follows from the
proof that we may choose the reference measure $\varphi_\delta$
to be the Wiener measure on $D([0,\delta];\mathbb{R}^n)$ and that
\[
g_\delta(\xi,\eta) = \prod_{i=1}^n
\exp \biggl[ \int_0^\delta\xi(s,z_i)
\,d \eta^i(s) -\frac{1}{2}\int_0^\delta
\bigl|\xi(s,z_i)\bigr|^2 \,ds \biggr] %
\]
for $\xi\in C([0,\delta];H)\cap L^2([0,\delta],H^1)$
[for simplicity, let $g_\delta(\xi,\eta)=1$ otherwise].
We begin by establishing the Lipschitz property of the observations.

\begin{lem}
\label{lem:h1obs}
For all $\delta\le1$ and $\xi,\xi'\in C([0,\delta];H)$
\[
\int \bigl\{{ \sqrt{g_\delta(\xi,\eta)} -\sqrt{g_\delta \bigl(
\xi',\eta \bigr)}} \bigr\}^2 \varphi_\delta(d
\eta) \le\frac{n}{12} \sup_{t\in[0,\delta]}\bigl\|\xi(t)-
\xi'(t)\bigr\|_{H^1}^2. %
\]
\end{lem}

\begin{pf}
The result is trivial unless $\xi,\xi'\in L^2([0,\delta];H^1)$,
in which case
\[
\int \bigl\{{ \sqrt{g_\delta(\xi,\eta)} -\sqrt{g_\delta \bigl(
\xi',\eta \bigr)}} \bigr\}^2 \varphi_\delta(d
\eta) = 2-2 e^{
-({1}/{8})\int_0^\delta
\sum_{i=1}^n|\xi(s,z_i)-\xi'(s,z_i)|^2 \,ds}. %
\]
Now use $1-e^{-x}\le x$ and
$|\xi(s,z_i)-\xi'(s,z_i)|\le3^{-1/2}\|\xi(s)-\xi'(s)\|_{H^1}$.
\end{pf}

Thus, the second assumption of Theorem~\ref{thmm:contwfstab} is satisfied
for $\tilde d(x,y) = \|x-y\|_{H^1}$ and $\Delta=1$. It is clear
that the observations cannot be continuous with respect to
$\|\cdot\|_H$, which is the reason that we have introduced the
pseudodistance $\tilde d$ in Theorem~\ref{thmm:contwfstab}.
To establish filter stability, it remains to produce an asymptotic
coupling in $H^1$, which is trivial in this example.

\begin{lem}
For all $x,x'\in H$, there exists
$\mathbf{Q}\in\mathcal{C}(\mathbf{P}^x,\mathbf{P}^{x'})$ such that
\[
\sum_{n=1}^\infty \sup_{t\in[n,n+1]}
\bigl\|x_t-x_t'\bigr\|_{H^1}^2<
\infty,\qquad\mathbf{Q}\mbox{-a.s.} %
\]
\end{lem}

\begin{pf}
Choose $\mathbf{Q}$ such that the processes
$x_t$ and $x_t'$ are
solutions to the stochastic heat equation driven by the same Brownian
motion realization, but with different initial conditions $x_0=x$ and
$x_0'=x'$. Then
\[
\rho_t = x_t-x_t',\qquad d
\rho_t^k = -\pi^2k^2
\rho_t^k \,dt %
\]
as in the proof of Lemma~\ref{lem:exheat}. As the difference
$\rho_t$ is deterministic, the result follows readily (e.g.,
$\|\rho_t\|_{H^1}$ can be computed explicitly).
\end{pf}

As we have verified all the assumptions of
Theorem~\ref{thmm:contwfstab}, it follows that
\[
\bigl\|\pi_t^\mu-\pi_t^\nu
\bigr\|_{\mathrm{BL}} \mathop{\longrightarrow}^{t\to\infty}0\qquad \mbox{in }\mathbf{P}^\gamma
\mbox{-probability} \mbox{ for all } \mu,\nu,\gamma\in\mathcal{P}(E), %
\]
that is, we have established filter stability in the bounded-Lipschitz
norm.

\begin{rem}
Beside that it admits a trivial ergodic theory, the example considered
this section is special in that it is a linear Gaussian model.
In finite dimension, such filtering problems are amenable to explicit
analysis as the filter reduces to the well-known Kalman filter, which
is a rather simple linear equation~\cite{LS01}. Some results in this
direction for linear stochastic evolution equations were considered by
Vinter \cite{Vin77}. However, the present example does not fit in the
setting of \cite{Vin77} as the observation operator $C\dvtx H\to\mathbb{R}^n$,
$Cu=(u(z_1),\ldots,u(z_n))$ is unbounded, which significantly complicates
even the definition of the Kalman filtering equations in infinite
dimension. It is therefore interesting to note the ease with which
we have obtained stability results from our general nonlinear theory
even in this trivial linear example.
\end{rem}

\subsection{Stochastic Navier--Stokes equation}
\label{sec:exsns}

We now turn to a much less trivial example inspired by
\cite{Stu10}, Section~3.6: we will consider discrete time
Eulerian (point) observations of the velocity of a fluid
that is modeled by a Navier--Stokes equation with white in
time, smooth in space random forcing.

We consider a velocity field $u(t,z)\in\mathbb{R}^2$ on the
two-dimensional torus $z\in\mathbb{T}^2=[-\pi,\pi]^2$
such that $\int u(t,z) \,dz=0$ and $\nabla\cdot u(t,z)=0$
for all $t\ge0$.
The dynamics of $u(t,z)$ are given by the stochastic Navier--Stokes equation
\[
du(t,z) = \bigl\{\nu\Delta u(t,z) - \bigl(u(t,z)\cdot\nabla \bigr)u(t,z) -\nabla
p(t,z) \bigr\} \,dt + d\tilde w(t,z) %
\]
with periodic boundary conditions, where $\nu>0$ is the fluid
viscosity, $\tilde w$ is a spatial mean zero stochastic forcing
to be specified later, and the pressure $p$ is chosen to enforce
the divergence-free constraint $\nabla\cdot u(t,z)=0$.

To define the observations, let us fix points
$z_1,\ldots,z_r\in\mathbb{T}^2$ at which the fluid velocity
is measured. We assume that measurements are taken at the
discrete time instants $t_n=n\delta$, $n\ge0$, where we fix
the sampling interval $\delta>0$ throughout this section.
The observations are then given by\setcounter{footnote}{1}\footnote{%
The observation equation makes sense when $u(t,\cdot)\in H^2$,
as this implies that $z\mapsto u(t,z)$ is continuous
by the Sobolev embedding theorem. For concreteness, we can
define $Y_n^i = u(t_n,z_i)\mathbf{1}_{u(t_n,\cdot)\in H^2} + \xi_n^i$
which makes sense for any velocity field. As we will
always work under
assumptions that ensure sufficient smoothness of the solutions
of the stochastic Navier--Stokes equations for all $t>0$, this
minor point will not affect our results.}
\[
Y_n^i = u(t_n,z_i) +
\xi_n^i,\qquad i=1,\ldots,r, n\ge0, %
\]
where $(\xi_n)_{n\ge0}$ are i.i.d. $\mathbb{R}^{2r}$-dimensional
Gaussian random variables with nondegenerate covariance that are
independent of $(u(t,\cdot))_{t\ge0}$.

Following \cite{HM06}, it will be convenient to eliminate the
divergence-free constraint from the stochastic Navier--Stokes equation by
passing to an equivalent formulation. Define the vorticity $v(t,z) =
\nabla\times u(t,z) = \partial u^1(t,z)/\partial z^2- \partial
u^2(t,z)/\partial z^1$, which is a scalar field on $\mathbb{T}^2$. As
$u$ is divergence-free and has spatial mean zero, we can reconstruct the
velocity field from the vorticity as $u=\mathcal{K}v$, where the
integral operator $\mathcal{K}$ is defined in the Fourier domain as
$\langle e_k,\mathcal{K}v\rangle= -i(k^\perp/|k|^2)\langle e_k,v\rangle$
with $e_k(z)=(2\pi)^{-1}e^{ik\cdot z}$, $k\in\mathbb{Z}^2\setminus\{
(0,0)\}$,
and $k^\perp=(k^2,-k^1)$. In terms of vorticity,
the Navier--Stokes equation reads
\[
dv(t,z) = \bigl\{\nu\Delta v(t,z) - \mathcal{K}v(t,z)\cdot\nabla v(t,z) \bigr\}
\,dt + dw(t,z), %
\]
where $w(t,z)=\nabla\times\tilde w(t,z)$, and the observation equation
becomes
\[
Y_n^i = \mathcal{K}v(t_n,z_i) +
\xi_n^i,\qquad i=1,\ldots,r, n\ge0. %
\]
From now on, we will work with the vorticity equation, which we consider
as an evolution equation in the Hilbert space $H=\{v\in L^2(\mathbb{T}^2)\dvtx \int v(z) \,dz=0\}$. This formulation is equivalent to
considering the original stochastic Navier--Stokes equation in
$\{u\in H^1\dvtx \nabla\cdot u=0, \int u(z) \,dz=0\}$.
We also define the Sobolev norm
$\|v\|_{H^s}^2=\sum_k |k|^{2s}|\langle e_k,v\rangle|^2$ and
$H^s=\{v\in H\dvtx \|v\|_{H^s}<\infty\}$.

It remains to specify the structure of the forcing $w(t,z)$. As in
\cite{HM06}, we let $\mathbb{Z}^2_0=\mathbb{Z}^2\setminus
\{(0,0)\}=\mathbb{Z}^2_+\cup\mathbb{Z}^2_-$ with
$\mathbb{Z}^2_+=\{k\in\mathbb{Z}^2\dvtx k^2>0\mbox{ or }
k^2=0, k^1>0\}$ and $\mathbb{Z}^2_-=-\mathbb{Z}^2_+$, and we define the
trigonometric basis $f_k(z)=\sin(k\cdot z)$ for $k\in\mathbb{Z}^2_+$
and $f_k(z)=\cos(k\cdot z)$ for $k\in\mathbb{Z}^2_-$. The forcing
is now given by
\[
w(t,z) = \sum_{k\in\mathbb{Z}^2_0} \sigma_k
f_k(z) W_t^k, %
\]
where $(W_t^k)_{t\ge0}$,
$k\in\mathbb{Z}^2_0$, are independent standard Brownian motions, and we
will assume that $\sum_k |k|^2\sigma_k^2<\infty$ (so that the
forcing is in $H^1$).

\begin{lem}
Let $X_n=v(t_n,\cdot)$. Then $(X_n,Y_n)_{n\ge0}$ defines a nondegenerate
hidden Markov model in $H\times\mathbb{R}^{2r}$ that admits an invariant
probability.
\end{lem}

\begin{pf}
It is well known that the stochastic Navier--Stokes equation defines
a stochastic flow; see \cite{HM06,KS12} and the references therein.
Under our assumptions, this implies that the vorticity equation defines
a Markov process in $H$. Thus $(X_n,Y_n)_{n\ge0}$ is evidently
a hidden Markov model, and nondegeneracy follows as the observation
kernel has a nondegenerate Gaussian density. Moreover, as we assumed
that $\sum_k|k|^2\sigma_k^2<\infty$, standard Sobolev
estimates (e.g., \cite{KS12}, Proposition~2.4.12) show that
$v(t,\cdot)\in H^1$ for all $t>0$ a.s. for any initial condition $v(0,\cdot)\in H$. Thus, $u(t,\cdot)=
\mathcal{K}v(t,\cdot)\in H^2$ for all $t>0$ a.s., and the observation
model is defined as intended. The existence of an invariant probability
is standard (e.g., \cite{DZ96,KS12}).
\end{pf}

Our aim is now to establish stability of the nonlinear filter for the
hidden Markov model $(X_n,Y_n)_{n\ge0}$. This is much more difficult
than for the heat equation in the previous section. First, in the
present case the Fourier modes are coupled by the nonlinear term in the
equation, so that energy can move across scales. Second, unlike in the
heat equation example, only sufficiently fine scales are contracting.
Nonetheless, in the case that all Fourier modes are forced (i.e.,
$\sigma_k>0$ for all $k\in\mathbb{Z}^2_0$), it is possible to establish
local mixing using the Girsanov method developed in \cite{EMS01,Mat02}.
In fact, the approach taken in these papers is well suited to our
local zero--two laws (e.g., Lemmas 3.1 and 3.2 in \cite{EMS01} can
be used directly in conjunction with Corollary~\ref{cor:absreg} to
establish absolute regularity of a finite number of Fourier modes, and
some additional effort yields the assumptions of Corollary~\ref{cor:locmix}).
However, these methods do not extend to the
degenerate setting.

We intend to illustrate that our results are applicable even in
highly degenerate situations. To this end, we adopt the following
assumptions \cite{HM06}.

\begin{aspt}
Let $\mathcal{Z}=\{k\in\mathbb{Z}^2_0\dvtx \sigma_k\ne0\}$ be the set of
forced modes. We assume that (a) $\mathcal{Z}$ is a finite set;
(b) $\mathcal{Z}=-\mathcal{Z}$; (c) there exist $k,k'\in\mathcal{Z}$
with $|k|\ne|k'|$; (d) integer linear combinations of elements of
$\mathcal{Z}$ generate $\mathbb{Z}^2$.
\end{aspt}

It was shown by Hairer and Mattingly \cite{HM06} that under these
(essentially minimal) assumptions the stochastic Navier--Stokes equation
is uniquely ergodic. In the remainder of this section, we will show
that this assumption also ensures stability of the filter in the
bounded-Lipschitz norm. Let us emphasize that no new ergodic theory is
needed: we will simply verify the assumptions of Theorem~\ref{thmm:wfstab} by a direct application of the machinery developed in
\cite{HM06,HMS11}, together with a standard interpolation argument.

We will use the following tool to construct asymptotic couplings.

\begin{thmm}
\label{thmm:hms}
Let $Q$ be a transition kernel on $H$, and consider a continuous
function $W\dvtx H\to[1,\infty[$.
Suppose that for every $\varphi\in C^1(H)$
\[
\bigl\|\nabla Q\varphi(x)\bigr\|_H \le W(x) \bigl( C_1 \bigl\{Q
\|\nabla\varphi\|_H^2(x) \bigr\}^{1/2} +
C_2\|\varphi\|_\infty \bigr) %
\]
($\nabla$ denotes the Fr\'echet derivative).
Assume moreover that for some $p>1$
\[
QW^{2p}\le C_3^2W^{2p-2},\qquad
4C_1C_3<1. %
\]
Let $\mathbf{Q}^x$ be the law of the Markov chain $(X_n)_{n\ge0}$
with transition kernel $Q$ and $X_0=x$.
Then there exists a coupling $\mathbf{Q}^{x,x'}\in
\mathcal{C}(\mathbf{Q}^x,\mathbf{Q}^{x'})$ such that
\[
\mathbf{Q}^{x,x'} \bigl[ \bigl\|X_n-X_n'
\bigr\|_H\le C_2^{-1}2^{-(n+1)} \mbox{ for
all }n\ge1 \bigr]\ge\tfrac{1}{2} %
\]
whenever $\|x-x'\|_H\le(4C_2R)^{-1}$,
$W^p(x)\le R$, $W^p(x')\le R$ for some $R>1$.
Moreover, the map $(x,x')\mapsto\mathbf{Q}^{x,x'}$ can be chosen to be
measurable.
\end{thmm}

\begin{pf}
We have simply rephrased the proofs of
Propositions 5.5 and 4.12 in~\cite{HMS11},
making explicit choices for the constants involved.
\end{pf}

Denote by $P_0(x,\cdot)=\mathbf{P}^x[X_1\in\cdot]$ the transition kernel
of $(X_n)_{n\ge0}$. To verify the assumptions of Theorem~\ref{thmm:hms},
we require the following deep result.
This is the combined statement of Proposition~4.15 and Lemma A.1 in
\cite{HM06}.

\begin{thmm}
\label{thmm:hm}
For every $\eta>0$ and $C_1>0$, there exists $C_2>0$ so that
\[
\bigl\|\nabla P_0\varphi(x)\bigr\|_H \le \exp \bigl(\eta\|x
\|_H^2 \bigr) \bigl( C_1 \bigl
\{P_0\|\nabla\varphi\|_H^2(x) \bigr
\}^{1/2} + C_2\|\varphi\|_\infty \bigr) %
\]
for all $\varphi\in C^1(H)$ and $x\in H$. Moreover, there exist
constants $\eta_0>0$ and $C_3>0$ such that for every $0<\eta'\le\eta_0$,
$x\in H$, and $n\ge1$ we have
\[
\mathbf{E}^x \bigl[\exp \bigl(\eta'\|X_n
\|_H^2 \bigr) \bigr] \le C_3^2
\exp \bigl(\eta'e^{-\nu n\delta}\|x\|_H^2
\bigr). %
\]
\end{thmm}

Finally, we require the following reachability lemma \cite{EM01}, Lemma~3.1.

\begin{lem}
\label{lem:em}
For any $R_1,R_2>0$, there exist $n\ge1$ and $q>0$ such that
\[
\inf_{\|x\|_H\le R_1}\mathbf{P}^x\bigl[\|X_n
\|_H\le R_2\bigr] \ge q>0. %
\]
\end{lem}

Using these results, we can now obtain the following asymptotic coupling.

\begin{cor}
\label{cor:snscoup}
There exists $\alpha>0$ such that
\[
\forall x,x'\in H, \exists \mathbf{Q}\in\mathcal{C} \bigl(
\mathbf{P}^x,\mathbf{P}^{x'} \bigr) \mbox{ s.t.} \qquad\mathbf{Q}
\Biggl[ \sum_{n=1}^\infty
\bigl\|X_n-X_n'\bigr\|_{H^1}^2<
\infty \Biggr]\ge\alpha. %
\]
\end{cor}

\begin{pf}
Let $W(x)=\exp(\eta\|x\|_H^2)$ with $\eta= (1-e^{-\nu\delta})\eta_0/2$,
and define $p = (1-e^{-\nu\delta})^{-1}$ and $C_1 = 1/8C_3$ (here
$\eta_0$ and $C_3$ are as in Theorem~\ref{thmm:hm}). Defining $C_2$
as in Theorem~\ref{thmm:hm}, it is easily verified that the assumptions
of Theorem~\ref{thmm:hms} are satisfied for $Q=P_0$. Therefore,
for any $u,u'\in H$ such that $\|u\|_H\le R_2$ and $\|u'\|_H\le R_2$,
there exists $\mathbf{Q}^{u,u'}\in\mathcal{C}(\mathbf{P}^u,\mathbf{P}^{u'})$
such that
\[
\mathbf{Q}^{u,u'} \Bigl[ \sup_{n\ge1}2^n
\bigl\|X_n-X_n'\bigr\|_{H}<\infty \Bigr]
\ge\frac{1}{2}, %
\]
where we defined the constant
$R_2 = \sqrt{(2\log2)/\eta_0}\wedge(16 C_2)^{-1}$.
On the other hand, define the constant
$R_1=\sqrt{1+(\log2+2\log C_3)/\eta_0}$. Then by Lem\-ma~\ref{lem:em}, there exist $q>0$ and $n_2\ge1$ (depending on $R_1$ and
$R_2$ only) such that
\[
\inf_{\|x\|_H\le R_1}\mathbf{P}^x\bigl[\|X_{n_2}
\|_H\le R_2\bigr] \ge q>0. %
\]
From now on, let us fix $x,x'\in H$. Define
$n_1=2\log(\|x\|_H\vee\|x'\|_H)/\nu\delta$.
Then $\mathbf{E}^u[\exp(\eta_0\|X_{n_1}\|_H^2)]\le
C_3^2 \exp(\eta_0)$ for $u=x,x'$ by Theorem~\ref{thmm:hm}.
Using Chebyshev's inequality, we obtain
$\mathbf{P}^u[\|X_{n_1}\|_H\le R_1]\ge1/2$ for $u=x,x'$.

We now construct the coupling $\mathbf{Q}\in\mathcal{C}(\mathbf{P}^x,
\mathbf{P}^{x'})$ such that
\begin{eqnarray*}
\mathbf{Q} \bigl[X_{0,n_1+n_2},X_{0,n_1+n_2}'\in\cdot
\bigr]& =& \mathbf{P}^x|_{\mathcal{F}_{0,n_1+n_2}}\otimes \mathbf{P}^{x'}|_{\mathcal{F}_{0,n_1+n_2}},
\\
\mathbf{Q} \bigl[X_{n_1+n_2,\infty},X_{n_1+n_2,\infty}'\in\cdot |
\mathcal{F}_{0,n_1+n_2} \bigr] &=& \mathbf{Q}^{X_{n_1+n_2},X_{n_1+n_2}'}.
\end{eqnarray*}
Setting $\alpha=q^2/8$ (which does not depend on
$x,x'$), it is now easily seen that
\[
\mathbf{Q} \Bigl[ \sup_{n\ge1} 2^n
\bigl\|X_n-X_n'\bigr\|_{H}<\infty \Bigr]
\ge\alpha>0. %
\]
It remains to strengthen the $\|\cdot\|_H$-norm to
$\|\cdot\|_{H^1}$ in this expression.
To this end, we employ an interpolation argument. Recall the
interpolation inequality $\|u\|_{H^1}\le\|u\|_{H}^{1/2}
\|u\|_{H^{2}}^{1/2}$ (e.g., \cite{KS12}, Property~1.1.4).
Therefore, in\vadjust{\goodbreak} order to complete the proof, it evidently suffices to
show that
\[
\mathbf{P}^u \Biggl[ \sum_{n=1}^\infty2^{-n}
\|X_n\|_{H^2}<\infty \Biggr]=1 \qquad\mbox{for all }u\in H.
\]
But as we assume that only finitely many Fourier modes are forced,
we have
$\mathbf{E}^u[\|X_n\|_{H^2}^2]\le C(1+\mathbf{E}^u[\|X_{n-1}\|_H^m])$
for some constants $m\ge1$ and $C>0$ independent of $n$
by a standard Sobolev estimate \cite{KS12}, Proposition~2.4.12.
As $\sup_n \mathbf{E}^u[\|X_{n}\|_H^m]<\infty$
by Theorem~\ref{thmm:hm}, the result follows readily.
\end{pf}

We can now verify the assumptions of Theorem~\ref{thmm:wfstab}.
Note that for any $u\in H^2$, we have $\|u\|_\infty\lesssim\|u\|
_{H^2}$ by
the Sobolev embedding theorem. In particular,
$\|\mathcal{K}v\|_\infty\lesssim\|v\|_{H^1}$ for any $v\in H^1$.
We can therefore easily compute
\[
\int \bigl\{{ \sqrt{g(x,y)}-\sqrt{g \bigl(x',y \bigr)}} \bigr
\}^2 \varphi(dy) \le C\bigl\|x-x'\bigr\|_{H^1}^2
\qquad\mbox{for all }x,x'\in H %
\]
as in Lemma~\ref{lem:h1obs}. In view of Corollary~\ref{cor:snscoup},
we have verified the assumptions of Theorem~\ref{thmm:wfstab}
for $\tilde d(x,y)=\|x-y\|_{H^1}$. We therefore conclude that
\[
\bigl\|\pi_n^\mu-\pi_n^\nu
\bigr\|_{\mathrm{BL}} \mathop{\longrightarrow}^{n\to\infty}0 \qquad\mbox{in }\mathbf{P}^\gamma
\mbox{-probability} \mbox{ for all } \mu,\nu,\gamma\in\mathcal{P}(E), %
\]
that is, we have established filter stability in the bounded-Lipschitz
norm.

\subsection{Stochastic spin systems}
\label{sec:exspin}

We now turn to an example of an essentially different nature: we
consider a stochastic spin system with counting observations (this could
serve a stylized model, e.g., of photocount data from optical
observations of a chain of ions in a linear trap). In this setting, the
unobserved process $(x_t)_{t\ge0}$ describes the configuration of spins
in one dimension; that is, $x_t$ takes values in the space
$E=\{0,1\}^\mathbb{Z}$, where $x_t^i\in\{0,1\}$ denotes the state of
spin $i\in\mathbb{Z}$ at time $t\ge0$. The observations $(y_t)_{t\ge
0}$ are modeled by a counting process, so that $y_t$ takes values in
$F=\mathbb{Z}_+$.

To define the dynamics of $(x_t)_{t\ge0}$, we introduce a function
$c_i\dvtx E\to\mbox{}]0,\infty[\mbox{}$ for every spin $i\in\mathbb{Z}$. We
interpret $c_i(\sigma)$ as the rate at which spin $i$ flips when the
system is in the configuration $\sigma$. We will make the following
assumptions.

\begin{aspt}
We assume the flip rates are (a) uniformly bounded:
$\sup_{i,\sigma}c_i(\sigma)<\infty$; (b) finite range:
$c_i(\sigma)$ depends only on $\sigma_j$, $|i-j|\le R<\infty$;
(c) translation invariant:
$c_i(\sigma) = c_{i+1}(\sigma')$ if $\sigma_j=\sigma_{j+1}'$ for all $i,j$.
\end{aspt}

The interpretation of $c_i(\sigma)$ is made precise by defining
the pregenerator
\[
\mathscr{L}f(\sigma) = \sum_{i\in\mathbb{Z}} c_i(
\sigma) \bigl\{f \bigl(\sigma^i \bigr)-f(\sigma) \bigr\} \qquad\mbox{for }
\sigma\in E, f\in\mathscr{C}, %
\]
where $\sigma^i_j=\sigma_j$ for $j\ne i$ and $\sigma^i_i=1-\sigma_i$ and
$\mathscr{C}$ is the space of cylinder functions on $E$. Then the
closure of $\mathscr{L}$ in $C(E)$ is the generator of a Markov
semigroup~\cite{Lig05}, Chapter III, and we let $(x_t)_{t\ge0}$ be the
associated Markov process.
To ensure good ergodic properties of $(x_t)_{t\ge0}$, we will assume
the following.

\begin{aspt}
The spin system $(x_t)_{t\ge0}$ is reversible with respect to
some probability $\lambda$. Moreover, the flip rates are attractive:
if $\sigma\le\sigma'$, then we have
$c_i(\sigma)\le c_i(\sigma')$ if $\sigma_i=\sigma_i'=0$ and
$c_i(\sigma)\ge c_i(\sigma')$ if $\sigma_i=\sigma_i'=1$.
\end{aspt}

It is known that under our assumptions, $\lambda$ is necessarily a Gibbs
measure \cite{Lig05}, Theorem IV.2.13 (so this is a stochastic Ising model).
The attractive dynamics will tend to make neighboring spins agree; in
this setting, $(x_t)_{t\ge0}$ admits $\lambda$ as its unique
invariant measure \cite{Lig05}, Theorem IV.3.13.

To define the observations, we will fix a strictly positive continuous
function $h\dvtx E\to\mbox{}]0,\infty[\mbox{}$. The conditional law of
$(y_t)_{t\ge0}$ given $(x_t)_{t\ge0}$ is modeled as an inhomogeneous
Poisson process with rate $\lambda_t=h(x_t)$.

\begin{lem}
The pair $(x_t,y_t)_{t\ge0}$ defines
a nondegenerate Markov additive process in
$\{0,1\}^\mathbb{Z}\times\mathbb{Z}_+$ that admits a unique
invariant probability $\lambda$.
\end{lem}

\begin{pf}
That $(x_t,y_t)_{t\ge0}$ defines a Markov additive process is evident,
and the existence of a unique invariant probability under the assumptions
of this section was stated above. To establish nondegeneracy, it suffices
to note that as $h$ is strictly positive,
the conditional law of $y_{0,\delta}$ given $x_{0,\delta}$
is equivalent to the law $\varphi_\delta$ of a unit-rate Poisson process
by \cite{LS01}, Theorem~19.4.
\end{pf}

We will require below the stronger assumption that the observation
function $h$ is Lipschitz continuous with respect to a suitable metric.
Note that for any choice of scalars $\alpha_i>0$ (for
$i\in\mathbb{Z}$) such that
$\sum_i\alpha_i<\infty$, the quantity
\[
d \bigl(\sigma,\sigma' \bigr) = \sum_{i\in\mathbb{Z}}
\alpha_i\mathbf{1}_{\sigma_i\ne\sigma_i'},\qquad \sigma,\sigma'\in E
= \{0,1\}^\mathbb{Z}, %
\]
metrizes the product topology of $\{0,1\}^\mathbb{Z}$. We will assume
throughout this section that $h$ is Lipschitz with respect to
$d$ for a suitable choice of $(\alpha_i)_{i\in\mathbb{Z}}$.

We now aim to establish stability of the filter. As we can naturally
write $E=\prod_{i\in I}E^i$ with $I=\mathbb{Z}$ and $E^i=\{0,1\}$, we
are in the setting of Theorem~\ref{thmm:clmfs}. To apply it, we must
establish the local mixing property. To this end, we will use two
essential tools: a uniform ergodicity result due to Holley and Stroock
\cite{HS89}, and the well-known Wasserstein coupling
\cite{Lig05}, Section~III.1.

\begin{prop}
$(x_t,y_t)_{t\ge0}$ is locally mixing.
\end{prop}

\begin{pf}
Fix a point $x\in E$ and a finite subset $J\subseteq\mathbb{Z}$,
$|J|<\infty$ throughout the proof, and let
$\mathcal{F}_{s,t}^J=\sigma\{x_{s,t}^J,y_{s,t}\}$. It evidently
suffices to show that
\[
\bigl\|\mathbf{P}^x-\mathbf{P}^o\bigr\|_{\mathcal{F}_{t,\infty}^J}
\mathop{\longrightarrow}^{t\to\infty}0, %
\]
where $o\in E$ is the zero configuration. Let $\mathbf{Q}\in
\mathcal{C}(\mathbf{P}^o,\mathbf{P}^x)$ be the Wasserstein coupling
\cite{Lig05}, Section III.1. As obviously $o\le x$, we have
\cite{Lig05}, Theorem III.1.5
\[
x_t \le x_t'\qquad \mbox{for all }t\ge0,
\mathbf{Q}\mbox{-a.s.} %
\]
To proceed, we recall a result of Holley and Stroock \cite{HS89}, Theorem~0.4:
\[
\sup_{\sigma,\sigma'\in E}\bigl |\mathbf{P}^\sigma \bigl[x_t^0=1
\bigr]-\mathbf{P}^{\sigma'} \bigl[x_t^0=1 \bigr]\bigr|
\le Ce^{-\gamma t}\qquad \mbox{for all }t\ge0 %
\]
for some constants $C,\gamma>0$.
By translation invariance, it follows that
\[
\sup_{i\in\mathbb{Z}} \sup_{\sigma,\sigma'\in E} \bigl|
\mathbf{P}^\sigma \bigl[x_t^i=1 \bigr]-
\mathbf{P}^{\sigma'} \bigl[x_t^i=1 \bigr]\bigr| \le
Ce^{-\gamma t} \qquad\mbox{for all }t\ge0. %
\]
Therefore, by monotonicity,
\[
\mathbf{Q} \bigl[x_t^i\ne x_t^{i\prime}
\bigr] = \mathbf{E}_{\mathbf{Q}} \bigl[x_t^{i\prime}-x_t^i
\bigr] \le C e^{-\gamma t} \qquad\mbox{for all }t\ge0, i\in\mathbb{Z}. %
\]
Using the Lipschitz property of $h$, it follows easily that
\[
\mathbf{E}_{\mathbf{Q}} \biggl[ \int_0^\infty\bigl|h(x_t)-h
\bigl(x_t' \bigr)\bigr| \,dt \biggr]<\infty \quad\mbox{and}\quad
\mathbf{E}_{\mathbf{Q}} \biggl[\int_0^\infty
\mathbf{1}_{x_t^J\ne x_t^{J\prime}} \,dt \biggr]<\infty. %
\]
We claim that the second inequality implies that $x_t^J=x_t^{J\prime}$
for all $t$ sufficiently large $\mathbf{Q}$-a.s.; we postpone the
verification of this claim until the end of the proof. Assuming the claim,
we now complete the proof of local mixing.

Let us extend the Wasserstein coupling $\mathbf{Q}$ to the triple
$(x_t,x_t',y_t)_{t\ge0}$ by letting $(y_t)_{t\ge0}$
be an inhomogeneous Poisson process with rate $\lambda_t=h(x_t)$
conditionally on $(x_t,x_t')_{t\ge0}$. Define for any
$t\ge0$ the random variable
\[
\Lambda_t = \exp \biggl[ \int_t^\infty
\bigl\{\log h \bigl(x_{s-}' \bigr)-\log
h(x_{s-}) \bigr\} \,dy_s -\int_t^\infty
\bigl\{h \bigl(x_s' \bigr)-h(x_s) \bigr\}
\,ds \biggr]. %
\]
Applying \cite{LS01}, Lemma~19.6, conditionally yields
$\mathbf{E}[\Lambda_t]=1$ for all $t\ge0$. By the
change of measure theorem for Poisson processes
\cite{LS01}, Theorem~19.4,
\[
\mathbf{P}^o(A) = \mathbf{E}_{\mathbf{Q}} \bigl[
\mathbf{1}_A \bigl(x_{t,\infty}^{J},y_{t,\infty}
\bigr) \bigr], \qquad\mathbf{P}^x(A) = \mathbf{E}_{\mathbf{Q}} \bigl[
\mathbf{1}_A \bigl(x_{t,\infty}^{J\prime},
y_{t,\infty} \bigr) \Lambda_t \bigr] %
\]
for any $A\in\mathcal{F}_{t,\infty}^J$. Thus, we can estimate
\[
\bigl\|\mathbf{P}^x-\mathbf{P}^o\bigr\|_{\mathcal{F}_{t,\infty}^J} \le
\mathbf{Q} \bigl[x_{t,\infty}^{J}\ne x_{t,\infty}^{J\prime}
\bigr] + \mathbf{E}_{\mathbf{Q}}\bigl[|1-\Lambda_t|\bigr]\mathop{\longrightarrow}^{t\to
\infty}0, %
\]
where the convergence follows by
Scheff\'e's lemma and the above claim.

It remains to prove the claim. To this end, define the stopping times
$\tau_0=0$ and
$\tau_n' = \inf\{t\ge\tau_n\dvtx x_t^J=x_t^{J\prime}\}$ and
$\tau_{n+1} = \inf\{t\ge\tau_n'\dvtx x_t^J\ne x_t^{J\prime}\}$ for
$n\ge0$. By right-continuity $\tau_n'>\tau_n$
on $\{\tau_n<\infty\}$ for all $n\ge1$, and
\[
\mathbf{E}_{\mathbf{Q}} \Biggl[ \sum_{n=0}^\infty
\bigl(\tau_n'-\tau_n \bigr)
\mathbf{1}_{\tau_n<\infty} \Biggr] = \mathbf{E}_{\mathbf{Q}} \biggl[ \int
_0^\infty\mathbf{1}_{x_t^J\ne x_t^{J\prime}} \,dt \biggr] <
\infty. %
\]
Now denote by $\tau_n''=\inf\{t\ge\tau_n\dvtx (x_t^J,x_t^{J\prime})\ne
(x_{\tau_n}^J,x_{\tau_n}^{J\prime})\}$. As the Wasserstein coupling
is itself a particle system with uniformly bounded rates,
it is a routine exercise to verify that there exists a constant $c>0$
such that
\[
\mathbf{E}_{\mathbf{Q}} \bigl[\tau_n'-
\tau_n|\mathcal{F}_{\tau_n} \bigr] \mathbf{1}_{\tau_n<\infty}
\ge \mathbf{E}_{\mathbf{Q}} \bigl[\tau_n''-
\tau_n|\mathcal{F}_{\tau_n} \bigr] \mathbf{1}_{\tau_n<\infty}
\ge c \mathbf{1}_{\tau_n<\infty},\qquad \mathbf{Q}\mbox{-a.s.} %
\]
It follows that $\tau_n=\infty$ eventually
$\mathbf{Q}$-a.s., which yields the claim.
\end{pf}

We have now verified all the assumptions of Theorem~\ref{thmm:clmfs}.
Thus, we have
\[
\bigl\|\pi_t^\mu-\pi_t^\nu
\bigr\|_{\mathcal{E}^J}\mathop{\longrightarrow}^{t\to\infty}0 \qquad\mbox{in }\mathbf{P}^\gamma
\mbox{-probability} \mbox{ for all }|J|<\infty, \mu,\nu,\gamma\in\mathcal{P}(E),
\]
where $\mathcal{E}^J\subseteq\mathcal{B}(E)$ be the cylinder $\sigma$-field
generated by the spins $J$.

\begin{rem}
The proof just given works only in one spatial dimension $I=\mathbb{Z}$.
In a higher-dimensional lattice $I=\mathbb{Z}^d$, the (unconditional)
ergodic theory of the spin system becomes much more subtle as phase
transitions typically appear. A Dobrushin-type sufficient condition
for local mixing in any dimension is given by F\"ollmer
\cite{Fol79} for a related discrete-time model.
With some more work, this approach can also be applied
to continuous time spin systems in the high-temperature regime
(e.g., by showing that the requisite bounds
hold for spatially truncated and time-discretized models, uniformly in
the truncation and discretization parameters).
\end{rem}

\subsection{Stochastic differential delay equations}
\label{sec:exdelay}

Our final example is concerned with filtering in stochastic differential
delay equations. Time delays arise naturally in various engineering and
biological applications, and the corresponding filtering problem has been
investigated by a number of authors \cite{Vin77,KW78,CFN07}. In particular,
some results on filter stability for linear delay equations have
been investigated in \cite{Vin77,KW78} by means of the associated
Kalman equations. We tackle here the much more difficult nonlinear case.

Fix throughout this section a delay $r\in\mathbb{R}_+$.
Following \cite{CFN07}, for example, we introduce the following
nonlinear filtering model with time delay.
The unobserved process is defined by the stochastic differential delay
equation
\[
dx(t) = f(x_t) \,dt + g(x_t) \,dW_t,
\]
where $(W_t)_{t\ge0}$ is $m$-dimensional Brownian motion,
$x(t)$ takes values in $\mathbb{R}^n$, and we have introduced the
notation $x_t = (x(t+s))_{s\in[-r,0]}\in C([-r,0];\mathbb{R}^n)$.
The $\mathbb{R}^d$-valued observations are defined by the white noise
model
\[
dy_t = h(x_t) \,dt + dB_t, %
\]
where $(B_t)_{t\ge0}$ is $d$-dimensional Brownian motion
independent of $(x_t)_{t\ge0}$.

In the following, we will exploit heavily the ergodicity results for
stochastic delay equations established in \cite{HMS11}. To this end, we
work under the following assumptions. Here and in the sequel, we denote
by $\|x\|=\sup_{t\in[-r,0]}|x(t)|$ for $x\in C([-r,0];\mathbb{R}^n)$,
and by $|M|^2=\mathrm{Tr}[MM^*]$ for any matrix $M$.

\begin{aspt}
Assume (a) there exists $g^{-1}\dvtx C([-r,0];\mathbb{R}^n)\to
\mathbb{R}^{m\times n}$ with $\|g^{-1}\|_\infty<\infty$ and
$g(x)g^{-1}(x)=\mathrm{Id}_n$ for all
$x$; (b) $f$ is continuous and
bounded on bounded subsets of $C([-r,0];\mathbb{R}^n)$; (c)
for all $x,y$, we have
$2\langle f(x)-f(y),x(0)-y(0)\rangle^+
+|g(x)-g(y)|^2+|h(x)-h(y)|^2\le L\|x-y\|^2$.
\end{aspt}

Under this assumption, the equation for $(x(t))_{t\ge0}$ possesses a
unique strong solution for any initial condition $(x(t))_{t\in[-r,0]}$
such that $(x_t)_{t\ge0}$ is a $C([-r,0];\mathbb{R}^n)$-valued
strong Markov process \cite{HMS11}. Thus, the pair
$(x_t,y_t)$ is evidently a nondegenerate Markov additive process
in $C([-r,0];\mathbb{R}^n)\times\mathbb{R}^d$.

The previous assumption does not ensure the existence of
an invariant probability. Rather than imposing explicit sufficient
conditions (see, e.g., \cite{ESV10,DZ96}), it will suffice
simply to assume that such a probability exists.

\begin{aspt}
$(x_t)_{t\ge0}$ admits an invariant
probability $\lambda$.
\end{aspt}

To establish stability of the filter, we will apply Theorem~\ref{thmm:contwfstab}. To construct an asymptotic coupling, the key
result that we will use is the following.

\begin{thmm}
\label{thmm:delay}
For every $x,x'\in C([-r,0];\mathbb{R}^n)$, there exists a
coupling $\mathbf{Q}^{x,x'}\in\mathcal{C}(\mathbf{P}^x,\mathbf{P}^{x'})$
such that the map $(x,x')\mapsto\mathbf{Q}^{x,x'}$ is measurable and
\[
\inf_{\|x\|,\|x'\|\le R} \mathbf{Q}^{x,x'} \Bigl[ \sup
_{t\ge0}e^t\bigl\|x_t-x_t'
\bigr\|<\infty \Bigr]=\beta_R>0 \qquad\mbox{for every }R<\infty. %
\]
\end{thmm}

We postpone the proof of this result to the end of this section.
Let us now show how the result can be used to verify the assumptions
of Theorem~\ref{thmm:contwfstab}.

We first construct the asymptotic coupling. Let us choose $R>0$ such
that $\lambda[\|x\|<R]>1/2$. By \cite{HMS11}, Theorem~3.7 and
the Portmanteau theorem, we have $\mathbf{P}^x[\|x_t\|<R]\ge1/2$
eventually as $t\to\infty$ for every $x\in C([-r,0];\mathbb{R}^n)$.
Let $\alpha=\beta_{R}/4$. Given any $x,x'\in C([-r,0];\mathbb{R}^n)$,
we now construct a coupling $\mathbf{Q}\in\mathcal{C}(\mathbf{P}^x,
\mathbf{P}^{x'})$ as follows. First, choose $s>0$ such that
$\mathbf{P}^x[\|x_s\|<R]\ge1/2$ and $\mathbf{P}^{x'}[\|x_s\|<R]\ge1/2$.
We then define the coupling $\mathbf{Q}$ such that
\[
\mathbf{Q} \bigl[x_{0,s},x_{0,s}'\in\cdot
\bigr]= \mathbf{P}^x|_{\mathcal{F}_{0,s}}\otimes
\mathbf{P}^{x'}|_{\mathcal{F}_{0,s}},\qquad
\mathbf{Q} \bigl[x_{s,\infty},x_{s,\infty}'\in\cdot |
\mathcal{F}_{0,s} \bigr] = \mathbf{Q}^{x_{s},x_{s}'}. %
\]
By construction, we have
\[
\mathbf{Q} \Bigl[ \sup_{t\ge0}e^t
\bigl\|x_t-x_t'\bigr\|<\infty \Bigr]\ge\alpha.
\]
Thus, we have evidently verified the first assumption of
Theorem~\ref{thmm:contwfstab} for $\tilde d(x,x')=\|x-x'\|$ and
$\Delta=1$ (e.g.). On the other hand, the second assumption follows
easily as in Lemma~\ref{lem:h1obs}, as we have assumed
the Lipschitz property of $h$. Thus, we have verified the assumptions of
Theorem~\ref{thmm:contwfstab}, so
\[
\bigl\|\pi_t^\mu-\pi_t^\nu
\bigr\|_{\mathrm{BL}} \mathop{\longrightarrow}^{t\to\infty}0 \qquad\mbox{in }\mathbf{P}^\gamma
\mbox{-probability} \mbox{ for all } \mu,\nu,\gamma\in\mathcal{P} \bigl(C
\bigl([-r,0];\mathbb{R}^n \bigr) \bigr), %
\]
that is, we have established filter stability in the bounded-Lipschitz
norm.

It remains to prove Theorem~\ref{thmm:delay}. This is a direct extension
of the proof of Theorem~3.1 in \cite{HMS11}; we finish the section
by sketching the necessary steps.

\begin{pf*}{Proof of Theorem~\ref{thmm:delay}}
In the proof of \cite{HMS11}, Theorem~3.1, a kernel
$(x,x')\mapsto\tilde{\mathbf Q}^{x,x'}$ was constructed on
$\Omega\times\Omega$ with the following properties. First, there
exists a constant $\gamma>0$ independent of $x,x'$ such that
\[
\tilde{\mathbf Q}^{x,x'} \Bigl[ \sup_{t\ge0}e^t
\bigl\|x_t-x_t'\bigr\|<\infty \Bigr]\ge\gamma\qquad
\mbox{for all }x,x'\in C \bigl([-r,0];\mathbb{R}^n
\bigr). %
\]
Second, there is a $\tilde{\mathbf Q}^{x,x'}$-Brownian motion
$(\tilde W_t)_{t\ge0}$ and an adapted
process $(z_t)_{t\ge0}$ that satisfies
$\int_0^\infty|z_t|^2 \,dt\le C\|x-x'\|^2$, $\tilde{\mathbf Q}^{x,x'}$-a.s. such that
\[
\tilde{\mathbf Q}^{x,x'}[x_{0,\infty}\in A]=\mathbf{P}^x(A),\qquad
\tilde{\mathbf Q}^{x,x'} \bigl[\mathbf{1}_A
\bigl(x_{0,\infty}' \bigr)\Lambda \bigr]= \mathbf{P}^{x'}(A)
\]
for every measurable set $A$, where $\Lambda$ is the Girsanov density
\[
\Lambda= \exp \biggl[ \int_0^\infty
z_t \,d\tilde W_t -\frac{1}{2}\int
_0^\infty |z_t|^2 \,dt
\biggr]. %
\]
Let $\mathbf{R}^{x,x'}\in
\mathcal{C}(\tilde{\mathbf Q}^{x,x'},\mathbf{P}^{x'})$ be the coupling
maximizing $\mathbf{R}^{x,x'}[x_{0,\infty}'=x_{0,\infty}'']$.
It is classical that $2 \mathbf{R}^{x,x'}[x_{0,\infty}'\ne x_{0,\infty}'']=
\|\tilde{\mathbf Q}^{x,x'}[x_{0,\infty}'\in\cdot ]-\mathbf{P}^{x'}\|$ and
that the maximal coupling can be chosen to be measurable in $x,x'$ (by the
existence of a measurable version of the Radon--Nikodym density
between kernels \cite{DM82}, Theorem~V.58).\vspace*{1.5pt} As $\int_0^\infty
|z_t|^2 \,dt\le C\|x-x'\|^2$, $\tilde{\mathbf Q}^{x,x'}$-a.s., we can chose
$\delta>0$ sufficiently small that
$\mathbf{R}^{x,x'}[x_{0,\infty}'\ne x_{0,\infty}'']\le\gamma/2$
whenever $\|x-x'\|\le\delta$. Then evidently
\[
\mathbf{R}^{x,x'} \Bigl[ \sup_{t\ge0}e^t
\bigl\|x_t-x_t''\bigr\|<\infty \Bigr]
\ge\frac{\gamma}{2}\qquad \mbox{whenever }\bigl\|x-x'\bigr\|\le\delta.
\]
Now define for any $x,x'\in C([-r,0];\mathbb{R}^n)$ the measure
$\mathbf{Q}^{x,x'}$ such that
\begin{eqnarray*}
\mathbf{Q}^{x,x'} \bigl[x_{0,2r},x_{0,2r}'
\in\cdot \bigr]&=& \mathbf{P}^x|_{\mathcal{F}_{0,2r}}\otimes
\mathbf{P}^{x'}|_{\mathcal{F}_{0,2r}},
\\
\mathbf{Q}^{x,x'} \bigl[x_{2r,\infty},x_{2r,\infty}'
\in\cdot | \mathcal{F}_{0,2r} \bigr] &=& \mathbf{R}^{x_{2r},x_{2r}'}
\bigl[x_{0,\infty},x_{0,\infty}''\in\cdot
\bigr].
\end{eqnarray*}
Then $\mathbf{Q}^{x,x'}\in\mathcal{C}(\mathbf{P}^x,\mathbf{P}^{x'})$,
$(x,x')\mapsto\mathbf{Q}^{x,x'}$ is measurable, and
\[
\inf_{\|x\|,\|x'\|\le R} \mathbf{Q}^{x,x'} \Bigl[ \sup
_{t\ge0}e^t\bigl\|x_t-x_t'
\bigr\|<\infty \Bigr] \ge \frac{\gamma}{2} \biggl( \inf_{\|x\|\le R}
\mathbf{P}^{x} \biggl[ \|x_{2r}\|\le\frac{\delta}{2}
\biggr] \biggr)^2. %
\]
It remains to note that the right-hand side is
positive by \cite{HMS11}, Lemma~3.8.
\end{pf*}

\begin{rem}[(On infinite-dimensional observations)]
\label{rem:idimobs}
All the examples that we have discussed in this section are concerned with
an infinite-dimensional unobserved process and a finite-dimensional observed
process. In this setting, it is natural to work with observation
densities, and the nondegeneracy assumptions of our main results are
easily verified. It is less evident in what situations the results in
this paper can be expected to be applicable when both unobserved and
observed processes are infinite-dimensional. We conclude Section~\ref{sec:examples} by briefly discussing this issue.

In the case of unobserved models that possess a significant degree of
spatial regularity, such as those in Sections \ref{sec:exheat} and
\ref{sec:exsns}, there are natural infinite-dimensional observation models
that are amenable to the theory developed in this paper. For example, in
the setting of Section~\ref{sec:exsns}, consider that we observe the
entire fluid velocity field corrupted by spatial white noise (rather than
at a finite number of spatial locations): that is, each Fourier mode of
the field is observed in an independent Gaussian noise $\xi_n^k\sim
N(0,I)$,
\[
\langle e_k,Y_n\rangle= \bigl\langle e_k,
\mathcal{K}v(t_n, \cdot ) \bigr\rangle+ \xi_n^k
\]
for $k\in\mathbb{Z}^2\setminus\{(0,0)\}$. As the fluid velocity field is
square-integrable, its Fourier coefficients are square-summable, and thus
the conditional law of the observation $Y_n$ given the unobserved process
$X_n$ has a positive density with respect to the law of the noise
$(\xi_n^k)$ by the Kakutani theorem. The nondegeneracy and continuity
assumptions in our weak-* stability results are therefore easily verified.
One could argue, however, that this observation model is still
``effectively'' finite-dimensional: due to the roughness of the noise,
only a finite number of (large) Fourier modes affect substantially the law
of the observations, while the remaining (small) modes are buried
in the noise.

In the above example, the same argument applies when the observations are
corrupted by spatially regular noise, provided that the fluid velocity
field is sufficiently smooth as compared to the noise.
However, if the noise is too smooth as compared to the unobserved model,
then nondegeneracy will fail for precisely the same reason that the
unobserved model may fail to be Harris. This example illustrates that
nondegeneracy in the presence of infinite-dimensional observations can be
a delicate issue.

More broadly, we recall that at the heart of the difficulties encountered
in infinite dimension is that most measures are mutually singular (cf. Example~\ref{ex:trivial}). The theory developed in this paper surmounts
this problem by considering \emph{local} notions of ergodicity.
Nonetheless, our main Assumptions \ref{aspt:locerg} and \ref{aspt:nondeg}
still rely on some degree of nonsingularity: Assumption~\ref{aspt:locerg}
allows us to localize the unobserved component, but still the entire
observation variable $Y_k$ must be included in the local filtration; and
Assumption~\ref{aspt:nondeg} requires the coupling between the unobserved
and observed components to be nonsingular. In practice, this implies that
while the unobserved process $X_k$ may be infinite-dimensional, the
observed process $Y_k$ must typically be finite-dimensional or at least
``effectively'' finite-dimensional in order to apply the general theory
developed in this paper. As was illustrated in this section, our general
theory covers a wide range of models of practical interest; however,
models in which the observations are degenerate are excluded (e.g.,
this would be the case if in the setting of Section~\ref{sec:exspin} each
spin $x_t^i$ were observed in independent noise). In the latter setting,
new probabilistic phenomena arise, such as conditional phase
transitions, that are of significant interest in their own right; such
issues will be discussed elsewhere.
\end{rem}

\section*{Acknowledgments}
We would like to thank M. Hairer and J. Mattingly for helpful
discussions on the papers \cite{HM06,HMS11,EMS01,Mat02}. We are grateful
to the referees for comments that helped us improve the presentation.

%



\printaddresses

\end{document}